%% file: KM1.tex
\documentclass[11pt]{amsbook}
\pdfoutput=1

\usepackage{graphicx}
\usepackage[colorlinks=true, citecolor=black, filecolor=black, linkcolor=black, urlcolor=black]{hyperref}
\usepackage{simplewick}

\def\SS{\textbullet~}
\newcommand\subsec[1]{\SS\emph{#1.}~}

\def\Gchiral{G^+}
\def\Phichiral{\Phi^+}

\def\arc{\stackrel{\textstyle\frown}}
\def\bp{{\bar\partial}}
\def\bs{\bigskip}
\def\bfs{\boldsymbol}

\def\const{\textrm{const}}
\def\dist{\textrm{dist}}

\def\ms{\medskip}

\def\pa{\partial}

\def\sm{\setminus}
\def\ss{\smallskip}
\def\ti{\tilde}

\def\ve{\varepsilon}

\def\End{\mathrm{End}}
\def\Aut{   \mathrm{Aut}}

\def\cov{   \mathrm{cov}}
\def\GFF{   \mathrm{GFF}}
\def\id{    \mathrm{id}}
\def\Lin{   \mathrm{span}}

\DeclareMathOperator{\Sing}{Sing}
\def\SLE{  \mathrm{SLE}}
\def\supp{  \mathrm{supp}}
\def\trace{ \mathrm{Tr}}
\def\var{   \mathrm{var}}

\def\Pa{    \bfs{\pa}}
\def\Re{    \mathrm{Re}}
\def\Im{    \mathrm{Im}}

\def\AA{\mathcal{A}}

\def\EE{\mathcal{E}}
\def\FF{\mathcal{F}}
\def\HH{\mathcal{H}}
\def\LL{\mathcal{L}}
\def\NN{\mathcal{N}}
\def\OO{\mathcal{O}}
\def\RR{\mathcal{R}}
\def\VV{\mathcal{V}}

\def\XX{\mathcal{X}}
\def\YY{\mathcal{Y}}
\def\ZZ{\mathcal{Z}}

\def\bfA{\mathbf{A}}
\def\bfa{\mathbf{a}}
\def\bfB{\mathbf{B}}
\def\bfb{\mathbf{b}}
\def\bfC{\mathbf{C}}
\def\bfc{\mathbf{c}}

\def\bfJ{\mathbf{J}}
\def\bfj{\mathbf{j}}
\def\bfl{\mathbf{l}}
\def\bfp{\mathbf{p}}
\def\bfT{\mathbf{T}}
\def\bft{\mathbf{t}}

\def\C{\mathbb{C}}
\def\D{\mathbb{D}}
\def\E{\mathbf{E}}
\def\H{\mathbb{H}}
\def\P{\mathbf{P}}
\def\R{\mathbb{R}}
\def\Z{\mathbb{Z}}

\theoremstyle{plain}
\newtheorem*{fact*}{Fact}
\newtheorem*{thm*}{Theorem}
\newtheorem*{lem*}{Lemma}
\newtheorem*{cor*}{Corollary}
\newtheorem*{prop*}{Proposition}
\newtheorem*{claim*}{Claim}
\newtheorem{thm}{Theorem}[chapter]
\newtheorem{lem}[thm]{Lemma}
\newtheorem{cor}[thm]{Corollary}
\newtheorem{prop}[thm]{Proposition}

\theoremstyle{definition}

\newtheorem*{eg*}{Example}
\newtheorem*{egs*}{Examples}
\newtheorem*{def*}{Definition}

\theoremstyle{remark}
\newtheorem*{rmk*}{Remark}
\newtheorem*{rmks*}{Remarks}

\numberwithin{section}{chapter}
\numberwithin{equation}{chapter}
\numberwithin{figure}{chapter}

\makeatletter
\def\@MRnumber{}
\def\@scanforMR#1 #2\endscan{%
    \def\@MRnumber{#1}%
    }
\def\reviewed#1{\ifx#1\stop \let\next=\relax \else \ifx#1(\advance\count255 by1 \else\let\next=\relax \fi \let\next=\reviewed \fi \next}
\def\MR#1{\relax
  \ifhmode\unskip\spacefactor3000 \space\fi
  \count255=0 \reviewed#1\stop
  \ifnum \count255>0
  {\@scanforMR#1\endscan\href{http://www.ams.org/mathscinet-getitem?mr=\@MRnumber}{MR#1}}
  \else
  {\href{http://www.ams.org/mathscinet-getitem?mr=#1}{MR#1}}
  \fi}
\makeatother

\def\@citestyle{\m@th\upshape\mdseries}
\def\citeform#1{{\normalfont#1}}
\def\@cite#1#2{{%
\@citestyle[\citeform{#1}\if@tempswa, #2\fi]}}

\newlength\savedwidth
\def\toprule{\noalign{\global\savedwidth\arrayrulewidth
\global\arrayrulewidth 1pt}%
\hline
\noalign{\global\arrayrulewidth\savedwidth}}
\def\midrule{\noalign{\global\savedwidth\arrayrulewidth
\global\arrayrulewidth .6pt}%
\hline
\noalign{\global\arrayrulewidth\savedwidth}}
\def\bottomrule{\noalign{\global\savedwidth\arrayrulewidth
\global\arrayrulewidth 1pt}%
\hline
\noalign{\global\arrayrulewidth\savedwidth}}

\newcommand\arXiv[1]{\href{http://arxiv.org/abs/#1}{arXiv:#1}}

\makeindex

\begin{document}

\frontmatter
\title{Gaussian free field and conformal field theory}

\author{Nam-Gyu Kang}
\address{Department of Mathematical Sciences, Seoul National University,\newline Seoul, 151-747, Republic of Korea}
\email{nkang@snu.ac.kr}
\thanks{The first author was partially supported by NRF grant 2010-0021628 and by the Research Settlement Fund for new faculty of Seoul National University. \newline
The second author was supported by NSF grant no. 1101735.}
\author{Nikolai~G.~Makarov}
\address{Department of Mathematics, California Institute of Technology,\newline Pasadena, CA 91125, USA}
\email{makarov@caltech.edu}


\begin{abstract}
In these mostly expository lectures, we give an elementary introduction to conformal field theory in the context of probability theory and complex analysis.
We consider statistical fields, and define Ward functionals in terms of their Lie derivatives.
Based on this approach, we explain some equations of conformal field theory and outline their relation to SLE theory.
\end{abstract}

\maketitle
\setcounter{page}{2}

\tableofcontents

\mainmatter

\chapter*{Introduction} \label{ch: intro}
Conformal field theory (CFT) has different formulations as well as multiple applications. One of the best known applications concerns the theory of 2D lattice models at their critical points.
Borrowing ideas and intuition from quantum field theory, Belavin, Polyakov, and Zamolodchikov \cite{BPZ84} introduced an operator algebra formalism which relates some critical models to the representation theory of Virasoro algebra.

\ms The underlying objects of BPZ theory are correlation functions of certain ``fields," apparently smeared-out and renormalized continuum versions of random fields on a lattice.
The mathematical meaning of these objects is not completely clarified, but the focus is instead on the algebraic structure of ``local operators" which act on and are identified with the fields.
The main assumption of the theory is that the operators (or fields) behave nicely under ``conformal transformations."
The operators related to the so-called stress-energy tensor (defined as the local response of the action in the functional integral) play a special role in generating a Virasoro algebra representation whose central charge $c$ is the fundamental characteristic of a critical model.
Belavin, Polyakov, and Zamolodchikov showed that in the case of degenerate representations, the correlation functions satisfy a special type of linear differential equations.
Finally they defined a class of conformal theories (``minimal models") which describe and ``solve" (in a physically accepted sense) discrete critical models such as Ising, Potts, etc.

\ms The paper~\cite{BPZ84} had a great influence on the developments of conformal field theory.
The operator formalism, which does not depend on a specific (e.g., statistical) nature of the underlying fields, has been applied to a variety of other physical problems, see \cite{DFMS97}.
In mathematics, the study of abstract vertex algebras became an important part of modern representation theory \cite{FLM88}, \cite{Kac98}.

\ms A different approach to critical lattice models was proposed by Schramm \cite{Schramm00} who introduced stochastic Loewner evolution (SLE) as the only possible candidates for the scaling limits of interface curves in several such models.
His idea turned out to be very successful and led to the rigorous proofs of some important conjectures in statistical physics, in particular some very non-trivial predictions of CFT.
The work of Lawler-Schramm-Werner (\cite{LSW01b}, \cite{LSW01c}, \cite{LSW02a}, \cite{LSW02b}, \cite{LSW04}) and Smirnov (\cite{Smirnov01}, \cite{Smirnov10}) exemplifies the remarkable achievements of complex analytic/probabilistic methods.
In connection with their developments in the SLE theory, there has been some interest in interpreting the original CFT arguments in (less abstract) terms of statistical models, and more generally in understanding the precise relation between CFT and SLE, see e.g., \cite{FW03} and, on the physical side, \cite{BB02}, \cite{BB03}, and \cite{Cardy05}.

\ms The goal of these mostly expository lectures is to give an elementary introduction to CFT from the point of view of random or statistical fields.
More precisely, we will describe an (rather pedestrian) implementation of CFT in the specific case of statistical fields generated by certain non-random modification of the Gaussian free field (GFF).
Gaussian free field is the simplest (``trivial") example of Euclidean field theory; its mathematical aspects are well understood, see \cite{Simon74}, \cite{Janson97}.
The modifications of the Gaussian free field that we will consider in these lectures are implicit in the work of Schramm and Sheffield \cite{SS10} and explicit in the physical paper \cite{RBGW07}.
Related ideas are certainly present in the much earlier papers by Cardy \cite{Cardy84}, \cite{Cardy92}.

\ms We will only cover some starting points of the BPZ theory: we will accurately define and explain such basic concepts as Ward's identities, stress tensor, and vertex fields in terms of correlation functions of our random fields, but we will not reach the part of the theory concerning minimal models, and the only degeneracy we study will be of level two.
In Appendix~\ref{appx: Op algebra} we will briefly explain the relation of our constructions to the operator algebra formalism by explicitly describing some form of the ``operator-field correspondence."
In the last two lectures we will discuss connections with the SLE theory.

\ms It should be mentioned that we only consider the simplest conformal type of the theory -- the case of a simply connected domain with a marked point on the boundary, cf. \cite{Cardy84}, \cite{Cardy92}, and we only consider the Gaussian free field with Dirichlet boundary conditions.
This conformal type of CFT is relevant to the theory of chordal SLE.
The more traditional setting -- CFT in the full plane (\cite{BPZ84}) -- is somewhat more involved and will not be discussed here.

\ms Many computations in these lectures are completely standard from the CFT perspective -- we include them for the sake of consistency and to make the exposition self-contained.
We want to emphasize one more time that what we are considering is a very specific model of CFT, and modern physical and algebraic theories go so much further.
At the same time, we believe that this model is interesting in its own right, and its generalizations to more sophisticated conformal geometries may turn out to be quite non-trivial.

\renewcommand\chaptername{Lecture}
\chapter{Fock space fields} \label{ch: F-fields}

We introduce a class of random fields defined in a simply connected domain $D$ in the complex plane.
All our fields, which we call Fock space fields, are constructed from the Gaussian free field and its derivatives by means of Wick's calculus.
Fock space fields may or may not be distributional random fields but their correlation functions are well-defined, and we can think of the fields as functions in $D$ whose values are correlation functionals.

\ms Later, in Lecture \ref{ch: conf geom}, we will revise the definition so that the fields will have certain geometric/conformal properties in the sense that their values will depend on local coordinates (``conformal fields").
The functionals and fields that we consider in this first lecture are conformal fields expressed in the identity chart of $D.$
In Lecture \ref{ch: chiral} we will further extend the concept to include some ``multivalued" (chiral) fields.

\ms In the first two sections we recall some basic facts concerning the Gaussian free field, its Fock space, and Wick's calculus, see \cite{Simon74} and \cite{Janson97}.
In Section~\ref{sec: corr fcnl} and Section~\ref{sec: F-field} we define correlation functionals and Fock space fields (as functional-valued functions).
In Appendix~\ref{appx: F-field} we will comment on the probabilistic meaning of Fock space fields.

\ms\section{Gaussian free field} \label{sec: GFF}

\ms\SS A real-valued random variable $\xi$ is \emph{Gaussian} or \emph{normal} with mean $\mu$ and variance $\sigma^2$ if
$$\E[e^{it\xi}] = e^{i\mu t-\frac12\sigma^2t^2}.$$
A family (finite or infinite) of random variables is \emph{jointly} Gaussian if any finite linear combination is Gaussian.
The joint distribution of such a family is determined by the means and covariances of the random variables.
In particular, if $\xi_1,\cdots,\xi_n$ are centered (i.e., $\mu=0$) jointly Gaussian random variables, then
\begin{equation}\label{eq: Wick's}
\E[\xi_1\cdots \xi_n] = \sum\prod_k\E[\xi_{i_k}\xi_{j_k}],
\end{equation}
where the sum is over all partitions of the set $\{1,\cdots,n\}$ into disjoint pairs $\{i_k,j_k\}.$

\ms A complex-valued random variable is Gaussian if its real and imaginary parts are jointly Gaussian.
Clearly, the formula \eqref{eq: Wick's} holds for complex-valued jointly Gaussian variables as well.

\ms\SS A \emph{Gaussian field} \index{Gaussian!field} indexed by some \emph{real} Hilbert space $\HH_\R$ is an isometry
$$\HH_\R~\to~ L_\R^2(\Omega, \P),\qquad h\mapsto \xi_h$$
such that the image consists of centered Gaussian variables; here $(\Omega,\P)$ is some probability space.
Alternatively a Gaussian Hilbert space may be thought of as a closed subspace of $L^2(\Omega,\P)$ consisting of Gaussian (square integrable) random variables.
Complexifying, we can extend this map to an isometry
$$\HH:=\HH_\R+i\HH_\R~\to~ L^2(\Omega, \P),$$
which we also call a Gaussian field (indexed by $\HH$).

\ms One way to construct a Gaussian field is to choose an orthonormal basis $\{e_\alpha\}$ in $\HH_\R$ and a family $\{\xi_\alpha\}$ of independent standard normal variables on some probability space, and set $e_\alpha\mapsto \xi_\alpha.$
A Gaussian field indexed by $\HH_\R$ is unique up to an isomorphism of $L^2$-spaces.

\ms\SS Let $D$ be a planar domain with the Green's function $G=G_D(\zeta,z).$
For example, in the \emph{upper half-plane}
$$\H:=\{z:\Im\, z>0\},$$
we have
$$G_\H(\zeta,z) = \log\Big|\frac{\zeta-\bar z}{\zeta-z}\Big|.$$
The \emph{Gaussian free field} \index{Gaussian!free field $\Phi$} $\Phi$ in $D$ with \emph{Dirichlet boundary condition} is the Gaussian field indexed by the \emph{Dirichlet energy space} \index{Dirichlet energy space} $\EE(D),$
$$\Phi:\EE(D) \to L^2(\Omega,\P).$$
The Hilbert space $\EE(D)$ can be defined as the completion of \emph{test} functions $f\in C^\infty_0(D)$ with respect to the norm
\begin{equation} \label{eq: norm}
\|f\|^2_\EE=\iint 2G(\zeta,z)\,f(\zeta)\,\overline{f(z)}~dA(\zeta)dA(z),
\end{equation}
where $A$ is the area measure.

\ms\SS By definition, the \emph{$n$-point correlation function} of $\Phi,$
$$(z_1,\dots, z_n)~\mapsto ~\E[\Phi(z_1)\dots \Phi(z_n)],\qquad (z_j\in D,\; \textrm{points }z_j\textrm{ are distinct}),$$
is a unique \emph{continuous} function such that
\begin{equation} \label{eq: n-pt fcn}
\E[\Phi(f_1)\dots \Phi(f_n)]=\int f_1(z_1)\dots f_n(z_n)~\E[\Phi(z_1)\dots \Phi(z_n)]~dA(z_1)\dots dA(z_n)
\end{equation}
for all test functions $f_j$ with disjoint supports (here, in fact, for all test functions).
Note that $\E$ has different meanings in \eqref{eq: n-pt fcn}; $\E$ in the left-hand side is the expectation of random variables and $\E$ in the right-hand side means the correlation function.
(An alternative and more traditional notation for $\E$ in the right-hand side is $\langle \cdot\rangle.$)
It is clear that the 2-point correlation function is $2G(z_1,z_2)$ by polarization and it can be shown that 
$$\E[\Phi(z_1)\dots \Phi(z_n)]=\sum\prod_k~2G(z_{i_k},z_{j_k}),$$
exactly as in \eqref{eq: Wick's}.
In other words, we can think of $\Phi(z)$ as a ``generalized" Gaussian and use the symbolic representation $\Phi(f) = \int \Phi(z)f(z)\,dA(z)$ in the computation of correlations.

\ms \subsec{Derivatives of GFF}
The fields $J=\pa \Phi,$ $\bar J=\bp \Phi,$ \index{current!field $J,\widehat J$} and higher order derivatives are well-defined as Gaussian \emph{distributional} fields, e.g.,
$$J(f)=-\Phi(\pa f),\qquad f\in C^\infty_0(D),$$
so $J$ is a map $C_0^\infty(D)\to L^2(\Omega,\P)$ (or $J: \EE(D)\to L^2(\Omega,\P)$).
We can compute the correlation functions of the derivatives by differentiating the correlation functions of the Gaussian free field.
For example, for $\zeta\ne z$ we have
$$\E[J(\zeta)\Phi(z)]=2\partial_\zeta G(\zeta,z)=\frac{1}{\zeta-\bar z}-\frac{1}{\zeta-z} \quad {\rm in}\;\H;$$
and
$$\E[J(\zeta)J(z)]=2\partial_\zeta\partial_z G(\zeta,z)= -\frac1{(\zeta-z)^2} \quad {\rm in}\;\H.$$
The meaning of these expressions is similar to formula \eqref{eq: n-pt fcn}.

\ms\section{Fock space of Gaussian free field and Wick's multiplication} \label{sec: F-space}

\ms\SS For $n\ge 0,$ let $\HH^{\odot n}$ denote the $n$-th symmetric tensor power of a Hilbert space $\HH$; it is the completion of linear combinations of elements $f_1\odot\dots \odot f_n,$ (the order does not matter: $f\odot g\equiv g\odot f$), with respect to the scalar product
$$( f_1\odot\cdots\odot f_n,g_1\odot\cdots\odot g_n) = \sum_{\sigma\in S_n} \prod_{j=1}^n ( f_j, g_{\sigma(j)}),$$
where $S_n$ is the group of permutations of the set $\{1,\cdots,n\}.$
The (symmetric) \emph{Fock space} \index{Fock space} over $\HH$ is the Hilbert space direct sum
$$ {\rm Fock}(\HH)=\bigoplus _{n=0}^\infty \HH^{\odot n},\qquad (\HH^{\odot0}:=\C).$$
The algebraic direct sum $\sum_{n=0}^\infty \HH^{\odot n},$ the ``symmetric tensor algebra", is a commutative algebra with respect to the natural multiplication $\odot.$

\ms\subsec{Wiener chaos decomposition}
Let $\HH\to L^2(\Omega,\P)$ be a Gaussian field indexed by $\HH.$
If we identify $\HH$ with its image in $L^2(\Omega,\P)$ and denote by $\sigma(\HH)$ the $\sigma$-algebra generated by $\HH,$ then the Fock space over $\HH$ can be identified with $L^2:=L^2(\Omega,\sigma(\HH),\P)$ as follows, see \cite{Janson97}.
Denote $$\HH_n={\rm span}\{\xi_1\cdots \xi_m:~ m\le n\}\ominus {\rm span}\{\xi_1\cdots \xi_m:~ m<n\}\subset L^2,$$
where $\xi_j$'s in the both spans are arbitrary elements of $\HH,$
and consider the map
$$\HH^{\odot n}\to \HH_n,\qquad \xi_1\odot \dots \odot\xi_n ~\mapsto ~\pi_n(\xi_1 \cdots \xi_n),$$
where $\pi_n$ is the orthogonal projection in $L^2$ onto $\HH_n.$
Under this correspondence,
the symmetric tensor algebra multiplication corresponds to the so-called \emph{Wick's multiplication} in $L^2$: \index{Wick's! multiplication}
$$\textrm{if }X\in \HH_m,\; Y\in \HH_n,\quad \textrm{then} \quad X\odot Y=\pi_{m+n}(XY).$$
(An alternative and more traditional notation is $:\!XY\!:.$)
The identification
$$L^2(\Omega, \sigma(\HH), \P) \cong \bigoplus _{n=0}^\infty \HH^{\odot n}$$
is called the Wiener chaos decomposition. \index{Wiener chaos decomposition}
The fact that the described construction gives a unitary map Fock$(\HH)\to L^2$ is based on the following Wick's formula, which provides the chaos decomposition for products of Gaussian variables, and which will play a central role in the definition of Fock space fields.
The formula is stated in terms of Feynman's diagrams.

\ms\SS A \emph{Feynman diagram} \index{Feynman diagram} $\gamma$ labeled by random variables $\xi_1,\cdots,\xi_n$ is a graph with vertices $1,2, \cdots, n,$ and edges (``Wick's contractions") $\{v,v'\}$ without common endpoints.
We denote the unpaired vertices by $v''.$
The Wick's value of the diagram is the random variable
\begin{equation}\label{eq: Wick's value}
\odot(\gamma)= \prod_{\{v,v'\}}\E[\xi_{v}\xi_{v'}]~\underset{v''}{\textstyle\bigodot} \xi_{v''}.
\end{equation}
For example, the Feynman diagram with two edges $\{1,4\},\{3,5\}$ and two unpaired vertices $2,6$
corresponds to
$$\contraction{}{(\xi}{_1 \odot \xi_2 \odot \xi_3)(}{\xi} \contraction[2ex]{(\xi_1 \odot \xi_2 \odot }{\xi}{_3)(\xi_4 \odot}{\xi} (\xi_1 \odot \xi_2 \odot \xi_3)(\xi_4 \odot \xi_5 \odot \xi_6) := \E[\xi_1\xi_4] \E[\xi_3\xi_5] \xi_2\odot \xi_6.$$

\bs\textbf{Wick's formula.} \index{Wick's!formula}
\emph{Let $\xi_{jk},$ $(1\le j\le l, 1\le k \le m_j),$ be centered jointly Gaussian random variables, and let $X_j = \xi_{j1}\odot \cdots \odot \xi_{jm_j}.$ Then
$$X_1\cdots X_l = \sum_\gamma\odot(\gamma),$$
where the sum is taken over all Feynman diagrams (labeled by the variables $\xi_{jk}$) such that no edge joins $\xi_{j_1k_1}$ and $\xi_{j_2k_2}$ with $j_1=j_2.$}

\ms\subsec{Wick's powers and exponentials}
If $\xi$ is a centered Gaussian with variance $\sigma^2,$ then
\begin{equation} \label{eq: Wick-power}
\xi^{\odot n} = \sigma^n H_n\left(\frac \xi\sigma\right),
\end{equation}
where $H_n$ are the \emph{Hermite polynomials},
\begin{equation} \label{eq: Hermite}
H_2(x)=x^2-1, \quad H_3(x)=x^3-3x,\quad H_4(x)=x^4-6x^2+3,\quad \cdots.
\end{equation}
Recall that the polynomials $H_n$ are monic and orthogonal with respect to the standard Gaussian measure on $\R,$ so \eqref{eq: Wick-power} is just the chaos decomposition in the case $\dim\HH=1.$
We define
$$e^{\odot\xi} := \sum_{n=0}^\infty \frac{\xi^{\odot n}}{n!}.$$
Using the generating function
\begin{equation} \label{eq: generating fcn}
e^{tx-{t^2}/2}=\sum_{n=0}^\infty \frac{t^n}{n!}H_n(x),
\end{equation}
we get
$$e^{\odot\xi} =e^{\xi -\frac12\E \xi^2}.$$
In particular, if $\xi$ and $\eta$ are jointly Gaussian, then
$$e^{\odot \xi}e^{\odot \eta} = e^{\odot (\xi+\eta)} e^{\E \xi\eta} ,\qquad \E[e^{\odot \xi}e^{\odot \eta}] = e^{\E \xi\eta}.$$

\ms\section{Fock space correlation functionals} \label{sec: corr fcnl}

\ms\SS Let $D$ be a domain in $\C$ and let $\Phi$ be the Gaussian free field in $D.$
By definition \emph{basic} correlation functionals \index{basic correlation functional} (the use of word ``functionals" will be explained later in this section) are formal expressions of the type
$$X_1(z_1)\odot\cdots\odot X_n(z_n),$$
where points $z_j\in D$ are not necessarily distinct and $X_j$'s are derivatives of the Gaussian free field, (i.e., $X_j=\pa^{\alpha}\bp^{\beta}\Phi$).
We also include the constant $1$ to the list of basic functionals.

\ms A general Fock space correlation functional \index{Fock space correlation functional} $\XX$ is a linear combination (over $\C$) of basic functionals.
We allow some infinite combinations, e.g., the exponentials
$$e^{\odot \alpha\Phi(z)}=\sum_{n=0}^\infty \frac{\alpha^n}{n!}\Phi^{\odot n}(z).$$
For our purposes it will suffice to consider the class of \emph{quasi-polynomial} functionals that consists of finite linear combination of $\odot$-products of exponentials and basic functionals.
This class is a graded commutative algebra (with respect to formal chaos decomposition and Wick's multiplication), e.g.,
$$e^{\odot \alpha\Phi(z)}\odot e^{\odot \beta\Phi(z)}=e^{\odot (\alpha+\beta)\Phi(z)}.$$

\ms\textbf{Notation.} We will write $S_\XX$ or $S(\XX)$ for the (finite) set of all points $z_j,$ the \emph{nodes} \index{node} of $\XX,$ appearing (after cancellations) in the expression of $\XX.$

\ms In the rest of the section we explain (or rather \emph{define}) various natural operations on correlation functionals such as (tensor) products, ``expectations", weak convergence, and complex conjugation.
In addition, we will need to explain the meaning of the statements like ``$J = \pa\Phi$ is purely imaginary on the boundary."

\ms\subsec{Tensor products} \index{tensor product}
We use Wick's formula, which describes products of Gaussians in terms of their Wick's products, to define the usual (or tensor) products $\XX_1\cdots \XX_m$ of correlation functionals with pairwise \emph{disjoint} sets $S(\XX_j).$
Namely, for basic functionals
$$\XX_j = X_{j1}(z_{j1}) \odot \cdots \odot X_{jn_j}(z_{jn_j}),\qquad (\textrm{fields~} X_{jk}\textrm{~are~ derivatives~ of}\; \Phi), $$
we set (cf. \eqref{eq: Wick's value})
\begin{equation} \label{eq: Wick's4FCF}
\XX_1\cdots \XX_m = \sum\prod_{\{v,v'\}} \E[X_{v}(z_{v})X_{v'}(z_{v'})] \underset{v''}{\textstyle\bigodot} X_{v''}(z_{v''}),
\end{equation}
where the sum is taken over Feynman diagrams with vertices $v$ labeled by functionals $X_{jk}$ such that there are no contractions of vertices with the same $j,$ and the Wick's product is taken over unpaired vertices $v''.$
By definition, the ``expectations" in \eqref{eq: Wick's4FCF} are given by the 2-point functions of derivatives of the Gaussian free field, e.g.,
$$\E[\pa^j\Phi(\zeta)\pa^k\Phi(z)] = \pa_\zeta^j \pa_z^k\E[\Phi(\zeta)\Phi(z)] = 2 \pa_\zeta^j \pa_z^k G(\zeta,z).$$
We extend the definition of tensor product to general correlation functionals by linearity.

\begin{prop}\label{tensor product}
The tensor product of correlation functionals is commutative and associative.
\end{prop}
Commutativity is of course obvious.
To prove that
\begin{equation} \label{eq: associative}
\XX_1\cdots \XX_m \YY_1\cdots\YY_n = (\XX_1\cdots \XX_m)(\YY_1\cdots\YY_n),
\end{equation}
one needs to show that there is one-to-one correspondence between Feynman's diagrams corresponding to the left-hand side and the right-hand side of \eqref{eq: associative}, which is an easy exercise.

\ms An alternative argument is as follows.
Approximate the values $X_j(z_j)$ of derivatives of the Gaussian free field involved in the formula by jointly Gaussian variables, see Appendix~\ref{appx: F-field}.
Then apply Wick's calculus to the Gaussians, and take the limit.

\ms\subsec{Expectation values of functionals}
We define $\E\XX$ in terms of the chaos decomposition of $\XX$:
$$\E[1] = 1,$$
and
$$\E[X_1(z_1)\odot\cdots\odot X_n(z_n)] = 0, \qquad (\textrm{fields~} X_j\textrm{~are ~derivatives ~ of ~ } \Phi).$$
For example,
$$\E[\Phi(z_1)\cdots \Phi(z_n)] = \sum\prod_k 2G(z_{i_k},z_{j_k})$$
(see \eqref{eq: Wick's}) and
$$\E[e^{\odot\alpha\Phi(z)}] = 1.$$
Since tensor products of functionals are defined by Wick's formula, our definition of $\E\XX$ is consistent with the definition of the $n$-point correlation functions of derivatives of the Gaussian free field introduced earlier.
Correlation functions are ``expected values" of correlation functionals.

\ms Given $\XX,$ consider the linear space $V_\XX=\{\YY:~S_\YY\cap S_\XX=\emptyset\}.$
We have a linear map
$$V_\XX\to \C,\qquad \YY \mapsto \E[\XX\YY],$$
so we can think of $\XX$ as a linear functional on $V_\XX.$
This explains our terminology (``functionals") and also introduces some kind of weak topology in the space of functionals.
For example, the statement
$$\Phi(z_1)\odot\Phi(z_2)\to\Phi^{\odot2}(z),\qquad (z_1,z_2\to z)$$
means (by definition) that
$$\E[(\Phi(z_1)\odot\Phi(z_2))\XX]\to\E[\Phi^{\odot2}(z)\XX]$$
for every $\XX$ such that $z\not\in S_\XX.$
Essentially all statements in conformal field theory have a similar meaning (they hold ``within correlations").

\ms \subsec{Trivial functionals}
From the point of view of calculus of correlations, we can identify functionals $\XX_1$ and $\XX_2$ such that
$$\E[\XX_1\YY] =\E[\XX_2\YY]$$
for all $\YY$ with nodes outside $S_{\XX_1}\cup S_{\XX_2}.$
In this case, we will write $\XX_1\approx \XX_2$ and later just $\XX_1=\XX_2.$

\begin{eg*}
$(\pa\bp\Phi)(z)\approx0.$
Of course, $\pa\bp\Phi\ne0$ as a Gaussian distributional field.
\end{eg*}

\ms It is easy to check that for all $\YY,$
$$\textrm{if }\XX\approx0,\quad \textrm{then}\quad \XX\odot\YY\approx0, \quad \XX\YY\approx0.$$
In particular, $\NN=\{\XX\approx0\}$ is an ideal of Wick's algebra, so we effectively consider Fock space functionals modulo $\NN.$
Also, it is clear that
$$\XX\approx0~\textrm{if and only if}~\E[\XX(\Phi(z_1)\odot\cdots\odot\Phi(z_n))] = 0\quad (\Phi \textrm{ is the Gaussian free field})$$
for all $n$ and all sets $\{z_1,\cdots,z_n\}$ in $D\sm S_\XX.$
In particular, $\XX$ is trivial if and only if all its chaos decomposition components are trivial, and therefore the factor algebra preserves the grading.

\ms\SS Often we can extend the concept of a correlation functional $\XX$ to the case when some of the nodes of $\XX$ lie on the boundary -- we simply define the correlations $\E[\XX\YY]$ in terms of the boundary values.
\begin{eg*}
For $z\in\pa D,$ $e^{\odot\alpha\Phi(z)} = 1.$
\end{eg*}
There is a natural operation of complex conjugation on correlation functionals:
$$\Phi(z)=\overline{\Phi(z)},\qquad (\bp\Phi)(z)=\overline{(\pa\Phi)(z)}, \qquad \overline{\XX\odot \YY}=\bar\XX\odot\bar\YY.$$
More generally, the functional $\bar\XX$ is defined (modulo $\NN$) by the equation
$$\E[\bar\XX\YY]=\overline{\E[\XX\YY]}$$
for all $\YY$'s of the form $\Phi(z_1)\odot\cdots\odot\Phi(z_n).$
\begin{eg*}
If $J=\pa\Phi$ in the half-plane $\H$ and if $z\in\pa\H,$ then
$J(z)$ is purely imaginary, i.e., $\overline{J(z)}=-J(z),$ and $J(z)\odot J(z)$ is real.
\end{eg*}

\ms \section{Fock space fields} \label{sec: F-field}

\SS Basic Fock space fields $X_\alpha$ are formal expressions written as Wick's products of derivatives of the Gaussian free field $\Phi,$ e.g., $$1,\quad \Phi\odot \Phi,\quad \pa \Phi\odot \bp\Phi,\quad \pa^2\Phi\odot\Phi\odot\Phi,\quad\rm etc.$$
A general Fock space field \index{Fock space field} is a linear combination of basic fields $X_\alpha,$
$$X=\sum_\alpha f_\alpha X_\alpha,$$
where the (\emph{basic field}) \emph{coefficients} \index{basic field coefficient} $f_\alpha $ are arbitrary (smooth) functions in $D.$
We think of $X$ as a map
$$z \mapsto X(z), \qquad (z\in D),$$
where the values $\XX=X(z)$ are correlation functionals with $S_\XX \subset \{z\}.$
Thus Fock space fields are functional-valued functions.
Wick's powers $\Phi^{\odot n}$ and Wick's exponentials $e^{\odot \alpha \Phi}$ of the Gaussian free field are important examples of Fock space fields.

\ms If $X_1,\cdots,X_n$ are Fock space fields and $z_1,\cdots,z_n$ are distinct points in $D,$ then
$$\XX = X_1(z_1)\cdots X_n(z_n)$$
is a correlation functional.
We often refer to its ``expectation"
\begin{equation} \label{eq: correlation}
\E[X_1(z_1)\cdots X_n(z_n)]
\end{equation}
as a correlation function.

\ms The collection of Fock space fields (modulo $\NN,$ the ideal of fields whose values are trivial functionals) is a graded commutative algebra (over smooth functions) with respect to pointwise Wick's multiplication.
On the other hand, the ``usual" product $(X_1,X_2)\mapsto X_1X_2$ is not defined, but we can consider the tensor products, which are multivariable fields.
For example,
$$X = X_1\otimes X_2$$
is defined in $D\times D \setminus \{ \textrm{diagonal} \}.$
Its value at $(z_1,z_2)$ is the ``string" $X_1(z_1)X_2(z_2).$

\ms\begin{rmk*}
We often consider Fock space fields with basic field coefficients defined only in some open set $U\subset D$ (``local fields").
It is important that underlying basic fields are global (originated from the Gaussian free field in $D$).
\end{rmk*}

\ms\SS We define the differential operators $\pa$ and $\bp$ on Fock space fields by specifying their action on basic fields so that the action on $\Phi$ is consistent with the definition of $\pa\Phi, \bp\Phi$ (as distributional fields) and so that
$$\pa (X\odot Y)=(\pa X)\odot Y+ X\odot (\pa Y),\qquad \bp (X\odot Y)=(\bp X)\odot Y+ X\odot (\bp Y).$$
We extend this action to general Fock space fields by linearity and by Leibniz's rule with respect to multiplication by smooth functions.

\begin{egs*}
$\pa\bp[\Phi\odot\Phi] \approx 2J\odot \bar J,\qquad
\pa e^{\odot\alpha\Phi} = \alpha J\odot e^{\odot\alpha\Phi}.$
\end{egs*}

\ms It is easy to see that $\pa X$ is a unique (modulo $\NN$) field satisfying
$$\E[(\pa X)(z)\YY] = \pa_z \E[X(z)\YY],\qquad (z\not\in S_\YY),$$
for all correlation functionals $\YY.$
Also, it is clear that $\bp \bar X=\overline{\pa X}.$

\ms\SS By definition, $X$ is \emph{holomorphic} \index{Fock space field!holomorphic} in $D$ if $\bp X\approx 0,$ i.e., all correlation functions $\E[X(\zeta) \YY]$ are holomorphic in $\zeta\in D\sm S_\YY.$ \index{holomorphic Fock space field}
\begin{egs*}
$J=\pa \Phi,$ $X = J\odot J$ are holomorphic fields.
\end{egs*}
Holomorphic fields play a prominent role in conformal field theory.
Their properties are quite different from those of usual holomorphic functions, and some formulas involving holomorphic fields look unfamiliar from the point of view of ``classical" complex analysis.

\renewcommand\chaptername{Appendix}
\chapter{Fock space fields as (very) generalized random functions} \label{appx: F-field}

In this appendix we want to substantiate the concept of Fock space functionals and fields, which we introduced as somewhat formal algebraic objects.
We already mentioned that we can think of functionals as ``generalized" elements of the Fock space, and therefore view fields as ``generalized" random functions (cf. fields in lattice models).
One way to make this point of view clear is to approximate correlation functionals by genuine random variables.

\section{Approximation of correlation functionals by elements of the Fock space} \label{sec: approx}

For each $z\in D,$ let us choose test functions $f_{\ve,z}(\cdot)$ supported in a disc of radius $\ve$ about $z$ and satisfying
$$f_{\ve,z}\to \delta_z\quad \textrm{as}\quad\ve\to0$$
(as measures). Define Gaussian random variables
$$\Phi_\ve(z)=\Phi(f_{\ve,z}),\quad (\Phi \textrm{ is the Gaussian free field}),$$
$$ J_\ve(z)=J(f_{\ve,z})=-\Phi(\pa f_{\ve,z}),\qquad
(\pa^2\Phi)_\ve(z)=\Phi(\pa^2 f_{\ve,z}),\qquad\rm etc.$$
Varying $z,$ we get random functions which approximate the Gaussian free field and its derivatives in the sense of convergence of correlation functions.
For example, we have
$$\E\left[J_\ve(z_1)J_\ve(z_2)\right] ~\to~\E\left[J(z_1)J(z_2)\right],\qquad (z_1\ne z_2).$$
Indeed, the left-hand side,
$$2\iint G(\zeta, \eta)\pa f_{\ve,z_1}(\zeta)\pa f_{\ve,z_2}(\eta)=2\iint \pa_\zeta\pa_\eta G(\zeta, \eta) f_{\ve,z_1}(\zeta) f_{\ve,z_2}(\eta),$$
converges to $2\pa_1\pa_2G(z_1,z_2)= \E\left[J(z_1)J(z_2)\right].$
Usually, when there is no danger of confusion, we omit $dA(\zeta),$ etc.

\ms Next, we extend this approximation to Wick's products, and therefore to general Fock space functionals/fields. For example, we define
$$ X_\ve=-\frac12(J\odot J)_\ve:=-\frac12(J_\ve \odot J_\ve), \qquad X_\ve: D\to\HH^{\odot2},$$
where $\HH^{\odot2}$ is the symmetric tensor square of the Hilbert space $\HH = \EE(D),$ see Section~\ref{sec: F-space}.
Again, it is clear that the correlations of $ X_\ve$ converge to the corresponding correlations of the field $X=-\frac12J\odot J.$
This follows from Wick's formula and from the convergence of the 2-point function of $J_\ve$ established in the previous paragraph.

\ms Thus we can say that Fock space fields are ``generalized" random functions -- they are limits of random functions in the sense of correlations.
(This point of view is somewhat similar to the definition of Colombeau's ``generalized" functions (see~\cite{Colombeau85}).)

\ms In practical terms, we can use approximating random functions to compute correlations of Fock space fields at distances much greater than the ``wavelength" $\ve.$
Moreover, we can give a similar interpretation to other equations of conformal field theory.
For instance, operator product expansions, which we discuss in the next lecture, hold on approximate level as $\ve\ll|\zeta-z|\to0$ so that the error term $``o(1)"$ has vanishing correlations with all fields at positive distance from $z.$
For example,
$$\Phi_\ve(\zeta) \Phi_\ve(z) = \log \frac{1}{|\zeta-z|^2} + 2c(z) + \Phi_\ve^{\odot 2}(z) + o(1) \qquad\textrm{as }\;\ve\ll|\zeta-z|\to0.$$
Here, $c(z)$ is the logarithm of conformal radius $C(z),$ \index{conformal radius $C$}
\begin{equation} \label{eq: log C}
c(z) = \log C(z),\qquad C(z) = \left|\frac{w(z)-\overline{w(z)}}{w'(z)}\right|,
\end{equation}
where $w$ is a conformal map from $D$ onto the upper half-plane $\H.$
The logarithm of conformal radius can be described in terms of the Green's function, see \eqref{eq: c1}, \eqref{eq: c2}, and \eqref{eq: c3}.

\section{Distributional fields} \label{sec: distribution}

\SS Some important Fock space fields admit a much stronger, more analytical interpretation.
We say that a Fock space field is \emph{distributional}\index{Fock space field!distributional} if it can be represented by a linear map $f\mapsto X(f)$ from a space of test functions to the space $L^2(\Omega,\P)$ of random variables on some probability space; the Gaussian free field and its derivatives are the simplest examples. \index{distributional Fock space field}
This is the kind of fields studied in axiomatic (Euclidean) field theory; distributional fields also play an important role in analysis and probability theory.
For any test functions with disjoint supports, assuming that $X(f_1)\cdots X(f_n)$ is in $L^1,$  we require
\begin{equation}\label{eq: distributional}
\E[X(f_1)\cdots X(f_n)] = \int\cdots\int \E[X(z_1)\cdots X(z_n)]f_1(z_1)\cdots f_n(z_n).
\end{equation}
As we explained before, $\E$ has different meanings in this formula; $\E$ in the left-hand side is the expectation of random variables and $\E$ in the right-hand side means the correlation function, see \eqref{eq: correlation}.

\ms Let us show that Wick's powers $\Phi^{\odot n}$ and exponentials $e^{\odot \alpha\Phi}$ with $|\alpha|<1$ exist as distributional fields
$$\Psi:~C^\infty_0(D)\to L^2(\Omega,\P),$$
see \cite{DS11} for a stronger statement.

\ms To construct the map $\Psi$ we follow the same idea as in the previous section but we interpret random functions
$$\Psi_\ve:~D\to L^2$$
as linear operators
$$\Psi_\ve:~C^\infty_0(D)\to L^2,\qquad f\mapsto \int f\Psi_\ve$$
and prove convergence in the strong operator topology.

\ms Almost any choice of approximating random functions will do the job but the estimates are particularly simple if we define
$$\Phi_\ve(z)=\Phi(m_{z,\ve}),$$
where $m_{z,\ve}$ is the normalized arclength of the circle of radius $\ve\ll1$ centered at $z.$

\begin{prop*}
As $\ve\to0,$ $\Phi_\ve\to \Phi$ in the sense that for all test functions $f$ the random variables $\Phi_\ve(f)$ converge to $\Phi(f)$ in $L^2.$
\end{prop*}

\begin{proof}
Note that $\Phi_\ve(\zeta)$ is a centered Gaussian random variable with
\begin{equation} \label{eq: var Phi e}
\var\left(\Phi_\ve(\zeta)\right) = 2\|m_{\zeta,\ve}\|^2_\EE=2\log\frac1\ve+2c(\zeta),
\end{equation}
where $c(\zeta)$ is the logarithm of conformal radius of $D,$ see \eqref{eq: log C}.
Indeed,
$$\var\left(\Phi_\ve(\zeta)\right) = 2\iint G(\xi,\eta) \,dm_{\zeta,\ve}(\xi)\,dm_{\zeta,\ve}(\eta).$$
Set $u(\zeta,z) = G_D(\zeta,z) + \log|\zeta-z|.$
Then the logarithm of conformal radius can be written in terms of the Green's function as follows:
\begin{equation} \label{eq: c1}
c(\zeta) = u(\zeta,\zeta).
\end{equation}

Using the harmonicity of the map $z\mapsto u(\zeta,z),$ we have the following expression for the Green potential $U_D^{m_{\zeta,\ve}}(=\int m_{\zeta,\ve}(\xi) G(\xi,\cdot))$ of $m_{\zeta,\ve}$:
$$
U_D^{m_{\zeta,\ve}}(\eta) =
\begin{cases}
u(\zeta,\eta) + \log \dfrac1\ve, &\quad \textrm{if }|\zeta -\eta|\le\ve;\\
u(\zeta,\eta) + \log \dfrac1{|\zeta -\eta|}, &\quad \textrm{otherwise.}
\end{cases}
$$
Thus we have
$$\var(\Phi_\ve(\zeta)) = 2\|m_{\zeta,\ve}\|^2_\EE = 2\log\frac1\ve+ 2\int u(\zeta,\eta) \,dm_{\zeta,\ve}(\eta) =2\log\frac1\ve + 2u(\zeta,\zeta),$$
which shows \eqref{eq: var Phi e}.
Arguing as above, we show
$\E[\Phi_\ve(\zeta)\Phi_\ve(z)] = 2K_\ve(\zeta,z) + 2u(\zeta,z),$ where
\begin{equation}\label{eq: K}
\begin{cases}\bs\phantom{|}K_\ve(\zeta,z)\phantom{|} = \log\dfrac1{|\zeta-z|},\qquad &\textrm{if } |\zeta-z|\ge2\ve; \\
|K_\ve(\zeta,z)| \le \log\dfrac1\ve, &\textrm{otherwise.}
\end{cases}
\end{equation}
Integrating against test functions $f,$
$\operatorname{var}\left(\Phi_\ve(f)\right) \to \operatorname{var}\left(\Phi(f)\right).$
In a similar way, $\operatorname{cov}(\Phi_\ve(f),\Phi(f)) \to \operatorname{var}\left(\Phi(f)\right).$
Therefore, we obtain
$\operatorname{var}\left(\Phi_\ve(f)-\Phi(f)\right) \to 0.$
\end{proof}

\ms\subsec{Exponentials and powers of the Gaussian free field}
We represent Wick's powers $\Phi^{\odot n}$ and exponentials $e^{\odot \alpha\Phi}$ with $|\alpha|<1$
as \emph{distributional} fields in the following way:
\begin{equation}
\Phi^{\odot n}=\lim_{\ve\to 0}\Phi_\ve^{\odot n},\qquad
e^{\odot\alpha\Phi}=\lim_{\ve\to 0} e^{\odot\alpha\Phi_\ve},
\end{equation}
where the limits are in the strong operator topology.
The existence of the limits is shown below.
Thus we have
$$
e^{\odot\alpha\Phi}=C(z)^{-\alpha^2}\lim_{\ve\to 0}\ve^{\alpha^2}\,e^{\alpha\Phi_\ve},\qquad (|\alpha|<1),
$$
where $C(z)$ is the conformal radius (see \eqref{eq: log C} and \eqref{eq: var Phi e}) and
$$\Phi^{\odot n}=\lim_{\ve\to 0} ~\sigma_\ve^n H_n\left(\frac{\Phi_\ve}{\sigma_\ve}\right),\qquad (\sigma_\ve^2 = \operatorname{var}(\Phi_\ve)),$$
where $H_n$'s are the Hermite polynomials, see \eqref{eq: Hermite}.
For example,
$$\Phi^{\odot 2}=\lim_{\ve\to 0} ~ \Phi_\ve^2-2\log\frac{C}{\ve},\qquad\Phi^{\odot 4}=\lim_{\ve\to 0} ~ \Phi_\ve^4-12\Phi_\ve^2\log\frac{C}{\ve}+12\log^2\frac{C}{\ve}.$$

\begin{prop*}
\renewcommand{\theenumi}{\alph{enumi}}
{\setlength{\leftmargini}{1.7em}
\begin{enumerate}
\item Suppose $|\alpha|<1.$
\renewcommand{\theenumii}{\roman{enumii}}
{\setlength{\leftmarginii}{1.2em}
\begin{enumerate}
\ms \item For all test functions $f,$ the random variables $e^{\odot\alpha\Phi_\ve}(f)$ converge in $L^2$ as $\ve\to0.$
\ms \item Let $e^{\odot\alpha\Phi}(f)$ denote the $L^2$-limit. 
For any test functions with disjoint supports, the random variable $e^{\odot\alpha\Phi}(f_1)\cdots e^{\odot\alpha\Phi}(f_n)$ is in $L^1.$
\ms \item The linear map $e^{\odot\alpha\Phi}:f\mapsto e^{\odot\alpha\Phi}(f)$ is distributional in the sense that \eqref{eq: distributional} holds.
\end{enumerate}}
\ms \item Similar properties hold for Wick's powers $\Phi^{\odot n}.$
\end{enumerate}}
\end{prop*}

\begin{proof}
(a) (i)
Given a sequence $\{\ve_m\}_{m=1}^\infty$ with $\ve_m\downarrow0,$ we set $\Phi_m=\Phi_{\ve_m}.$
Note that
$$\cov(e^{\odot\alpha\Phi_m}(f),e^{\odot\alpha\Phi_n}(f)) = \iint\cov(e^{\odot\alpha\Phi_m}(\zeta),e^{\odot\alpha\Phi_n}(z))f(\zeta)\overline{f(z)},$$
where
$$\cov(e^{\odot\alpha\Phi_m}(\zeta),e^{\odot\alpha\Phi_n}(z))= \exp(|\alpha|^2 \E[\Phi_m(\zeta)\Phi_n(z)])-1.$$
It follows from the estimate on $\E[\Phi_m(\zeta)\Phi_n(z)]$ similar to \eqref{eq: K} that
\begin{align*}
\Big|\cov(e^{\odot\alpha\Phi_m}(f),e^{\odot\alpha\Phi_n}(f))-&\iint_{|\zeta-z| \ge \ve_m + \varepsilon_n} \big(e^{2|\alpha|^2G(\zeta,z)}-1\big)f(\zeta)\overline{f(z)} \Big|\\
\le \Big|&\iint_{|\zeta-z| < \varepsilon_m+\ve_n}\big(\ve^{-2|\alpha|^2}e^{2u(\zeta,z)}-1\big)f(\zeta)\overline{f(z)}\Big|,
\end{align*}
where $\ve = \max(\ve_m,\ve_n).$
If $|\alpha|<1,$ then the right-hand side in the above estimate tends to $0$ as $\min(m,n)\to\infty.$
On the other hand, if $|\alpha|<1,$ then the integral
$$\iint e^{2|\alpha|^2G(\zeta,z)} f(\zeta)\overline{f(z)}$$
is finite.
Thus $\{e^{\odot\alpha\Phi_m}(f)\}$ is a Cauchy sequences in $L^2,$ which has an $L^2$-limit.
This limit does not depend on a particular sequence $\{\ve_m\}_{m=1}^\infty.$

\ss (ii)
By (i), there is an almost sure convergent subsequence $\{e^{\odot\alpha\Phi_{m_k}}(f_j)\}.$
We first note that for all $m,$
\begin{align} \label{eq: dist ve}
\E\big[e^{\odot \alpha\Phi_m}{(f_1)}~&\cdots~ e^{\odot \alpha\Phi_m}{(f_n)}\big] \\ &= \int\cdots\int e^{\alpha^2 \sum_{j<k}\E[\Phi_m(z_j)\Phi_m(z_k)]}~f_1(z_1)\cdots f_n(z_n). \nonumber
\end{align}
It follows from the estimate~\eqref{eq: K} that 
$$\sup_m \E|e^{\odot \alpha\Phi_m}{(f_1)}~\cdots~ e^{\odot \alpha\Phi_m}{(f_n)}|<\infty.$$
Thus the random variable $e^{\odot\alpha\Phi}(f_1)\cdots e^{\odot\alpha\Phi}(f_n)$ is in $L^1.$
Furthermore,
\begin{equation} \label{eq: conv in L1}
e^{\odot \alpha\Phi_{m_k}}{(f_1)}~\cdots~ e^{\odot \alpha\Phi_{m_k}}{(f_n)} \overset{L^1}{\longrightarrow} e^{\odot\alpha\Phi}(f_1)~\cdots~e^{\odot\alpha\Phi}(f_n).
\end{equation}

\ss (iii)
It follows from the estimate~\eqref{eq: K} that the right-hand side of \eqref{eq: dist ve} converges to
$$\idotsint e^{2\alpha^2\sum_{j<k}G(z_j,z_k)} ~f_1(z_1)\cdots f_n(z_n).$$
On the other hand, we have
\begin{equation} \label{eq: n-pt4expGFF}
\E\left[e^{\odot \alpha_1\Phi}{(z_1)}\cdots~ e^{\odot \alpha_n\Phi}{(z_n)}\right]=e^{2\sum_{j<k}\alpha_j\alpha_k~G(z_j,z_k)}.
\end{equation}

Using \eqref{eq: conv in L1}, by passing to a subsequence, the left-hand side of \eqref{eq: dist ve} converges to
$\E\left[e^{\odot \alpha\Phi}{(f_1)}~\cdots~ e^{\odot \alpha\Phi}{(f_n)}\right].$
Thus the linear map $e^{\odot \alpha\Phi}$ is distributional.

\ss (b)
Project $e^{\odot\alpha\Phi_\ve}(f)$ onto $\HH^{\odot n}.$
Then the convergence of $\Phi_\ve^{\odot n}(f)$ in $L^2$ follows from the convergence of $e^{\odot\alpha\Phi_\ve}(f).$
The other parts are left to the reader.
\end{proof}

\ms \begin{rmks*}
(a) As Fock space fields, Wick's exponentials satisfy \eqref{eq: n-pt4expGFF} without any restriction on $\alpha_j$'s.

\ss (b) Exponentials with $|\alpha|\ge1$ cannot be distributional since the \emph{positive} $2$-point function
$$\E\left[e^{\odot \alpha\Phi(z)}e^{\odot \bar\alpha\Phi(w)}\right] = e^{2|\alpha|^2G(z,w)}$$
is not integrable in $D\times D.$
\end{rmks*}

\section{Insertion operators} \label{sec: insert}

In this section we will use the distributional representation of the Gaussian free field to explain the mechanism of the insertion procedure, an operation widely used in the field theory.

\ms Let $\Phi:\EE\to L^2(\Omega,\P)$ be the Gaussian free field in $D.$
Given a real distribution $\rho\in \mathcal{E}$ we define the probability measure $\widehat\P\equiv\P_\rho$ on $\Omega$ by the equation (the ``Cameron-Martin" change of measure)
$$d\widehat \P=e^{\odot\Phi(\rho)}\,d{\P}.$$
The following proposition describes the random field $\Phi: \EE\to L^2(\Omega,\widehat \P),$ which is the composition of the Gaussian free field and the identity map $L^2(\Omega, \P)\to L^2(\Omega,\widehat \P),$ in terms of the Green potential $U_D^\rho = \int \rho(\zeta)G_D(\cdot,\zeta).$

\begin{prop*}
The law of $\Phi$ with respect to $\widehat\P$ (i.e., under the insertion of $e^{\odot\Phi(\rho)}$) is the same as the law of $\widehat\Phi:=2U_D^\rho+\Phi$ with respect to $\P.$
\end{prop*}

\begin{proof}
For a test function $f,$ let us compute the characteristic functions of $\Phi(f)$ with respect to $\P_\rho.$
We have
\begin{align*}
\log\E_{\rho}\left[e^{it\Phi(f)}\right] &= \log\E\left[e^{\odot\Phi(\rho)}e^{it\Phi(f)}\right]=\log\left(e^{-t^2\E\Phi(f)^2/2}\E\left[e^{\odot\Phi(\rho)}e^{\odot it\Phi(f)}\right]\right)\\
&=2 it\iint f(z)\rho(\zeta)G(z,\zeta) -\frac12t^2\|f\|_\EE^2\\
&= 2 it\int f(z)U_D^\rho(z)- \frac12t^2\|f\|_\EE^2.
\end{align*}
This means that $\Phi(f)$ is Gaussian with mean $2\int f U_D^\rho$ and variance $\|f\|_\EE^2,$ see \eqref{eq: norm}.
Proposition follows from uniqueness of the Gaussian free field.
\end{proof}

\ms We use this proposition as the motivation for the following construction on Fock space fields.
Let us now formally take $\rho = \alpha\delta_{z_0}$ (note that $\rho\not\in \EE,$ but $\alpha f_{\ve,z_0} \to \rho$ and $\alpha f_{\ve,z_0}\in \EE$) and define a linear operator $\XX\mapsto\widehat\XX$ on correlation functionals with nodes in $D\sm\{z_0\}$ by the following rules: $$\Phi(z)\mapsto\Phi(z)+2\alpha G(\cdot,z_0),$$
\begin{equation*} 
\pa\XX\mapsto\pa\widehat\XX,\qquad \bp\XX\mapsto\bp\widehat\XX,\qquad \XX\odot\YY\mapsto\widehat\XX\odot\widehat\YY. 
\end{equation*}

We define $\widehat\E[\XX]$ by 
$$\widehat\E[\XX]:=\E[e^{\odot\alpha\Phi(z_0)}\XX].$$

\ms The following proposition (with real $\alpha$) is immediate from the previous proposition if we use the approximation technique described in Section~\ref{sec: approx}.
It is also easy to give a direct proof (which works for complex $\alpha$'s as well).

\begin{prop}\label{insert}
We have 
$$\widehat \E[\XX]=\E[\widehat\XX].$$
\end{prop}

\begin{proof}
Let
$\XX=X_1(z_1)\odot\cdots\odot X_n(z_n), X_j = \pa^{\beta_j}\bp^{\widetilde\beta_j}\Phi.$
Then by Wick's formula we have
$$\widehat\XX= \widehat X_1(z_1) \odot\cdots \odot \widehat X_n(z_n),\qquad \widehat X_j(z_j) = X_j(z_j) + 2\alpha \pa^{\beta_j}\bp^{\widetilde\beta_j} G(z_j,z_0),$$
where we differentiate the Green's function with respect to the first variable.
By definition, we get
$$\widehat{\E}[X_j(z_j)] = \alpha \E[X_j(z_j)\Phi(z_0)] = 2\alpha \pa^{\beta_j}\bp^{\widetilde\beta_j} G(z_j,z_0) = \E[\widehat X_j(z_j)].$$
It follows from the definition \eqref{eq: Wick's4FCF} of tensor products of functionals that
\begin{align*}
\widehat{\E}[\XX] &= \sum_{k=0}^\infty\frac{\alpha^k}{k!}\E[\Phi^{\odot k}(z_0)~X_1(z_1)\odot\cdots\odot X_n(z_n)] \\
&= \alpha^n \prod_{j=1}^n\E[\Phi(z_0)X_j(z_j)]= \prod_{j=1}^n \E[\widehat X_j(z_j)] = \E[\widehat\XX].
\end{align*}

\end{proof}

\renewcommand\chaptername{Lecture}
\chapter{Operator product expansion} \label{ch: OPE} \index{OPE}

Operator product expansion (OPE) is the expansion of the tensor product of two fields near diagonal.
The name originates from the corresponding construction for local operators.
With our approach, we use reverse logic -- operator product expansions of fields are used to define local operators, see Appendix~\ref{appx: Op algebra}.
The concept of operator product expansion is quite general -- the definition does not depend on a particular nature of correlation functions.
In the case of Fock space fields, the OPE coefficients are again Fock space fields, and so we get important algebraic operations (OPE multiplications) on Fock space fields.

\ms \section{Definition and first examples}\label{sec: OPE def}

\ms\SS We start with a simple example.
\ms\begin{eg*}
Let $\Phi$ be the Gaussian free field in $D,$ and let $c(z),$ $z\in D$ denote the logarithm of conformal radius of $D,$ see \eqref{eq: log C} in Appendix~\ref{appx: F-field}.
Then
\begin{equation} \label{eq: OPE(Phi,Phi)}
\Phi(\zeta) \Phi(z) = \log \frac{1}{|\zeta-z|^2} + 2c(z) + \Phi^{\odot 2}(z) + o(1) \qquad\textrm{as }\;\zeta\to z,\;\zeta\ne z.
\end{equation}
As we mentioned in Section \ref{sec: corr fcnl}, the meaning of the convergence (here and in all similar statements) is the convergence of correlation functionals: the equation
$$\E[\Phi(\zeta) \Phi(z)\XX] = \log \frac{1}{|\zeta-z|^2}\E[\XX] + 2c(z)\E[\XX] + \E[\Phi^{\odot 2}(z)\XX] + o(1)$$
holds for all Fock space correlation functionals $\XX$ in $D$ satisfying $z\notin S_\XX.$

\ms To derive the operator product expansion \eqref{eq: OPE(Phi,Phi)} we use Wick's formula
\eqref{eq: Wick's4FCF},
$$\Phi(\zeta) \Phi(z) = \E[\Phi(\zeta) \Phi(z)] + \Phi(\zeta)\odot\Phi(z)$$
and the relation
\begin{equation} \label{eq: c2}
\E[\Phi(\zeta) \Phi(z)] = 2G(\zeta,z) = \log \frac{1}{|\zeta-z|^2} + 2c(z) + o(1),
\end{equation}
see \eqref{eq: c1} for the description of $c(z)$ in terms of the Green's function.
The convergence of $\Phi(\zeta)\odot\Phi(z)$ to $\Phi^{\odot2}(z)$ was already explained in Section \ref{sec: corr fcnl}.\qed
\end{eg*}

\ms\SS In general, the operator product expansion of two Fock space fields is an asymptotic expansion of the correlation functional $X(\zeta)Y(z)$ with respect to some appropriate (and \emph{independent} of $D$) growth scale as $\zeta\to z.$

\ms Particularly important is the case in which the field $X(\zeta)$ is
\emph{holomorphic} (recall that this means that all correlation functions $\E[X(\zeta) \YY]$ are holomorphic with respect to $\zeta\in D\sm S_\YY$).
The operator product expansion is then defined as a (formal) Laurent series expansion
\begin{equation} \label{eq: OPE(X,Y)}
X(\zeta)Y(z)= \sum {C_n(z)}{(\zeta-z)^n}, \qquad\zeta\to z.
\end{equation}
The function $\zeta\mapsto\E\,X(\zeta)Y(z)\ZZ$ is holomorphic in a punctured neighborhood of $z.$ 
Hence it has a Laurent series expansion and its radius of convergence is the shortest distance from $z$ to the nodes of $\ZZ.$

\ss \begin{eg*} Here is an example of a full operator product expansion.
For a given domain $D$ we defined
$$u(\zeta,z)=G(\zeta,z)+\log|\zeta-z|.$$
Since $c(z)=u(z,z),$ we have $\pa c(z)=2\pa_1u(z,z),$ where $\pa_1$ is the complex derivative with respect to the first variable.
The derivatives $$c_n(z):= 2\pa_1^{n+1}u(z,z)$$ appear in the operator product expansion of the fields $J=\pa\Phi$ and $\Phi$:
\begin{align} \label{eq: OPE(J,Phi)}
J(\zeta)\Phi(z)&=\E[J(\zeta)\Phi(z)]+J(\zeta)\odot\Phi(z)\\
&=-\frac1{\zeta-z}+\pa c+J(z)\odot\Phi(z)+\sum_{n=1}^\infty {C_n(z)}{(\zeta-z)^n}, \nonumber
\end{align}
where
$$C_n(z) = \frac1{n!}\Big(c_n(z) + (\pa^{n}J)\odot\Phi(z)\Big).$$
\hfill\qed
\end{eg*}

\ms It is easy to show that there are only finitely many terms in the principle (or \emph{singular}) part of the Laurent series \eqref{eq: OPE(X,Y)} (in the case of ``quasi-polynomial" Fock space fields that we only consider).
Sometimes, we use the notation $\sim$ for the singular part of the operator product expansion,
$$X(\zeta)Y(z)\sim\sum_{n<0} {C_n(z)}{(\zeta-z)^n}.$$
We also write $\Sing_{\zeta\to z}\,X(\zeta)Y(z)$ for the right-hand side of the above equation.
For example, we have (by Wick's calculus)
\begin{equation} \label{eq: OPE(J,V)}
J(\zeta)e^{\odot\alpha\Phi(z)}\sim -\frac{\alpha}{\zeta-z} e^{\odot\alpha\Phi(z)}.
\end{equation}
It is clear that we can differentiate operator product expansions \eqref{eq: OPE(X,Y)} both in $\zeta$ and $z$; and the differentiation preserves singular parts.
For example, differentiating \eqref{eq: OPE(J,Phi)} we have
\begin{equation} \label{eq: sOPE(J,J)}
J(\zeta)J(z)\sim-\dfrac1{(\zeta-z)^2}.
\end{equation}

\ms Also, we should keep in mind that operator product expansion is the expansion of functionals defined modulo $\NN,$ the trivial functionals (see Section~\ref{sec: corr fcnl}), so we can disregard terms like $\bar\partial J$ or $\delta$-functions and their derivatives, e.g.,
\begin{equation} \label{eq: sOPE(J,barJ)}
J(\zeta)\bar J(z)\sim-\partial_{\bar z}\left(\frac1{\zeta-z}\right)\equiv 0.
\end{equation}
More generally, if both $X$ and $Y$ are holomorphic, then
$$X(\zeta)\overline{Y(z)}\sim 0.$$

\ms \section{OPE coefficients} \label{sec: OPE coeffs}

\ms\SS The functionals appearing in the operator product expansions (e.g., $2c(z) + \Phi^{\odot2}(z)$ in \eqref{eq: OPE(Phi,Phi)}, $C_n(z)$ in \eqref{eq: OPE(X,Y)}, or $C_{j,k}$ in \eqref{eq: OPE(V,V)} below) are called \emph{OPE coefficients}. \index{OPE!coefficients}

\begin{prop} \label{OPE coeffs}
OPE coefficients of quasi-polynomial Fock space fields are quasi-polynomial Fock space fields (as functions of $z$).
\end{prop}

\ms The proof is straightforward -- use Wick's calculus and the definition of fields.
Proposition~\ref{OPE coeffs} allows us to define certain operations on Fock space fields.
In particular, if $X$ is \emph{holomorphic}, then we define the $*_n$ product
\begin{equation} \label{eq: OPE coeffs}
X *_{n} Y=C_n,
\end{equation}
see \eqref{eq: OPE(X,Y)} for $C_n.$

\ms\SS We will use the operations $*_{n}$ for all $n$'s, see Lecture \ref{ch: T} and Appendix~\ref{appx: Op algebra}, but in this section we focus on the special case $n=0.$

\ms \textbf{Notation.} We write $*$ for $*_{0}$ and call $X*Y$ the \emph{OPE multiplication}, or the \emph{OPE product} of $X$ and $Y.$ \index{OPE!multiplication (product)}

\ms For example, by \eqref{eq: sOPE(J,barJ)},
$$(J*\bar J)(z) = \lim_{\zeta\to z} [J(\zeta)\bar J(z)]$$
(so we can write $J*\bar J = J\bar J$) and
\begin{align} \label{eq: J*bar J}
(J\bar J)(z) &= (\bar J\odot J)(z) + \E[\bar J(z)J(z)] \\
& = (\bar J\odot J)(z) + \frac{w'(z)\overline{w'(z)}}{(w(z)-\overline{w(z)})^2} = (\bar J\odot J)(z) -\frac1{C(z)^2}, \nonumber
\end{align}
where $C(z)$ is the conformal radius, see \eqref{eq: log C}.
(More generally, if both $X$ and $Y$ are holomorphic, then we can write $X*\bar Y =X\bar Y.$)

\ms The OPE product $X*Y$ as the coefficient of 1 can be defined for some (but not all, see e.g., \eqref{eq: OPE(V,V)}) non-holomorphic fields $X,$ e.g.,
$$\Phi^{*2}=\Phi^{\odot 2}+2c,\qquad (\textrm{see}\; \eqref{eq: OPE(Phi,Phi)}).$$
The field $X*Y$ is obtained by subtracting all divergent terms in operator product expansion and taking the limit, which is a usual procedure in the field theory.

\ms\SS If $f$ is a non-random holomorphic function, then
$$f*X=X*f=f X.$$
However, simple examples show that $(f X)*Y\ne X*(fY)$ in general, so unlike Wick's multiplication, the OPE multiplication is neither associative nor commutative (on holomorphic fields).
On the other hand, $*_n$ satisfies Leibniz's rule
\begin{equation} \label{eq: OPE Leibnitz}
\partial (X*_n Y)=(\partial X*_n Y)+(X*_n \partial Y).
\end{equation}
If $X$ is holomorphic, then differentiation of operator product expansion \eqref{eq: OPE(X,Y)} with respect to $\zeta$ gives
$ (\partial X)*_n Y=(n+1)(X *_{n+1} Y)$ and therefore,
\begin{equation} \label{eq: X*nY}
X *_{n} Y= \frac{1}{n!}(\partial^n X)*Y,\qquad (n\ge1).
\end{equation}
Differentiation of operator product expansion \eqref{eq: OPE(X,Y)} with respect to $z$ then gives \eqref{eq: OPE Leibnitz}.

\ms \section{OPE powers and exponentials of Gaussian free field} \label{sec: OPE GFF}

We already computed
$\Phi^{*2}=\Phi^{\odot 2}+2c,$ where
$c$ is the logarithm of conformal radius $C.$
Further computation with Wick's formula gives
$$\Phi^{*3}:=\Phi*\Phi^{*2}\equiv \Phi^{*2}*\Phi=\Phi^{\odot 3}+6c\Phi,\qquad \textrm{etc.}$$
In fact, we have the following formula.

\begin{prop} We have 
\begin{equation} \label{eq: OPE^n GFF}
\Phi^{\odot n} = (2c)^{n/2}H_n^*\left(\frac{\Phi}{\sqrt{2c}}\right),
\end{equation}
where $H_n(z) = \sum_{k=0}^n a_k z^k$ are the Hermite polynomials (see \eqref{eq: Hermite}) and
$$H_n^*(\alpha\Phi) = \sum_{k=0}^n a_k \alpha^k\Phi^{*k}.$$
\end{prop}

\begin{proof} From
$$\Phi(\zeta)\Phi^{\odot n}(z)=2n G(\zeta,z)\Phi^{\odot (n-1)}(z)+\Phi^{\odot (n+1)}(z)+o(1),$$
we find
$$\Phi*\Phi^{\odot n}=2cn\Phi^{\odot (n-1)}+\Phi^{\odot (n+1)}.$$
Assuming that \eqref{eq: OPE^n GFF} holds for $\Phi^{\odot n}$ and $\Phi^{\odot (n-1)},$ and using recurrence relation
$$H_{n+1}(x)=xH_{n}(x)-nH_{n-1}(x),$$
we prove \eqref{eq: OPE^n GFF} for $\Phi^{\odot (n+1)}.$
\end{proof}

\ms By definition,
$$e^{*\alpha\Phi} = \sum_{n=0}^\infty \alpha^n \frac{\Phi^{*n}}{n!}.$$

\begin{prop} We have 
\begin{equation} \label{eq: exp(*Phi)}
e^{*\alpha\Phi}~=~C^{\alpha^2}~e^{\odot\alpha\Phi}.
\end{equation}
\end{prop}

\begin{proof}
Using the generating function \eqref{eq: generating fcn} for the Hermite polynomials, we get
$$e^{\odot\alpha\Phi} = \sum_{n=0}^\infty \frac{\alpha^n}{n!} \Phi^{\odot n} = \sum_{n=0}^\infty \frac{\alpha^n(2c)^{n/2}}{n!} H_n^*\left(\frac{\Phi}{\sqrt{2c}}\right)=e^{*\alpha\Phi}e^{-c\alpha^2}.$$
\end{proof}

\begin{rmk*}
If we define random functions $\Phi_\ve$ as in Section~\ref{sec: approx}, then we get the formula
$$e^{*\alpha\Phi} =\lim_{\ve\to0}\ve^{\alpha^2}e^{\alpha\Phi_\ve},$$
(convergence of correlation functionals but also convergence in the strong operator topology if $|\alpha|<1.$)
It is remarkable that two different types of normalizations, by averaging and by operator product expansion, produce the same result.
\end{rmk*}

As we mentioned earlier, the OPE multiplication (as the coefficient of $1$ in the operator product expansion) does not make sense for general non-holomorphic fields, but we can of course consider the corresponding OPE coefficients.

\ms \begin{eg*}
Let us denote $\VV^\alpha=e^{*\alpha \Phi}$ (``vertex fields", see Section~\ref{sec: vertex field}).
The operator product expansion of two such fields has the following form:
\begin{equation} \label{eq: OPE(V,V)}
\VV^\alpha(\zeta)\VV^\beta(z) = \frac{1}{|\zeta-z|^{2\alpha\beta}}\big(\sum_{j,k=0}^\infty C_{jk}(z)(\zeta-z)^j(\bar\zeta-\bar z)^k\big).
\end{equation}
The first coefficients are
$$C_{0,0} = \VV^{\alpha+\beta}$$
and
$$
C_{1,0} = \alpha\VV^{\alpha+\beta}\odot(J+(\alpha+\beta)\pa c),\qquad
C_{0,1} = \alpha\VV^{\alpha+\beta}\odot(\bar J+(\alpha+\beta)\bp c).
$$
To see this, first note that
$$\VV^{\alpha}(\zeta)\VV^{\beta}(z) = C(\zeta)^{\alpha^2}C(z)^{\beta^2}\exp(\alpha\beta\E[\Phi(\zeta)\Phi(z)]) e^{\odot\alpha\Phi(\zeta)}\odot e^{\odot\beta\Phi(z)}.$$
We expand both
\begin{align*}
C&(\zeta)^{\alpha^2}C(z)^{\beta^2}\exp(\alpha\beta\E[\Phi(\zeta)\Phi(z)])
= \frac1{|\zeta-z|^{2\alpha\beta}}e^{\alpha^2 u(\zeta,\zeta) + 2\alpha\beta u(\zeta,z) + \beta^2 u(z,z)}\\
&= \frac{C(z)^{(\alpha+\beta)^2}}{|\zeta-z|^{2\alpha\beta}}\big(1+(\alpha^2+\alpha\beta) \pa c(z)(\zeta-z) + (\alpha^2+\alpha\beta) \bp c(z)(\bar \zeta-\bar z) + \cdots \big)
\end{align*}
and
$e^{\odot\alpha\Phi(\zeta)} = e^{\odot\alpha\Phi(z)}\odot\big(1+\alpha J(z)(\zeta-z) + \alpha \bar J(z)(\bar \zeta-\bar z) + \cdots \big).$
\end{eg*}

\bs \section{\texorpdfstring{The field~$T= -\frac12\,J*J$}{The field T = -1/2 J*J}} \label{sec: T0}

\ms We use the OPE multiplication to introduce this important field
(in Lecture~\ref{ch: T} we will identify it with the \emph{Virasoro field} of the Gaussian free field, $\Phi$).
By \eqref{eq: sOPE(J,J)}, $T$ is defined by the operator product expansion
\begin{equation} \label{eq: OPE(J,J)}
J(\zeta)J(z)=-\frac{1}{(\zeta-z)^2}-2T(z)+o(1), \qquad(\zeta\to z).
\end{equation}

\ms To express $T$ in terms of Wick's calculus, we need the \emph{Schwarzian} of $D$: \index{Schwarzian}
\begin{equation} \label{eq: S of D}
S(z) = S(z,z), \qquad S(\zeta,z) := -12\partial_\zeta\partial_z u(\zeta,z),
\end{equation}
where $u(\zeta,z) = G(\zeta,z) + \log|\zeta-z|,$ as usual.

\begin{prop} \label{T in S(1/12)}
We have 
\begin{equation} \label{eq: T in S(1/12)}
T=-\frac12~J\odot
J+\frac1{12}S.
\end{equation}\end{prop}
\begin{proof}
Differentiating $\E[J(\zeta)\Phi(z)]=2\partial_\zeta G(\zeta,z),$ we have
\begin{equation} \label{eq: EJJ} \E[J(\zeta)J(z)]=2\partial_\zeta\partial_z G= -\frac1{(\zeta-z)^2}-\frac16S(\zeta,z).
\end{equation}
\end{proof}

We finish this lecture with several singular parts of operator product expansions involving $T,$ which we will need later.
The operator product expansions can be verified by Wick's calculus.
(Later we will explain them from a different perspective -- in terms of conformal geometry.)

\bs \begin{prop} \quad \label{OPE4T}
We have 
\renewcommand{\theenumi}{\alph{enumi}}
{\setlength{\leftmargini}{2.0em}
\begin{enumerate}
\ms \item \label{item: OPE(T,Phi)}
$T(\zeta)\Phi(z)\sim \dfrac{J(z)}{\zeta-z},$
\ms \item \label{item: OPE(T,J)}
$T(\zeta)J(z)\sim \dfrac{J(z)}{(\zeta-z)^2}+ \dfrac{\partial J(z)}{\zeta-z},$
\ms \item \label{item: OPE(T,T)}
$T(\zeta)T(z)\sim \dfrac{1/2}{(\zeta-z)^4}+\dfrac{2T(z)}{(\zeta-z)^2}+ \dfrac{\partial T(z)}{\zeta-z},$
\ms \item $T(\zeta)\VV^\alpha(z) \sim -\dfrac{\alpha^2}2~\dfrac {\VV^\alpha(z)}{(\zeta-z)^2}+\dfrac{\partial \VV^\alpha(z)}{\zeta-z}, \qquad(\VV^\alpha=e^{*\alpha \Phi}).$
\end{enumerate}
}
\end{prop}

\begin{proof}
We only explain \eqref{item: OPE(T,T)}.
Use \eqref{eq: T in S(1/12)}, Wick's theorem, and \eqref{eq: OPE(J,J)} to express the singular term of $T(\zeta)T(z)$ as
\begin{gather*}
\frac14\left(J(\zeta)\odot J(\zeta) \right) \left(J(z)\odot J(z)\right) \sim \frac12\left(\frac16S(\zeta,z)+\frac1{(\zeta-z)^2}\right)^2-\frac{J(\zeta)\odot J(z)}{(\zeta-z)^2} \\
\sim \frac{1/2}{(\zeta-z)^4}-\frac{J(\zeta) \odot J(z)-S(\zeta,z)/6}{(\zeta-z)^2}.
\end{gather*}
The numerator $J(\zeta) \odot J(z)-S(\zeta,z)/6$ of the last term in the above is equal (up to the second order terms) to
\begin{align*}
J(z) \odot J(z) &-\frac16 S(z,z)+(\zeta-z)\left(\partial J(z) \odot J(z) -\frac16\partial_\zeta\big|_{\zeta=z} S(\zeta,z)\right) \\
=&-2T(z)+(\zeta-z)\left(\partial J(z) \odot J(z) -\frac1{12}\partial_z S(z)\right) \\
=&-2T(z)-(\zeta-z)\partial T(z),
\end{align*} which completes the proof.
\end{proof}

\renewcommand\chaptername{Lecture}
\chapter{Conformal geometry of Fock space fields} \label{ch: conf geom}

In the first lecture we defined Fock space fields in a planar domain.
We will now revise this definition and equip fields with certain geometric (or conformal) structures.
We call them \emph{conformal} Fock space fields. \index{Fock space field!conformal}
After we explain the definition, we will typically drop the epithet ``conformal." \index{conformal Fock space field}

\ms Even if we only consider functionals and fields in the half-plane, it is necessary to think of them as defined on a Riemann surface -- their correlations depend on the choice of \emph{local} coordinates at the nodes.
For example, the fields $J\odot J$ and $J*J,$ as we defined them in Lecture \ref{ch: OPE}, have the same correlation functions in the half-plane but as conformal fields they are different -- the first one is a quadratic differential and the second one is a Schwarzian form.

\ms At the end of this lecture we discuss the concept of the Lie derivative of a conformal field. This concept will be used in the next lecture to define the stress tensor and to state Ward's identities.

\ms \section{Non-random conformal fields} \label{sec: non-random}
Recall that a local coordinate chart in a domain $D$ (or more generally on a Riemann surface $M$) is a conformal map
$$\phi:U\to\phi(U) \subset \C,\qquad (U\subset D \;\rm open).$$
The transition map between two overlapping charts $\phi$ and $\widetilde\phi$ is a conformal transformation
$$h=\widetilde\phi\circ\phi^{-1}:~ \phi(U\cap\widetilde U)\to \widetilde\phi(U\cap\widetilde U). $$

\ms By definition, a non-random \emph{conformal} field $f$ is an assignment of a (smooth) function
$$ (f\,\|\,\phi): ~\phi U\to\C$$
to each local chart $\phi:U\to\phi U.$
(We assume that this assignment respects restrictions to subcharts.)
When local coordinates are specified explicitly, we often write $f(z)$ for $(f\,\|\,\phi)(z).$

\ms The transformation law from one coordinate chart to another can be quite complicated for the fields that we will consider.
Several simpler cases have special names.

\ms A field $f$ is a \emph{differential} \index{differential} of degrees \index{degree} (or conformal dimensions) \index{conformal dimension} $(\lambda,\lambda_*)$ if for any two overlapping charts $\phi$ and $\widetilde\phi,$ we have
$$f = (h')^\lambda(\overline{h'})^{\lambda_*}\widetilde{f}\circ h,$$
where $f$ is the notation for $(f\,\|\,\phi),$ $\widetilde f$ for $(f\,\|\,\widetilde\phi),$ and $h$ is the transition map.
In particular, (0,0)-differentials are called \emph{scalars}. \index{scalar field}

\ms \emph{Schwarzian forms}, \index{form!Schwarzian} \emph{pre-Schwarzian forms}, \index{form!pre-Schwarzian} and \emph{pre-pre-Schwarzian forms} \index{form!pre-pre-Schwarzian} are fields with transformation laws \index{Schwarzian!form} \index{pre-Schwarzian!form}
$$f=(h')^2\widetilde{f}\circ h + \mu S_h,\qquad f=h'\widetilde{f}\circ h +\mu N_h, \qquad f=\widetilde{f}\circ h +\mu \log h',$$
respectively, where $\mu\in\C$ is called the order of the form, and
$$N_h =(\log h')',\qquad S_h = N_h' -{N_h^2}/2$$
are pre-Schwarzian and Schwarzian derivatives of $h.$ \index{pre-Schwarzian!derivative} \index{Schwarzian!derivative}
(In all cases we consider, forms are holomorphic.)

\begin{egs*}
\renewcommand{\theenumi}{\alph{enumi}}
{\setlength{\leftmargini}{2.0em}
\begin{enumerate}
\ms \item
Smooth $(-1,0)$-differentials $v$ can be identified with vector fields.
The local flow $z(t)$ of $v$ in a chart $\phi$ is given by the ordinary differential equation $\dot z=v(z).$
If we have another chart $\widetilde\phi,$ then the expression for the flow in this chart is $\ti z(t)=h(z(t)),$ and the vector field is
$\ti v(\ti z)=\dot{\ti z}=h'(z)v(z),$
so the transformation law is
\begin{equation} \label{eq: vector field}
v=\frac1{h'}\ti v\circ h.
\end{equation}
\item
If $f$ is a holomorphic scalar function on $M,$ then the derivatives $N_f$ and $S_f$ are computed in local coordinates,
e.g.,
$$(S_f\,\|\,\phi):=S_{(f\,\|\,\phi)}$$
are forms of order 1.
\ms \item The field $c,$ the logarithm of conformal radius is defined by the equation
\begin{equation} \label{eq: c3}
(c\,\|\,\phi)(z)= \lim_{\zeta\to z} \left[G(\zeta,z)+\log| \zeta-z|\right],
\end{equation}
where $G$ is the Green's function and $G(\zeta, z)$ of course means $G(\phi^{-1}\zeta,\phi^{-1}z)$ according to our convention.
It is easy to see that the transformation law is
\begin{equation*}
c=\tilde c\circ h-\log|h'|,
\end{equation*}
so $c$ is the real part of a pre-pre-Schwarzian form.
In particular,
$$(c\,\|\,\id_D)(z) = \log\left|\frac{w(z)-\overline{w(z)}}{w'(z)}\right|,$$
where $\id_D$ is the identity chart of $D$ and $w$ is a conformal map from $D$ onto the upper half-plane $\H,$ cf. \eqref{eq: log C}.
\ms \item
As a general rule, we define derivatives of fields by differentiating in local coordinates.
Thus the field $\pa c,$
$$(\pa c\,\|\,\phi):=\pa( c\,\|\,\phi)$$
is a pre-Schwarzian form of order $-\frac12.$
\ms \item
The conformal radius $C=e^c$ is a $\left(-\frac12, -\frac12\right)$-differential.
Indeed,
$$C = e^c = e^{\tilde c\circ h-\log|h'|} = (h'\overline{h'})^{-1/2}\widetilde C\circ h.$$
By Koebe, $(C\,\|\,\id_D)(z) \asymp \operatorname{dist}(z,\pa D).$
\item
Non-random fields of \emph{several variables} are defined similarly but one should keep in mind that we need to specify local coordinates for each variable (unless some of them coincide).
For example, the field $\pa_1\pa_2 G$ defined as
\begin{equation} \label{eq: d1d2G}
\pa_\zeta\pa_z G(\zeta,z) \qquad ({\rm in ~charts}~\phi, \psi)
\end{equation}
is a $(1,0)$-differential in both variables.
\end{enumerate}}
\end{egs*}

\ms \section{Conformal Fock space fields}\label{sec: conf F-field}

\ms\SS Let $\Phi$ denote the Gaussian free field on $M.$
(As a correlation functional, the Gaussian free field on a Riemann surface is well-defined as long as the Green's function exists.)
As in Section~\ref{sec: F-field}, we define basic Fock space fields $X_\alpha$ as formal Wick's products of the derivatives of $\Phi.$
A general conformal Fock space field is a linear combination of basic fields $X_\alpha,$
$$X=\sum_\alpha f_\alpha X_\alpha,$$
where the coefficients $f_\alpha $ are non-random conformal fields.

\ms We can define chaos decomposition, Wick's multiplication, and the differential operators $\pa,\bp$ in an obvious fashion so that the space of conformal Fock space fields will have the structure of a graded commutative differential algebra (with complex conjugation) over the ring of non-random fields.

\ms\SS We now want to interpret the values of conformal fields as chart dependent correlation functionals.
In particular we will explain the meaning of formulas like
$J = h'[\ti J\circ h],$
where $J, \ti J$ are expressions of $\pa \Phi$ in $2$ overlapping charts.

\ms The correlations
$$\E[X_1(p_1)\cdots X_n(p_n)],\qquad (p_j\in M~ {\rm are~ distinct}),$$
of conformal Fock space fields are non-random fields of several variables:
we define them by means of Wick's formula and differentiation rules like
$$\E[(\pa^\alpha\Phi)(z_1)(\pa^\beta\Phi)(z_2)]= 2\pa_1^\alpha\pa_2^\beta G(z_1,z_2).$$
In other words, we think of the derivatives of $\Phi$ as Gaussians, and we differentiate and Wick-multiply in local coordinates.

\ms As in Section~\ref{sec: F-field}, we think of ``strings"
$$\XX:=X_1(z_1)\cdots X_n(z_n)\equiv (X_1\,\|\,\phi_1)(z_1)\cdots(X_n\,\|\,\phi_n)(z_n) $$
as Fock space correlation functionals.
Note that $\XX$ specifies the choice of local charts.
Any such $\XX$ determines a linear map
$$ \XX: ~ \YY~\mapsto~ \E[\XX \YY]$$
on the space of $\YY$'s with nodes in $M\sm S_\XX.$
(The functionals $\YY$ also come with chart specifications and we define $S_\XX$ as a subset of $M.$)

\ms In particular, the value $X(p)$ of a conformal field $X$ at some point $p\in M$ is a coordinate dependent functional
$$(X(p)\,\|\,\phi)\equiv (X\,\|\,\phi)(\phi(p)).$$
Many formulas of conformal field theory (convergence, operator product expansions, transformation laws, etc.) are based on this interpretation.

\ms For example, we say that a field $X$ is a differential if its transformation law is
\begin{equation} \label{eq: diff}
X=(\widetilde X\circ h)(h')^\lambda(\overline{h'})^{\lambda_*}.
\end{equation}
Here
$$ X(\cdot):=(X\,\|\,\phi)(\cdot),\qquad \widetilde X(\cdot):=(X\,\|\,\widetilde \phi)(\cdot),$$
and the equation \eqref{eq: diff} means that for all $\YY,$
$$\E[X(z)\YY]=h'(z)^\lambda\overline{h'(z)}^{\lambda_*} \E[\widetilde X(h(z))\YY].$$
Equivalently, for all $\YY,$ the non-random field
$p\mapsto \E[X(p)\YY]$ is a differential in $p\in M.$
Moreover, it is enough to consider $\YY$'s of the form $\Phi(p_1)\odot \cdots \odot\Phi(p_n).$

\begin{egs*}
\renewcommand{\theenumi}{\alph{enumi}}
{\setlength{\leftmargini}{2.0em}
\begin{enumerate}
\ms \item $\Phi$ is a scalar field, i.e., a (0,0)-differential,
$J$ is a (1,0)-differential, and $J\odot J$ is a (2,0)-differential;
\ms \item The field (see Proposition \ref{T in S(1/12)})
$$T=-\frac12 J*J=-\frac12~J\odot J+\frac1{12}S $$
 is a Schwarzian form of order $\frac1{12}$;
\ms \item $e^{\odot \alpha\Phi}$ is a scalar field but $e^{*\alpha\Phi}$ is a differential of degrees
$(-{\alpha^2}/2,-{\alpha^2}/2);$
\ms \item $J*\bar J = J\bar J$ is a (1,1)-differential, see \eqref{eq: J*bar J}.
\end{enumerate}}
\end{egs*}

\ms\SS The operator product expansion of conformal fields (in particular, the OPE multiplication which we used in the examples above) is defined in terms of local charts.
For example, if $X$ is a holomorphic field, then there are conformal Fock space fields $C_n$ such that in every chart $\phi$ we have
$$X(\zeta)Y(z)=\sum(\zeta-z)^{n} C_n(z),\qquad (\zeta\to z), $$
where $$X(\zeta):=(X\,\|\,\phi)(\zeta),\quad Y(z):=(Y\,\|\,\phi)(z),\quad C_n(z):=(C_n\,\|\,\phi)(z).$$
It is crucial that we use the same asymptotic scale in all local charts, which results in non-trivial conformal structure of OPE coefficients.

\ms\SS We often consider the values of conformal fields at boundary points of $D$ (or ideal boundary points of a finite Riemann surface $M$).
By convention, we always model local coordinate charts on the half-plane $\H,$ i.e., we use \emph{standard} boundary charts
$$\phi:U\to\phi(U) \subset \H,\qquad \phi(\pa M \cap \bar U) \subset \R.$$
Note that the transition map between standard boundary charts extends by symmetry to a map which is analytic at the points in $\R.$
For example, the field $J=\pa\Phi$ is purely imaginary on $\pa M$ (in all standard boundary charts).

\ms \section{Conformal invariance} \label{sec: conf inv} \index{conformal invariance}

\SS A \emph{non-random} conformal field $f$ is invariant with respect to some conformal automorphism $\tau$ of $M$ if
$$\textrm{for~all~}\, \phi,\qquad (f\,\|\,\phi)=(f\,\|\,\phi\circ \tau^{-1}).$$
Note that in this equation we compare $f(p)$ with $f(\tau p).$
\begin{center}\begin{figure}[ht]
\includegraphics{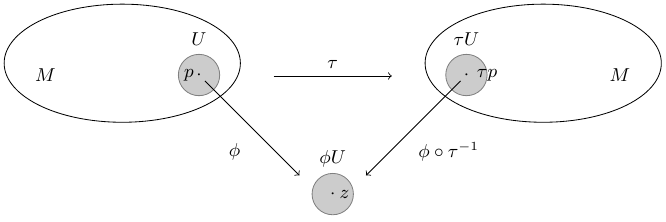}
\caption{Conformal invariance} \label{fig: conformal invariance}
\end{figure}\end{center}

For example, let $D$ be a domain in $\C$ and let us write $f(z)$ for $(f\,\|\,\id_D)(z).$
Then $f$ is a $\tau$-invariant $(\lambda,0)$-differential if
$$f(z) =f(\tau z)~\tau'(z)^\lambda,$$
and $f$ is a $\tau$-invariant Schwarzian form if
$$f(z) =f(\tau z)~\tau'(z)^2+\mu S_\tau(z).$$
This is because $\tau$ is the transition between the charts $\phi\circ \tau^{-1}$ and $\phi=\id_D.$

\ms The concept of conformal invariance extends to multi-variable non-random conformal fields.
For example,
$$f(z_1,z_2)=\frac{w'(z_1)w'(z_2)}{(w(z_1)-w(z_2))^2},$$
where $w: D\to\H$ is a conformal map, is a conformally invariant differential in both variables.

\ms\SS By definition, a \emph{random} conformal field (or a family of conformal fields) is $\tau$-\emph{invariant} if all correlations are invariant as non-random conformal fields.

\ms Clearly, the Gaussian free field is conformally invariant (i.e., invariant with respect to the full group $\Aut(M)$), and a family of conformally invariant fields is closed under differentiations, $\pa,\bp,$ and Wick's multiplication.
So all basic fields are conformally invariant.
It follows that a conformal Fock space field is $\tau$-invariant if and only if all its basic field coefficients are $\tau$-invariant.
It is also clear that the OPE coefficients of two conformally invariant fields are conformally invariant.

\ms Caution: it would be wrong to define conformal invariance by the equation
$$(X\,\|\,\phi)=(X\,\|\,\phi\circ \tau^{-1})$$
for correlation functionals. In fact, $X$ is $\tau$-invariant if and only if
$$\E\,[\,(X\,\|\,\phi)\,\Phi(p_1)\odot\cdots\odot\,\Phi(p_n)\,]\, =\E\,[\,(X\,\|\,\phi\circ \tau^{-1})\,\Phi(\tau p_1)\odot\cdots\odot\,\Phi(\tau p_n)\,].$$
(This is different from $\E\,[\,(X\,\|\,\phi)\,\YY\,]=\E\,[\,(X\,\|\,\phi\circ \tau^{-1})\,\YY\,].$)

\ms\SS The following simple but useful construction (see e.g., Lecture~\ref{ch: SLE theory}) depends on conformal invariance.
Suppose we have conformally equivalent Riemann surfaces $M$ and $\widetilde{M}.$
Given a conformally invariant field $X$ on $M,$
we define the field $\widetilde X$ on $\widetilde{M}$ as follows.
Let $f:M \to \widetilde{M}$ be a conformal map.
We write $p=f^{-1}(\widetilde{p}),$ $p_j=f^{-1}(\widetilde{p}_j),$ $\widetilde\phi$ for a fixed chart at $\widetilde{p},$ and $\phi = \widetilde\phi\circ f.$
We set
$$\E\,[\,(\widetilde X(\widetilde{p})\,\|\,\widetilde\phi)\,\widetilde\YY\,]=\E\,[\,(X(p)\,\|\,\phi)\,\YY\,],$$
where $\YY(p_1,\cdots,p_n) = \Phi(p_1)\odot\cdots\odot\Phi(p_n)$ and $\widetilde\YY(\widetilde{p}_1,\cdots,\widetilde{p}_n) = \Phi(\widetilde{p}_1)\odot\cdots\odot\Phi(\widetilde{p}_n).$
Clearly, $\widetilde X$ does not depend on the choice of $f.$

\ms \section{Lie derivatives} \label{sec: Lie}

Let $v$ be a non-random smooth vector field, i.e., a $(-1,0)$-differential, see \eqref{eq: vector field}, on a Riemann surface $M$; it determines a local flow
$$\psi_t:U\to M,\qquad \dot \psi_t(z)=v(\psi_t(z)),\qquad z\in U, \; |t|\ll1.$$
Suppose $v$ is holomorphic in some open set $U\subset M$ (so the flow is also holomorphic).
For a conformal Fock space field $X,$ we define the Lie derivative $\LL_vX$ in $U$ as follows.

\ms We first define the fields $X_t$ by the equation of correlation functionals
$$(X_t\,\|\,\phi)(z)=(X\,\|\,\phi\circ \psi_{-t})(z),\qquad z\in \phi U, \; |t|\ll1,$$
where $\phi$ is an arbitrary chart in $U.$

\ms For example, if $M$ is a domain in $\C$ with the identity chart, then the equation is
\begin{equation}
\label{eq: Xt4diff}
X_t(z)=(X(\psi_tz)\,\|\,\psi_{-t})=(\psi'_t(z))^\lambda~(\overline{\psi'_t(z)})^{\lambda_*}~ X(\psi_tz)
\end{equation}
for $(\lambda,\lambda_*)$-differentials, and
\begin{equation}
\label{eq: Xt4S-form}
X_t(z)=(\psi'_t(z))^2~X(\psi_tz)+\mu S_{\psi_t}(z)
\end{equation}
for Schwarzian forms.

\ms It is easy to see that if $X$ is a differential or a form, then $X_t$ is a differential or a form of the same type.

\ms We now define the \emph{Lie derivative} of $X$ \index{Lie derivative} by
$$\LL_vX=\frac d{dt}\Big|_{t=0}X_t.$$
As usual, this means that for every chart $\phi$ and every functional $\YY$ we have
$$\E\,[\,(\LL_vX\,\|\,\phi)\,\YY\,] = \frac{d}{dt}\Big|_{t=0}\E\,[\,(X_t\,\|\,\phi)\,\YY\,].$$

\ms This definition is very general -- the only assumption that we make is that $X$ depends smoothly on local coordinates, so the derivative exists.
We need higher smoothness when we consider commutations of Lie derivatives.
Smooth dependence on local coordinates can be defined as follows: $\E\,[\,(X\,\|\,h_\ve\circ\phi )\,\YY\,]$ is a smooth function of $\ve = (\ve_1,\cdots,\ve_n)$ for any $\ve$-perturbation $h_\ve\circ\phi$ of the chart $\phi,$
$$h_\ve(z)= z + \ve_1 v_1(z) + \cdots + \ve_n v_n(z).$$
In particular, the smooth dependence of $X$ on local charts implies that
\begin{equation} \label{eq: Lie}
\frac{d}{dt}\Big|_{t=0}\E\,[\,(X\,\|\,\phi\circ f_t^{-1})\,\YY\,]=\E\,[\,(\LL_vX\,\|\,\phi)\,\YY\,]
\end{equation}
for any  flow $f_t(z) = z + tv(z,t) + o(t)$ with the time-dependent vector field $v(z,t)$ $(v = v(z,0)).$
It is easy to see that if $X$ and $Y$ depend smoothly on charts, then so does $X*_nY.$

\ms Lie derivative of a differential is a differential but Lie derivative of a Schwarzian form is a quadratic differential.

\ss \begin{prop} \label{LvX}
If $X$ is a differential, then
\begin{equation} \label{eq: Lie4diff}
\LL_vX=\left(v\pa+\bar v\bar\pa+\lambda v'+\lambda_*\overline{v'}\right)X;
\end{equation}
if $X$ is a pre-Schwarzian form of order $\mu,$ then
\begin{equation} \label{eq: Lie4PS-form}
\LL_vX=\left(v\pa+v'\right)X +\mu v'';
\end{equation}
if $X$ is a Schwarzian form of order $\mu,$ then
\begin{equation} \label{eq: Lie4S-form}
\LL_vX=\left(v\pa+2v'\right)X +\mu v'''.
\end{equation}
\end{prop}

\begin{proof} Without loss of generality we can consider the planar case and the identity chart.
Differentiate \eqref{eq: Xt4diff}, \eqref{eq: Xt4S-form} and use
$\psi'_0=1,\quad \dot \psi_0=v,\quad\dot \psi_0'=v',$
$(\dot N_{\psi})_0 = v'',$ and $(\dot S_{\psi})_0 = v'''.$
\end{proof}

\ms It turns out that the converse is also true.
For example,
\begin{prop} \label{LvX4diff}
Suppose the equation
$$\LL_vX=\left(v\pa+\bar v\bar\pa+\lambda v'+\lambda_*\overline{v'}\right)X$$
holds in $D_{\rm hol}(v)$ for every vector field $v.$ Then $X$ is a differential.
\end{prop}

\ms\textbf{Notation.} If $v$ is a smooth vector field in $D,$ then we denote by $D_{\rm hol}(v)$ the maximal open set where $v$ is holomorphic.

\ms\begin{proof}
In a fixed chart $\phi$ we have
\begin{align*}
X_t(z)&=X+tL_vX+o(t)=X+t\left(v\pa+\bar v\bar\pa+\lambda v'+\lambda_*\overline{v'}\right)X+o(t)\\
&=(1+t\lambda v'+o(t))(1+t\lambda_*\overline{v'}+o(t))(X+(\psi_t-z)\pa X+(\bar \psi_t-\bar z)\bp X+o(t))\\
&=(\psi'_t(z))^\lambda~(\overline{\psi'_t(z)})^{\lambda_*}~ X(\psi_tz)+o(t).
\end{align*}
On the other hand, by definition $X_t=(X\,\|\,\phi\circ\psi_{-t})$ we have
$$(X\,\|\,\phi\circ\psi_{-t})(z)=(\psi'_t(z))^\lambda~(\overline{\psi'_t(z)})^{\lambda_*}~(X\,\|\,\phi)(\psi_tz)+o(t),$$
which is the infinitesimal version of the transformation law for a differential.
\end{proof}

\ms The next statement follows from the elementary properties of Lie derivatives that we record in the next section.

\ss \begin{prop}
If $X$ is a conformal Fock space field, then $\LL_vX$ is also a (local) conformal Fock space field.
\end{prop}

\ms \section{Properties of Lie derivatives}\label{sec: basic Lv}
\SS \emph{Basic properties}:
\renewcommand{\theenumi}{\alph{enumi}}
{\setlength{\leftmargini}{2.0em}
\begin{enumerate}
\item\ms $\LL_v$ is an $\R$-linear operator on Fock space fields;
\item\ms $\LL_v (\bar X)=\overline {(\LL_vX)}$;
\item\ms $\E[\LL_vX]=\LL_v(\E[X]);$
\item\ms Leibniz's rule applies to Wick's products;
\item\ms $\LL_v(\pa X)=\pa(\LL_v X)$ and $\LL_v(\bar\pa X)=\bar\pa(\LL_v X).$
\end{enumerate}}

\ms Let us show that Leibniz's rule also applies to OPE products.

\begin{prop} \label{Leibnitz4OPE}
If $X$ is a holomorphic Fock space field, then
$$\LL_v(X*_n Y)=(\LL_v X)*_n Y+X*_n (\LL_v Y).$$
\end{prop}

\begin{proof}
Without loss of generality we can consider the planar case and the identity chart.
Suppose
$$X(\zeta)Y(z)=\sum(\zeta-z)^nC_n(z),\qquad \zeta\to z.$$
Then
\begin{align*}
X_t(\zeta)Y_t(z)&=(X\,\|\,\psi_{-t})(\zeta)(Y\,\|\,\psi_{-t})(z)\\&=\sum(\zeta-z)^n(C_n\,\|\,\psi_{-t})(z)=\sum(\zeta-z)^n ~(C_n)_t(z),
\end{align*}
so
$X_t*_n Y_t= (C_n)_t.$
We now take the time derivative at $t=0.$
\end{proof}

\ms\SS Recall that the Lie derivative of a vector field is defined in the smooth category as follows:
$$\LL_{v_1}v_2 \equiv [v_1,v_2] = \frac{\pa^2}{\pa s\pa t}\Big|_{s=t=0}~(\chi_s\circ\psi_t-\psi_t\circ\chi_s),$$
where $\psi_t$ is the flow of $v_1$ and $\chi_s$ is the flow of $v_2$; the local flow of $[v_1,v_2]$ is
$\chi_{-\sqrt t}\circ\psi_{-\sqrt t}\circ\chi_{\sqrt t}\circ\psi_{\sqrt t}.$
If both vector fields are holomorphic, then
\begin{equation} \label{eq: Lv1v2}
\LL_{v_1}v_2 = [v_1,v_2]=v_1v_2'-v_1'v_2 \quad \textrm{ and } \quad \LL_{v_1}\bar v_2 = \overline{[v_1,v_2]},
\end{equation}
which is of course a special case of \eqref{eq: Lie4diff}.

\begin{prop} \label{Lie []}
If $X$ is a conformal Fock space field, then
\begin{equation} \label{eq: Lie []}
\LL_{v_1}\LL_{v_2}X-\LL_{v_2}\LL_{v_1}X=\LL_{[v_1,v_2]}X
\end{equation}
in the region where both vector fields are holomorphic.
\end{prop}

\begin{proof} From the definition of Lie derivative we see that the left-hand side is equal to
$$\frac{\pa^2}{\pa s\pa t}\Big|_{s=t=0}~\left[(X\,\|\,\chi_{-s}\circ\psi_{-t})-(X\,\|\,\psi_{-t}\circ\chi_{-s})\right].$$
Expanding the flows up to second order (we use $\dot\psi=v_1,$ $\ddot \psi=v_1'v_1,$ etc.), we get
$$\chi_{-s}\circ\psi_{-t}=\id-tv_1-sv_2+\frac{t^2}2v_1v_1'+\frac{s^2}2v_2v_2'+stv_1v_2'+\cdots,$$
$$\psi_{-t}\circ\chi_{-s}=\id-tv_1-sv_2+\frac{t^2}2v_1v_1'+\frac{s^2}2v_2v_2'+stv_1'v_2+\cdots,$$
and the statement easily follows if we assume sufficient smoothness with respect to local coordinates.
\end{proof}

\ss\SS The concept of Lie derivative extends to conformal fields of several variables.
For example, for
$$X(p_1, \cdots, p_n)=X_1(p_1)\cdots X_n(p_n), \qquad (p_j\in M \textrm{~are distinct}),$$
we fix coordinate charts $\phi_j$ at $p_j,$ and assuming that $v$ is holomorphic at the nodes, we define
$$\LL_vX(z_1, \cdots, z_n)=\frac d{dt}\Big|_{t=0}~ \left[(X_1\,\|\,\phi_1\circ\psi_{-t})(z_1)~\cdots~(X_n \,\|\,\phi_n\circ\psi_{-t})(z_n)\right].$$

\begin{prop} \label{Leibnitz4strings}
Leibniz's rule holds for tensor products:
$$\LL_v\left[X_1(p_1) X_2(p_2)\right]=\left[\LL_vX_1(p_1)\right] X_2(p_2)+\left[\LL_vX_2(p_2)\right] X_1(p_1).$$
\end{prop}

\ss For example, if $X$ is a tensor product of differentials, then
$$\LL_v X=\sum_j\left[v(p_j)\pa_j+\bar v(p_j)\bar\pa_j+\lambda_jv'(p_j)+\lambda_{*j}\overline{v'}(p_j)\right]\,X.$$

\ss\SS As we mentioned, $\LL_v$ depends $\R$-linearly on $v.$
It is convenient to separate the $\C$-linear \index{Lie derivative!$\C$-linear part} and anti-linear parts \index{Lie derivative!anti-linear part} of the Lie derivative.
Denote
\begin{equation}\label{eq: Lpm}
2\LL_v^+= \LL_v-i\LL_{iv},\qquad 2\LL_v^-= \LL_v+i\LL_{iv},
\end{equation}
so that
$$\LL_v=\LL_v^++ {\LL_v^-}$$
and
$\LL_v^-=\overline{\LL_v^+}$ in the following sense: $\bar \LL_v X=\overline{\LL_v \bar X}.$

\ms For example, if $X$ is a tensor product of differentials, then
\begin{equation} \label{eq: Lplus4diff}
\LL_v^+ X=\sum_j\left[v(p_j)\pa_j+\lambda_jv'(p_j)\right]\,X,
\end{equation}
and $\LL^+_vX=\LL_vX$ in the case of forms (see \eqref{eq: Lie4PS-form}, \eqref{eq: Lie4S-form}, and \eqref{eq: Lpm}). 

\ms It is easy to justify the corresponding
Leibniz's rule for $\LL^+_v$ and also to verify the identity
\begin{equation} \label{eq: Leibnitz4+}
\LL^+_{v_1}\LL^+_{v_2}-\LL^+_{v_2}\LL^+_{v_1}=\LL^+_{[v_1,v_2]}.
\end{equation}
For example, using \eqref{eq: Lie []} and \eqref{eq: Lpm} we have
$$\LL^+_{v_1}\LL^+_{v_2}-\LL^+_{v_2}\LL^+_{v_1} = \frac14\left(\LL_{[v_1,v_2]}-\LL_{[iv_1,iv_2]}-i\LL_{[iv_1,v_2]}- i\LL_{[v_1,iv_2]}\right)=\LL^+_{[v_1,v_2]}.$$

\renewcommand\chaptername{Lecture}
\chapter{Stress tensor and Ward's identities} \label{ch: Ward}

We define the stress tensor for a family $\FF$ of conformal Fock space fields as a pair of quadratic differentials which represent the Lie derivative operators in application to the fields in $\FF$ (and their tensor products).
The corresponding formulas are known as Ward's identities.
More precisely, for every local holomorphic vector field $v,$ we use the quadratic differentials to construct a functional (``generalized random variable") $W(v)$ such that the action of the operator $\LL_v$ on any string of fields in $\FF$ is equivalent, in correlations with arbitrary Fock space fields, to the multiplication by $W(v).$
Alternately, the stress tensor $W$ can be defined as the correspondence $v\mapsto W(v).$

\ms The existence of $W$ is not at all obvious, and in fact it is a very special property of some \emph{particular} families of Fock space fields.
In this lecture we mostly discuss various forms of Ward's identities.
We will comment on the nature of existence of stress tensor in the appendix to this lecture.

\ms \section{Residue operators} \label{sec: residue op} \index{residue operator}

Let $A$ be a Fock space \emph{holomorphic} quadratic differential in $D,$ let $p\in D,$ and let $v$ be a non-random holomorphic vector field defined in some neighborhood of $p.$
Then for every Fock space field $X$ we define the correlation functional \index{residue operator}
$$\frac1{2\pi i}\oint_{(z)} vA~X(z) \qquad (\textrm{in ~a~ given~ chart}\; \phi,\;\phi(p)=z)$$
as a map
$$\XX~\mapsto ~ \lim_{\ve\to 0}~\frac1{2\pi i}\int_{|\zeta-z|=\ve} v(\zeta)~\E\left[A(\zeta)~X(z)\XX\right]\,d\zeta,$$
where $\XX$ is any Fock space correlation functional with nodes in $D\sm\{p\}.$

\ms This functional is of course just the residue term $(vA)*_{-1}X$ in the operator product expansion of $vA$ and $X,$ see Section \ref{sec: OPE coeffs}, and therefore by Proposition~\ref{OPE coeffs} it can be expressed as the value of some Fock space field.
We can view the map
$$ X(z)\mapsto \frac1{2\pi i}\oint_{(z)} vA~X(z)$$
as an operator on correlation functionals (represented by the values of Fock space fields); we denote this operator by
$$A_v(z)\equiv\frac1{2\pi i}\oint_{(z)} vA.$$
Varying $z,$ we can also think of $A_v$ as an operator on Fock space fields:
$$(A_vX)(z)=A_v(z)X(z).$$

\begin{prop} \label{ALALLALA}
We have 
$$A_{v_1}\LL^+_{v_2}-A_{v_2}\LL^+_{v_1}=\LL^+_{v_2}A_{v_1}-\LL^+_{v_1}A_{v_2}.$$
\end{prop}

\begin{proof} By Leibniz's rule (Proposition \ref{Leibnitz4OPE}) we have
\begin{align*}
A_{v_1}\LL^+_{v_2}X&=(v_1A)*_{-1}(\LL^+_{v_2}X)=\LL^+_{v_2}[(v_1A)*_{-1}X]-[\LL^+_{v_2}(v_1A)]*_{-1}X\\
&=\LL^+_{v_2}A_{v_1}X-[\LL^+_{v_2}(v_1A)]*_{-1}X.
\end{align*}
Similarly,
$$A_{v_2}\LL^+_{v_1}X=\LL^+_{v_1}A_{v_2}-[\LL^+_{v_1}(v_2A)]*_{-1}X.$$
Since $v_1A$ and $v_2A$ are (1,0)-differentials, by \eqref{eq: Lplus4diff} we have
$$\LL^+_{v_2}(v_1A)=v_2\pa(v_1A)+v_2'v_1A=v_2v_1'A+v_2'v_1A+v_1v_2\pa A=\LL^+_{v_1}(v_2A),$$
which proves the statement.
\end{proof}

\ms For an anti-holomorphic quadratic differential $A^-$ we define
$$A_v^-(z)=-\frac1{2\pi i}\oint_{(z)} \bar vA^-.$$
This operator is anti-linear in $v,$ and if $A^-=\bar A,$ then $A_v^-=\overline{A_v}.$

\ms \section{Stress tensor} \label{sec: W} \index{stress tensor}

Let $X$ be a Fock space field in $D.$
By definition, a pair of quadratic differentials $$W=(A^+,A^-)$$ is a \emph{stress tensor} for $X$ if $A^+$ is holomorphic, $A^-$ anti-holomorphic, and the following equation (the ``residue form of Ward's identity")
\begin{equation} \label{eq: LvWv}
\LL_vX= A^+_vX+ A^-_vX
\end{equation}
holds in $D_{\rm hol}(v)$ for all non-random local vector fields $v.$
(Recall that we write $D_{\rm hol}(v)$ for the maximal open set where $v$ is holomorphic.)
Thus we require that the equation
$$\LL_vX(z)=\frac1{2\pi i}\oint_{(z)} vA^+~X(z)-\frac1{2\pi i}\oint_{(z)} \bar vA^-~X(z)$$
holds in all charts and for all vector fields $v$ holomorphic in a neighborhood of $z.$

\ss The differentials $A^\pm$ (if they exist) are not uniquely determined by the equation \eqref{eq: LvWv}.
Moreover, we can add (anti-)holomorphic non-random fields -- they will not change the residue operators.
For example, the Virasoro fields determine the same residue operators as the differentials $A^\pm$ do for local holomorphic vector fields.
We will discuss this in the next lecture.

\ms \textbf{Notation.} $\FF(W) \equiv \FF(A^+,A^-)$ is the \emph{linear} space of all Fock space fields $X$ such that $W$ is a stress tensor for $X.$
Clearly, this space contains the scalar field $I\,(I(z)\equiv1).$
If $\FF(W)$ is closed under complex conjugation, then we can choose
$$A^+=A,\qquad A^-=\bar A;$$
and
\begin{equation} \label{eq: Ward4CFT}
X\in \FF(W) \quad\textrm{if and only if}\quad
\LL^+_vX=A_vX,\; \LL^+_v\bar X= A_v\bar X.
\end{equation}

In what follows, we will only consider the case $W = (A,\bar A).$
There is no difficulty in extending results to the anti-symmetric ($A^-\ne\overline{A^+}$) case.

\begin{prop} \label{commutation4A}
If $X\in\FF(A,\bar A),$ then
$$[A_{v_1}, A_{v_2}]X=-A_{[v_1,v_2]}X.$$
\end{prop}

\begin{proof} Since $A_{v_j}X=\LL^+_{v_j}X,$ it follows from Proposition \ref{ALALLALA} and \eqref{eq: Leibnitz4+} that
$$[A_{v_1}, A_{v_2}]X=[\LL^+_{v_2}, \LL^+_{v_1}]X=\LL^+_{[v_2,v_1]}X=-A_{[v_1,v_2]}X.$$
\end{proof}

\ms \section{Ward's OPEs} \label{sec: Ward's OPEs}

We can restate the definition of stress tensor in terms of the singular part of the operator product expansion.

\ms For a given chart $\phi: U\to\phi U$ and $\zeta\in \C,$ let us denote by $v_\zeta$ the (local) vector field defined by the equation
$$[v_\zeta\,\|\,\phi](\eta)=\frac1{\zeta-\eta}.$$
(This vector field depends on $\phi.$)
Then we have
\begin{equation}\label{eq: Sing}
\Sing_{\zeta\to z}[A(\zeta)X(z)]= \frac1{2\pi i}\oint_{(z)} v_\zeta A~X(z), \qquad z\in\phi U,
\end{equation}
where the left-hand side means the singular part of the operator product expansion in chart $\phi.$
Indeed, if
$$A(\eta)X(z)\sim \sum_{j\le-1}{C_j(z)}{(\eta-z)^j}, \qquad (\eta\to z),$$
then using
$$\frac1{2\pi i}\oint_{(z)}{(\eta-z)^j}~\frac{d\eta}{\zeta-\eta}={(\zeta-z)^j},\qquad (j\le-1),$$
we derive
\begin{equation} \label{eq: sing OPE}
\frac1{2\pi i}\oint_{(z)}A(\eta)X(z)~\frac{d\eta}{\zeta-\eta}=\sum_{j\le -1}{C_j(z)}{(\zeta-z)^j}.
\end{equation}

\begin{prop} \label{Ward's OPEs} \index{Ward's!OPE}
$X\in\FF(A,\bar A)$ if and only if the identities (``Ward's OPEs")
$$\Sing_{\zeta\to z}[A(\zeta)X(z)]= (\LL_{v_\zeta}^+X)(z), \quad \Sing_{\zeta\to z}[A(\zeta)\bar X(z)]= (\LL_{v_\zeta}^+\bar X)(z)$$
hold in every local chart $\phi.$
\end{prop}

\begin{proof} If $X\in\FF(A,\bar A),$ then
$$\Sing_{\zeta\to z}[A(\zeta)X(z)]= \frac1{2\pi i}\oint_{(z)} v_\zeta A~X(z)= (\LL_{v_\zeta}^+X)(z)$$
by \eqref{eq: Sing} and the definition of stress tensor.
In the opposite direction, we need to show that $$\frac1{2\pi i}\oint_{(z)} v_\zeta A~X(z)= (\LL_{v_\zeta}^+X)(z)$$ implies
 $$\frac1{2\pi i}\oint_{(z)} v A~X(z)= (\LL_{v}^+X)(z)$$
for all vector fields $v$ holomorphic near $z.$
Let us write $f$ for $(v\,\|\,\phi).$
By Cauchy,
$$v=\frac1{2\pi i}\int f(\zeta) v_\zeta~d\zeta$$
(integration is over some simple curve surrounding $z$), and since $\LL^+_v$ is $\C$-linear with respect to $v,$ we have
\begin{align*}\LL^+_v X(z)&=\frac1{2\pi i}\int f(\zeta) \LL^+_{v_\zeta}X(z)~d\zeta\\&=\frac1{2\pi i}\int f(\zeta)d\zeta ~ \frac1{2\pi i}\oint_{(z)} v_\zeta A~X(z) = \frac1{2\pi i}\oint_{(z)} v A~X(z).\end{align*}
\end{proof}

In the case of differentials or forms, it is enough to verify Ward's OPEs in just one chart, e.g., in the half-plane uniformization.
This is clear from the corresponding transformation laws.

\ms \begin{cor} \label{OPE4diff}
Let $X$ be a $(\lambda, \lambda_*)$-differential. Then $X\in \FF(A,\bar A)$ if and only if the following operator product expansions hold in every/some chart:
\begin{equation} \label{eq: OPE4diff}
A(\zeta)X(z)\sim\frac {\lambda X(z)}{(\zeta-z)^2}+\frac {\pa X(z)}{\zeta-z},\quad A(\zeta)\bar X(z)\sim \frac {\bar \lambda_*\bar X(z)}{(\zeta-z)^2}+\frac {\pa \bar X(z)}{\zeta-z}.
\end{equation}
\end{cor}

\begin{cor} \label{OPE4form}
Let $X$ be a form of order $\mu.$
Then $X\in \FF(A,\bar A)$ if and only if the following operator product expansion holds in every/some chart:
\begin{align*}
A(\zeta)X(z)&\sim \frac{\mu}{(\zeta-z)^2}+\frac {\pa X(z)}{\zeta-z} &\textrm{ for a pre-pre-Schwarzian form }X;\\
A(\zeta)X(z)&\sim \frac{2\mu}{(\zeta-z)^3}+\frac {X(z)}{(\zeta-z)^2}+\frac {\pa X(z)}{\zeta-z} &\textrm{ for a pre-Schwarzian form }X;\\
A(\zeta)X(z)&\sim \frac{6\mu}{(\zeta-z)^4}+ \frac {2X(z)}{(\zeta-z)^2}+\frac {\pa X(z)}{\zeta-z} &\textrm{ for a Schwarzian form }X.
\end{align*}
\end{cor}

By Proposition \ref{LvX4diff}, we also have the following:
\begin{cor}
Suppose $X\in \FF(A,\bar A).$ Then $X$ is a differential if and only if the operator product expansions \eqref{eq: OPE4diff} hold for $X.$
\end{cor}

\ms \section{Stress tensor of Gaussian free field} \label{sec: W of GFF} \index{stress tensor! of Gaussian free field}

Let us return to Proposition \ref{OPE4T}, where we stated some operator product expansions involving $T = -\frac12 J*J;$ as usual $J = \pa\Phi$ and $\Phi$ is the Gaussian free field.
Denote
$$A=-\frac12 J\odot J.$$
Then $A$ is a holomorphic quadratic differential and $A$ coincides with $T$ in the upper half-plane uniformization.
The first relation~\eqref{item: OPE(T,Phi)} in Proposition~\ref{OPE4T} can be written as
$$A(\zeta)\Phi(z)\sim \frac{\pa\Phi(z)}{\zeta-z}.$$
Applying Corollary~\ref{OPE4diff} we conclude:

\begin{prop} \label{W of GFF} We have 
$\Phi\in \FF(A,\bar A).$
\end{prop}

\ss The other three relations in Proposition~\ref{OPE4T} imply that $W=(A,\bar A)$ is a stress tensor also for the fields $J,$ $T,$ and $\VV^\alpha.$
Indeed, as we mentioned, it is sufficient to check Ward's OPEs in just one chart, and in the case of half-plane uniformization, this is what our relations give.
Note that we have arrived to this conclusion as a result of (rather lengthy) Wick's calculus computation.
There is a much easier way -- the proof of Proposition \ref{OPE4T} is immediate without any computation from Proposition \ref{W of GFF} and the following fact.

\ss \begin{prop}\label{*n in CFT}
\renewcommand{\theenumi}{\alph{enumi}}
{\setlength{\leftmargini}{2.0em}
\begin{enumerate}
\item If $X\in\FF(W),$ then $\pa X\in\FF(W)$;
\ms \item
all OPE coefficients of fields in $\FF(W)$ belong to $\FF(W).$
\end{enumerate}}
\end{prop}

The first statement is of course a simple special case of the second one.
(Recall that the non-random field $I(z)\equiv1$ is in $\FF(W)$ and $\pa X = X *_1 I.$)
We will explain the proof of the second statement in the next section.
There is a short algebraic argument in the case of \emph{holomorphic} fields $X$ and $Y.$
In this case,
$$\LL_v^+X = (vA)*_{-1}X,\qquad \LL_v^+Y = (vA)*_{-1}Y,$$
and we need to check that
$$\LL_v^+(X*_nY) = (vA)*_{-1}(X*_nY).$$
By Leibniz's rule, the left-hand side is
$$(\LL_v^+X)*_nY+X*_n(\LL_v^+Y),$$
while
\begin{align} \label{eq: Leibnitz4v1}
(vA)*_{-1}(X*_nY)&=[(vA)*_{-1}X]*_nY + X*_n[(vA)*_{-1}Y]\\
&=(\LL_v^+X)*_nY+X*_n(\LL_v^+Y), \nonumber
\end{align}
see \eqref{eq: Leibnitz4v2} below.
\hfill\qed

\ms Proposition~\ref{*n in CFT} allows us to construct infinitely many fields in the family $\FF(W).$
On the other hand, the field $A=-\frac12 J\odot J$ itself is not in $\FF(A,\bar A)$ because otherwise it would have the operator product expansion~\eqref{eq: OPE4diff} (as a differential).
But, by Wick's calculus, we easily verify
$$\E[A(\zeta)A(z)] = \frac{1/2}{(\zeta-z)^4}\qquad(\textrm{in }\H).$$

\ms Further examples of fields which have a stress tensor can be obtained by various modifications of the Gaussian free field, see Lecture~\ref{ch: modifications}.
The simplest modification is the following.
\begin{eg*}
Let $u$ be a real-valued harmonic function in $D.$
Define
$$\widehat\Phi= \Phi+u,$$
where $\Phi$ is the Gaussian free field and denote
$$\widehat A =-\frac12 J\odot J- (\pa u) J.$$
Then $(\widehat A,\overline{\widehat A})$ is a stress tensor for $\widehat\Phi.$

\ms If we take $u$ complex-valued, we will get an asymmetric stress tensor $(\widehat A^+, \widehat A^-)$ for $\widehat\Phi,$ where $\widehat A^+=\widehat A$ is as above, and
$$\widehat A^-=-\frac12 \bar J\odot \bar J- (\bp u) \bar J.$$
\end{eg*}

\ms \section{Ward's identities}\label{sec: Ward identity}

Let $W=(A,\bar A)$ be a stress tensor for some family of Fock space fields.
We will assume that $A$ is continuous up to the (ideal) boundary in the sense that all correlations of $A(\cdot)$ with Fock space fields extend to $\pa D$ continuously; we understand continuity on the boundary in terms of standard boundary charts.
See the end of Section~\ref{sec: conf F-field}.
For a smooth vector field $v$ in $D$ continuous up to the boundary, we define
\begin{equation} \label{eq: Ward}
W(v)=2\Re ~W^+(v),\qquad W^+(v)=\frac 1{2\pi i}\int_{\pa D}vA-\frac1\pi\int_D (\bp v)A.
\end{equation}

\ms Since $vA$ is a linear form, and $(\bp v)A$ is a (1,1)-differential, the integrals are coordinate independent, and by the continuity assumption, their correlations with Fock space functionals $\XX$ are well-defined provided that $S_\XX\subset D_{\rm hol}(v).$
(Recall that we write $S_\XX$ for the set of all nodes of $\XX$ and $D_{\rm hol}(v)$ for the maximal open set where $v$ is holomorphic.)

\ms The application of ``random variables" $W(v)$ is based on Green's formula
$$2i\iint_{D}\bp g=\int_{\pa D} g.$$
For example, since $(\bp v)A=\bp(vA)$ in $D,$ we have
$$\E W^+(v)=0.$$
By Green's formula we can write symbolically
$$W^+(v)=\frac1\pi\int_D v(\bp A);$$
however, to interpret this integral as a correlation functional we need to integrate by parts and therefore use the definition~\eqref{eq: Ward}.

\ms We can extend the definition of Ward's ``random variables" to the case of \emph{local} vector fields.
Namely, for an open set $U\subset D$ we denote
$$ W^+(v;U)=\frac 1{2\pi i}\int_{\pa U}vA-\frac1\pi\int_U (\bp v)A,$$
so that
$$W^+(v)=W^+(v;D),$$
and (with a usual interpretation)
$$A_v(z)\equiv\frac1{2\pi i}\oint_{(z)}vA=\lim_{\ve\to 0} W^+(v; B(z,\ve)),\qquad (z\in D_{\rm hol}(v)).$$
Green's formula shows that if $U_1\subset U_2$ and if $\XX$ has no nodes in the closure of $U_2\sm U_1,$ then
$$\E\,[\,W(v; U_1)\XX\,]=\E\,[\,W(v; U_2)\XX\,].$$
In particular, in the computation of $\E[W(v)\XX]$ we can replace $D$ by the union of small discs around the nodes of $\XX.$

\ss \begin{prop} \label{Ward identity}
Suppose $\{X_j\}\subset \FF(W)$ and $\{z_j\}\subset U \cap D_{\rm hol}(v).$
Then
\begin{equation} \label{eq: Ward identity}\index{Ward's!identity}
\E\, \YY\,\LL_v\left[X_1(z_1)\cdots X_n(z_n) \right ]~=~\E\,W(v;U)X_1(z_1)\cdots X_n(z_n) \YY
\end{equation}
for all correlation functionals $\YY$ with nodes in $D\sm \bar U.$
\end{prop}

\ss \begin{proof}
As mentioned, we can replace $U$ with the union of small discs $U_j$ around $z_j$'s.
Clearly, $W(v;U)=\sum W(v;U_j).$
Let us also use a partition of unity to represent $v=\sum v_j,$ where $v_j=v$ in $U_j$ and $v_j=0$ in other discs.
Thus the statement reduces to the case of a single node, where the formula is just the definition \eqref{eq: LvWv} of a stress tensor.
\end{proof}

\ms We emphasize that Ward's identities \eqref{eq: Ward identity} hold for any choice of local coordinates at the nodes $z_j.$
Their meaning is the following: we can represent the action of the Lie derivative operator $\LL_v$ by the insertion of the ``random variable" $W(v)$ into correlation functions, and this works collectively for all fields in the family $\FF(W).$

\ms The last proof gives the following restatement of the definition of a stress tensor in terms of Ward's identities (cf. Appendix~\ref{appx: Gibbs}).

\begin{prop} \label{Ward's identities}
$W=(A,\bar A)$ is a stress tensor for $X$ if and only if the following equation holds for all vector fields $v$ with compact supports, for all points $z\in D_{\rm hol}(v),$ and for all Fock space functionals $\ZZ$ with nodes outside $\supp(v)$:
$$\E\,\LL_vX(z)\ZZ~=~\E\,W (v)X(z)\ZZ. $$
\end{prop}

\ms We can use this restatement to derive:

\begin{proof}[Proof of Proposition~\ref{*n in CFT}]
The argument works for all types of operator product expansions, but for simplicity of notation we assume that $X,Y$ are holomorphic, so we have
$$X(\zeta)Y(z)=\sum (\zeta-z)^nC_n(z).$$
We want to show that
$$\E\, \LL_vC_n(z)\ZZ~=~\E\, W(v)C_n(z)\ZZ. $$
As in the proof of Proposition \ref{Leibnitz4OPE}, we have
$$\E\,X_t(\zeta)Y_t(z)\ZZ=\sum (\zeta-z)^n\E\,[C_n]_t(z)\ZZ.$$
Taking the time derivative at $t=0$ we get
\begin{align*}
\frac{d}{dt}\Big|_{t=0}\E\,X_t(\zeta)Y_t(z)\ZZ &=\E\,\LL_v[X(\zeta)Y(z)]\ZZ=\E\,W(v)X(\zeta)Y(z)\ZZ\\
&=\sum (\zeta-z)^n\E\,W(v)C_n(z)\ZZ,
\end{align*}
and
$$\frac{d}{dt}\Big|_{t=0}\sum (\zeta-z)^n\E\,[C_n]_t(z)\ZZ =\sum (\zeta-z)^n\E\,\LL_vC_n(z)\ZZ.$$
\end{proof}

\ms \section{Meromorphic vector fields} \label{sec: meromorphic}

Let $v$ be a meromorphic vector field in $D$ continuous up to the boundary, and let $\{p_j\}$ be the poles of $v.$
We define
$$W(v)=\lim_{\ve\to0} W(v; U_\ve),$$
where $U_\ve=D\sm\bigcup B(p_j,\ve).$
(We can use any fixed local coordinates at the poles.)
Somewhat symbolically, we have
$$W^+(v)=\frac1{2\pi i}\int_{\pa D} vA-\sum_j\frac1{2\pi i}\oint_{(p_j)} vA,$$
and also
\begin{equation} \label{eq: Wplus}
W^+(v)=\frac1{2\pi i}\int_{\pa D} vA-\frac1{\pi }\int_{D} (\bp v)A
\end{equation}
(as in the case of smooth vector fields) with the interpretation of $\bp v$ in the last integral in the sense of distributions.

\ms Our goal now will be to express the differential $A$ in terms of Ward's functionals $W(v)$ with meromorphic $v$'s.
We will only consider the case where $A$ is continuous and \emph{real} on the boundary (in standard boundary charts); this will allow us to extend $A$ to the double of $D$ accordingly.
We will do our computation in the half-plane $\H$ and use the global identity chart in $\C;$ note that $\widehat \C$ is the double of $\H.$
In the next section we combine the obtained representation of $A$ with Ward's identities \eqref{eq: Ward identity} and derive some useful equations for correlations involving the stress tensor.

\ms \begin{prop} \label{represent A}
Let $A$ be a holomorphic quadratic differential in $\H,$ and $W=(A,\bar A).$
Suppose $A$ is continuous and real on the boundary (including $\infty$).
Then
\begin{equation} \label{eq: A}
(A\,\|\,\id)(\zeta)=W^+\left(v_\zeta\right)+\overline{W^+\left(v_{\bar \zeta}\right)},
\end{equation}
where we use the notation
$$(v_\zeta\,\|\,\id)(z)=\frac1{\zeta-z},\qquad (\zeta\in\C).$$
\end{prop}

We understand the equation~\eqref{eq: A} in the sense of correlations with Fock space correlation functionals with nodes in $\widehat\H\sm \{\zeta\}.$
Note that $\E A\equiv 0$ by assumption: in the identity chart of $\H,$ $\E A$ is a holomorphic function vanishing at infinity.

\ms \begin{proof}
Let us start with a general observation which works for arbitrary Riemann surfaces.
If $v$ is a meromorphic vector field in $\widehat\C$ without poles on $\R\cup\{\infty\}$ such that the
reflected vector field
$$v^{\#}(z)=\overline {v(\bar z)},\qquad z\in\C,$$
is holomorphic in $\H,$
then we have (see \eqref{eq: Wplus})
$$W^+(v)=-\frac1\pi\int_D (\bp v)A+\frac 1{2\pi i}\int_{\pa D} vA,\qquad W^+(v^{\#})=\frac1{2\pi i}\int_{\pa D} v^{\#}A.$$
Since $A=\bar A$ on $\R,$ we have
$$\frac 1{2\pi i}\int_{\pa D} vA=-\overline{\frac 1{2\pi i}\int_{\pa D} v^{\#}A}=-\overline{ W^+(v^{\#})},$$
and
$$\frac1\pi\int_D (\bp v)A=-W^+(v)-\overline{ W^+(v^{\#})}.$$

\ss Let us choose $v=v_\zeta$ with $\zeta\in\H.$
(We could have chosen $v=v_\zeta+a+bz+cz^2$; note that $v=z^3$ as a vector field has a pole at infinity.)
Then $v^{\#}=v_{\bar\zeta}$ and $\bp v=-\pi \delta_\zeta,$ so
$$\frac1\pi \int_D (\bp v)A=-A(\zeta).$$
\end{proof}

\ms \section{Ward's equations in the half-plane} \label{sec: Ward's in H} \index{Ward's!equations}

We continue to consider the case $D=\H$ with the global identity chart.

\begin{prop} \label{Ward's in H}
Suppose $A$ satisfies the conditions of the previous proposition.
Let $X=X_1\cdots X_n$ be the tensor product of $(\lambda_j,\lambda_{*j})$-differentials $X_j$ in $\FF(W).$
Then
\begin{equation} \label{eq: Ward's in H}
\E\,A(\zeta)X=\sum_j\left[\frac{\pa_j}{\zeta-z_j}+\frac{\lambda_j}{(\zeta-z_j)^2}+\frac{\bp_j}{\zeta-\bar z_j}+\frac{\lambda_{*j}}{(\zeta-\bar z_j)^2}\right]~\E X,
\end{equation}
where all fields are evaluated in the identity chart of $\H$ and $\pa_j=\pa_{z_j}.$
\end{prop}

\begin{proof}
Let us choose $v=v_\zeta$ with $\zeta\in\H.$ Then $v^{\#}=v_{\bar\zeta}.$
By Ward's identities \eqref{eq: Ward identity}, we have
$$\E\,W^+({v_\zeta})X=\E\,\LL^+_{v_\zeta}X=\sum_j\left[\frac{\pa_j}{\zeta-z_j}+\frac{\lambda_j}{(\zeta-z_j)^2}\right]~\E X,$$
$$\E\,W^+({v_{\bar\zeta}})\bar X=\E\,\LL^+_{v_{\bar\zeta}}\bar X=\sum_j\left[\frac{\pa_j}{\bar\zeta-z_j}+\frac{\bar \lambda_{*j}}{(\bar\zeta-z_j)^2}\right]~\E\bar X.$$
Note that
$$\E\,\overline{W^+({v_{\bar\zeta}})}X= \overline{\E\,W^+(v_{\bar\zeta})\bar X}=\sum_j\left[\frac{\bp_j}{\zeta-\bar z_j}+\frac{\lambda_{*j}}{(\zeta-\bar z_j)^2}\right]~\E X,$$
and apply Proposition \ref{represent A}.
\end{proof}

\ms We repeat that we have derived the equations~\eqref{eq: Ward's in H} in the half-plane uniformization.
Furthermore, we assumed that $A$ was real on $\pa D$ and has no singularities.
For example, in the case of non-trivial boundary condition ($\widehat\Phi = \Phi + u,$ where $\Phi$ is the Gaussian free field and $u$ is a real-valued harmonic function in $D,$ see Section~\ref{sec: W of GFF} and Section~\ref{sec: BC}), the differential
$$\widehat A = -\frac12 J\odot J -(\pa u)J$$
is not necessarily real on $\pa D$ and $\widehat A$ may have singularities.
It is of course not difficult to derive Ward's equations for $\widehat A$ -- they will be different from \eqref{eq: Ward's in H}.

\ms Here is a generalization of the last proposition.

\begin{prop} \label{Ward equation}
We assume that $A$ satisfies the conditions of Proposition \ref{represent A}.
Let $Y, X_1,\cdots, X_n\in \FF(W)$ and let $X$ be the tensor product of $X_j$'s.
Then
$$\E\, (A*Y)(z)~X=\E\,Y(z) \LL^+_{v_z}X+\E\,\LL^-_{v_{\bar z}}[Y(z)X],$$
where all fields are evaluated in the identity chart of $\H.$
\end{prop}
(Proposition~\ref{Ward's in H} is the special case when $Y$ is the scalar field $I,$ i.e., $Y(z)\equiv1.$)
\begin{proof}
By Proposition~\ref{Ward's OPEs} we have
$$(A*Y)(z)=\lim_{\zeta\to z}[A(\zeta)Y(z)-(\LL^+_{v_\zeta}Y)(z)]$$
(we subtracted the singular part of operator product expansion).
We have
\begin{align*}
\E\,[A(\zeta) Y(z)\,X]&=\E\,[W^+_{v_\zeta} \,Y(z)\,X]+\E\,[\overline{W^+_{v_{\bar \zeta}}} \,Y(z)\,X]\\&
=\E\,\LL^+_{v_\zeta} [Y(z)X]+ \E\,\LL^-_{v_{\bar \zeta}} [Y(z)X] \\
&=\E\,Y(z) \LL^+_{v_\zeta}X+\E (\LL^+_{v_\zeta} Y)(z)X+ \E\,\LL^-_{v_{\bar \zeta}} [Y(z)X].
\end{align*}
It follows that
$$\E\,[(A*Y)(z)\,X]=\lim_{\zeta\to z}\,\E\, [Y(z)\,\LL^+_{v_\zeta}\,X]+ \E \,\LL^-_{v_{\bar \zeta}}[Y(z)X].$$
\end{proof}

\renewcommand\chaptername{Appendix}
\chapter{Ward's identities for finite Boltzmann-Gibbs ensembles} \label{appx: Gibbs}

We will construct the stress tensor $v \mapsto W(v)$ for the density fields of finite Boltzmann-Gibbs ensembles and derive the corresponding Ward's identities (well-known in the literature under the name of the \emph{loop equation}).
We hope that this discussion will somewhat clarify the meaning of the stress tensor of statistical models.
(A similar intuitive approach in the context of functional integration is one of the standard methods of introducing stress-energy tensor in quantum field theory.)

\ms Consider the following probability measure in $\C^n:$
\begin{equation} \label{eq: P}
\frac{1}{Z_n}~e^{H(\eta)}\,|d\eta|; \quad Z_n=\int_{\C^n}~e^{H}\equiv \int~e^{H(\eta)}~|d\eta|,
\end{equation}
where $|d\eta|$ is the Euclidean volume, and $H=H(\eta)$ is a given real smooth function which has sufficient growth at $\infty.$

\ms For example, the probability measure \eqref{eq: P} corresponding to
\begin{equation} \label{eq: H}
H(\eta)=\frac12\sum_{j\ne k}\log|\eta_j-\eta_k|^2-2n \sum_j Q(\eta_j), \qquad \eta=\{\eta_j\}\in\C^n,
\end{equation}
where $Q(\eta) \gg \log|\eta|$ is a given real function, describes the distribution of eigenvalues of $n\times n$ random normal matrices.

\ms Let $v$ be a smooth vector field in $\C$ (with compact support), and let $\psi_t$ denote its flow.
We will write $\Psi_t$ for the flow $\{\eta_j\}\mapsto\{\psi_t\eta_j\}$
in $\C^n.$

\ms Changing the variables
$$\eta=\Psi_t(\lambda),\qquad {\rm i.e.,}\quad \eta_j=\psi_t(\lambda_j),$$
we get
$$\int e^{H}=\int e^{H\diamond \Psi_t}~J_t,$$
where the Jacobian $J_t$ is given by the equation $|d\eta|=J_t(\lambda)~|d\lambda|,$ so
$$ J_t=1+2t\,\Re\,\trace[\pa v]+\cdots.$$
(We use the notation $\trace[f] = \sum f(\lambda_j)$ and $\diamond$ for composition.)
Denote
$$W_v\equiv W[v]=\frac d{dt}\Big|_{t=0}\left[ H\diamond \Psi_t+ J_t \right]=\sum_j[v_j\pa_jH+\bar v_j\bp_jH]+2\,\Re\,\trace[\pa v],$$
where $ v_j(\lambda)=v(\lambda_j)$ and $\pa_j=\pa_{\lambda_j}.$
Clearly, $W$ is a $\R$-linear map taking vector fields to random variables and its $\C$-linear component is
$$W^+[v]=\sum_jv_j\pa_jH+\trace[\pa v].$$

\ms In the random normal matrix case \eqref{eq: H} we have
$$W^+[v](\lambda)=\frac12\sum_{j\ne k}\frac{v(\lambda_j)-v(\lambda_k)}{\lambda_j-\lambda_k}-2n \trace[v\pa Q] +\trace[\pa v].$$

\ms Let $F=F(\lambda)$ be a random variable on $\C^n.$
Denote
$$\nabla_vF= \frac d{dt}\Big|_{t=0} F\diamond \Psi_t=\sum[v_j\pa_jF+\bar v_j\bp_jF].$$
Its $\C$-linear component is of course $\nabla^+_vF= \sum v_j\pa_jF$
and
$W^+[v]=\nabla^+_vH+\trace[\pa v].$

\ms Note that $\nabla_v$ is a differentiation (i.e., Leibniz's rule applies) in the algebra of random variables.

\begin{prop} We have 
\begin{equation} \label{eq: EWvF}
\E[W[v]F]+\E[\nabla_v F]=0.
\end{equation}

\end{prop}
\begin{proof}
\begin{align*}
\frac d{dt}\Big|_{t=0}\E[F\diamond \Psi_{-t}]&= \frac1Z~\frac d{dt}\Big|_{t=0}\int (F\diamond \Psi_{-t})~ e^H\\
&=\frac1Z~\frac d{dt}\Big|_{t=0}\int F~ e^{H\diamond \Psi_{t}}~J_t= \frac1Z~\int FW[v] e^{H}.
\end{align*}
\end{proof}

It follows that $\E[W^+[v]F]+\E[\nabla_v^+ F]=0,$ in particular $\E[W^+[v]]=0$ (``loop equation").

\ms Consider now the \emph{density} field $\rho$ of the point process \eqref{eq: P}.
By definition, $\rho$ is a $(1,1)$-differential such that
$$\int f\rho=\frac1n\sum f(\lambda_j)$$
for all (scalar) test functions $f.$

\ms \begin{prop} If $v$ is holomorphic in a neighborhood of $z,$ then
\begin{equation} \label{eq: density field}
\nabla_v\rho(z)=-\LL_v\rho(z).
\end{equation}
\end{prop}

\begin{proof}
For a test function $f$ supported in the region where $v$ is holomorphic, consider the random variable
$$\frac1n\trace[f]=\int f\rho.$$
Then
$$\frac1n\trace[f]\diamond \Psi_{-t}=\frac1n \trace[f\circ \psi_{-t}]=\int (f\circ \psi_{-t})\cdot \rho=\int f \rho_t.$$
Taking derivative in $t$ we get
$$-\int f\nabla_v\rho=\int f\LL_v\rho.$$
\end{proof}

\ms Combining \eqref{eq: EWvF} and \eqref{eq: density field} we conclude $$\E[\LL_v\rho(z)]=\E[W_v\rho(z)].$$ More generally, applying Leibniz's rules to $\LL_v$ and $\nabla_v$ we get the following:

\begin{cor*}
$$\E\,\LL_v[\rho(z_1)\cdots \rho(z_k)]~=~\E\,[W_v\rho(z_1)\cdots \rho(z_k)]$$
(as in Proposition~\ref{Ward's identities}).
\end{cor*}

\ms If a statistical model has a properly defined scaling limit as $n\to\infty,$ then one can ask the question about the validity of Ward's identities in the limit.
For example, the rescaled density field (subtract expectation and multiply by $n$) of a random normal matrix model \eqref{eq: H}, under some rather general conditions, converges to the Laplacian of the Gaussian free field with \emph{free} boundary conditions on some compact set $S = S[Q],$ see \cite{AHM11}.
Taking the logarithmic potential of the density field and subtracting the terms corresponding to the boundary, one can recover the expression for the stress tensor of the Gaussian free field from Section~\ref{sec: Ward identity}.

\renewcommand\chaptername{Lecture}
\chapter{Virasoro field and representation theory} \label{ch: T}

Let $W=(A,\bar A)$ be a stress tensor for some family $\FF$ of Fock space fields.
(We can assume that $\FF$ is the maximal such family and write $\FF=\FF(W)$ in this case.)
As we mentioned, in general the holomorphic quadratic differential $A$ does not belong to $\FF(W).$
The theory gets much more interesting if we can find a holomorphic field $T$ which does belong to $\FF(W)$ and which produces the same residue operators (for holomorphic vector fields) and therefore the same Ward's ``random variables" $W^+(v)$ as the differential $A$ does.
In the appendix to this lecture we will show that under some rather general conditions (the family $\FF$ has to be conformally invariant), such a field $T$ exists and is a Schwarzian form.
In this lecture, we will take this last property for the definition of the \emph{Virasoro} field $T.$
The Virasoro field of the Gaussian free field $\Phi$ is
$$T = -\frac12 J*J, \qquad J = \pa\Phi,$$
see Section~\ref{sec: T0}.
Further examples will be given in the next lecture.

\ms It should be mentioned that the whole theory could be constructed (as is customary in the conformal field theory literature) without representing the stress tensor $v\mapsto W(v)$ in terms of quadratic differentials $A,\bar A$ -- we could have just defined $T$ as a Schwarzian form satisfying the Virasoro operator product expansion (with \emph{central charge} $c$)
\begin{equation} \label{eq: Virasoro OPE}\index{Virasoro!OPE}
T(\zeta)T(z)\sim\frac{c/2}{(\zeta-z)^4}+\frac{2T(z)}{(\zeta-z)^2}+\frac{\pa T(z)}{\zeta-z}.
\end{equation}
In our approach, we are trying to separate two different issues.
As we argued in Appendix~\ref{appx: Gibbs}, for certain fields of statistical mechanics one can expect the existence of Ward's identities and the stress tensor.
This aspect is not specific for 2D (in the smooth category).
On the other hand, it is remarkable that conformal invariance in two dimension then gives us a Schwarzian form with the stated properties.

\ms The material of this lecture is completely standard, see e.g., \cite{DFMS97}; we just adapt it to the setting of Fock space fields.
For the sake of completeness we recall some elementary facts of representation theory, in particular, the description of level two singular vectors, which we will use later in connection with the SLE theory.

\ms \section{Virasoro field} \label{sec: T} \index{Virasoro!field $T, \widehat T$}

By definition, a Fock space field $T$ is the \emph{Virasoro field} for the family $\FF(A,\bar A)$ if
\renewcommand{\theenumi}{\alph{enumi}}
{\setlength{\leftmargini}{2.0em}
\begin{enumerate}
\ms\item $T\in\FF(W),$ and
\ms\item $T-A$ is a non-random holomorphic Schwarzian form.
\end{enumerate}}

\ms (In the asymmetric case $W=(A^+,A^-)$ one should consider two fields $T^\pm$ with the corresponding properties.)

\ms The Virasoro field $T$ is unique (if exists).
Indeed, if we have two such fields $T_1,T_2,$ then the non-random holomorphic Schwarzian form $f:=T_1-T_2$ belongs to $\FF(W).$
Therefore, $$\LL_v^+f(z)=\frac1{2\pi i}\oint_{(z)} vA\,f(z),$$
and it is clear that $\oint_{(z)} vA\,f(z)=0$ for all holomorphic local vector fields, hence by \eqref{eq: Lie4S-form}
$$\LL^+_vf=vf'+2v'f+\mu v'''=0.$$
Considering constant and linear vector fields $v,$ we see that $f$ has to be zero.

\ms It follows that the order of $T$ as a Schwarzian form is uniquely determined; traditionally it is denoted by $c/12,$ and $c$ is called the \emph{central charge} \index{central charge} of the family $\FF(W).$

\ms Since the fields $T$ and $A$ determine the same residue operators for local holomorphic vector fields, we can replace $A$ by $T$ in the local Ward's identities
$$\LL_v^+X(z)=\frac 1{2\pi i}\oint_{(z)} vT\, X(z)$$
as well as in Ward's OPEs, see Sections \ref{sec: W} and \ref{sec: Ward's OPEs}.
We can also use $T$ to define the functionals
$$W^+(v)=\frac 1{2\pi i}\int_{\pa D} vT-\frac 1{\pi }\int_D T\,(\bp v),$$
though we now need to use Green's formula to interpret such integrals.
Since $T$ belongs to $\FF(W)$ and $T$ is a Schwarzian form, by Corollary~\ref{OPE4form} we have Virasoro operator product expansion~\eqref{eq: Virasoro OPE},
which shows in particular that we can find the central charge from the relation
$$c=2\lim_{\zeta\to z}(\zeta-z)^4\,\E\,A(\zeta)A(z).$$

\ss In the simply connected case it is often convenient to choose $A$ so that
$$A=T \qquad {\rm in}\quad(\H,\id).$$

\ms (Recall that unlike $T,$ the differential $A$ is not uniquely defined -- we can add non-random holomorphic quadratic differentials to $A.$)
If $T$ is real and continuous up to the boundary, then Ward's equations in Section~\ref{sec: Ward's in H} obviously hold (in $\H$) with $T$ instead of $A.$

\ms \begin{eg*} As we mentioned, $T=-\frac12 J*J$ is the Virasoro field for the Gaussian free field.
The central charge is $c=1.$
Indeed, if we set $A=-\frac12 J\odot J,$ then $W=(A, \bar A)$ is a stress tensor for the Gaussian free field, see Proposition~\ref{W of GFF}.
Then $T$ is the Virasoro field because
\renewcommand{\theenumi}{\alph{enumi}}
{\setlength{\leftmargini}{2.0em}
\begin{enumerate}
\ms\item $T\in \FF(W)$ by Proposition~\ref{*n in CFT};
\ms\item $T$ is a Schwarzian form of order 1/12 $(c=1)$ by Proposition~\ref{T in S(1/12)}.
\end{enumerate}}
\end{eg*}

\ms Examples with $c\ne1$ will be given in Lecture~\ref{ch: modifications}.

\ms Our next goal is to explain the relation of the Virasoro field to the representation theory of the Virasoro algebra.

\ms \section{Commutation of residue operators} \label{sec: commutation}

It will be important to extend the definition of residue operators
$$T_v(z)=\frac1{2\pi i}\oint_{(z)}vT \qquad ({\rm in~ a~ given ~chart}\;\phi),$$
see Section~\ref{sec: residue op}, to the case of meromorphic (local) vector fields $v.$
Note that $T_v(z)\ne A_v(z)$ if $v$ has a pole at $z,$ and unlike $A_v,$ the operators $T_v$ are chart dependent.

\ms Since this part of the theory is local, we can work in a fixed chart and consider the operators
$$Y_f(z)=\frac1{2\pi i}\oint_{(z)} f(\zeta)Y(\zeta)~d\zeta$$
for arbitrary holomorphic Fock space fields $Y$ and meromorphic non-random function $f.$
(We do not require $f$ to satisfy any particular transformation rule.)
The following \emph{commutation identity} \index{commutation identity} is the source of many useful relations, and is a typical example of the contour integration technique in conformal field theory.

\ss \begin{prop} \label{commutation id}
Let $Y^1$ and $Y^2$ be two holomorphic Fock space fields, and let $f_1$ and $f_2$ be meromorphic functions.
Then
$$\left[Y^1_{f_1}(z),Y^2_{f_2}(z)\right]= \frac1{2\pi i}\oint_{(z)}f_2(\eta)\,d\eta~\frac1{2\pi i}\oint_{(\eta)}f_1(\zeta)Y^1(\zeta)Y^2(\eta)\,d\zeta.$$
\end{prop}

\begin{proof} Let us check the identity in application to $Z(z)$; $\XX$ denotes an arbitrary string satisfying $\phi^{-1}z\not\in S_\XX.$
Let $C_-,$ $C,$ and $C_+$ be three small circles around $z,$ with increasing radii.
The nodes of $\XX$ and the poles of $f$ other than $z$ should be outside of the discs.
We have
\begin{align*}
\E\,\XX Y^1_{f_1}Y^2_{f_2}\,Z(z)&=\frac1{2\pi i}\,\E\,\XX \oint_{\zeta\in C_+}f_1(\zeta)Y^1(\zeta) ~Y^2_{f_2}\,Z(z)\,d\zeta\\
&=\frac1{(2\pi i)^2}\,\E\,\XX \oint_{\zeta\in C_+}f_1(\zeta)Y^1(\zeta) \,d\zeta~\oint_{\eta\in C}f_2(\eta)Y^2(\eta) \,Z(z)\,d\eta\\
&=\frac1{(2\pi i)^2}\,\E\,\XX \oint_{\zeta\in C_+}\oint_{\eta\in C}f_1(\zeta)f_2(\eta)Y^1(\zeta) ~Y^2(\eta) \,Z(z)\,d\eta d\zeta .
\end{align*}
Similarly, we compute $Y^2_{f_2}Y^1_{f_1}\,Z(z)$ integrating the variable of $f_2$ over the bigger circle ($C$) and the variable of $f_1$ over the smaller circle ($C_-$),
$$\E\,\XX Y^2_{f_2}Y^1_{f_1}\,Z(z)=\frac1{(2\pi i)^2}\,\E\,\XX \oint_{\zeta\in C_-}\oint_{\eta\in C}f_1(\zeta)f_2(\eta)Y^1(\zeta) ~Y^2(\eta) \,Z(z)\,d\zeta d\eta.$$
Subtracting, we get
\begin{align*}
\E\,\XX \left[Y^1_{f_1},Y^2_{f_2}\right]Z(z)&=\frac1{(2\pi i)^2}\,\E\,\XX \oint_{\eta\in C}f_2(\eta)\,d\eta\oint_{\zeta\in [C_+-C_-]}f_1(\zeta)Y^1(\zeta)Y^2(\eta) \,Z(z)\,d\zeta\\
&=\frac1{(2\pi i)^2}\,\E\,\XX \oint_{\eta\in C}f_2(\eta)\,d\eta\oint_{(\eta)}f_1(\zeta)Y^1(\zeta)Y^2(\eta) \,Z(z)\,d\zeta,
\end{align*}
which completes the proof.
\end{proof}

\ms Similar formula holds for the commutator
$\big[Y^1_{f_1}(z),\overline {Y^2_{f_2}(z)}\,\big].$

\begin{egs*}
(a) Set $Y^1 = X$ and $Y^2 = Y.$ If we take $f_1\equiv1,$ then the inner integral gives the field $X*_{-1}Y.$
Taking $f_2(\eta)=(\eta-z)^{-n-1},$ by Proposition~\ref{commutation id}, we get
$$[X*_{-1}, Y*_n]~Z(z)~=~(X*_{-1}Y)*_{n}Z(z),$$
which can be rewritten as Leibniz's rule:
\begin{equation} \label{eq: Leibnitz4v2}
X*_{-1}(Y*_nZ)=(X*_{-1}Y)*_nZ+Y*_n (X*_{-1}Z).
\end{equation}
This is the formula \eqref{eq: Leibnitz4v1} we mentioned in Section~\ref{sec: W of GFF}.

\ms (b) If $X$ and $Y$ are holomorphic, then
\begin{equation} \label{eq: XYZminusYXZ}
X*(Y*Z)-Y*(X*Z)=[X,Y]*Z,
\end{equation}
where $[X,Y] = X*Y-Y*X.$
This follows from a similar argument with $f_1(\zeta)=1/(\zeta-z)$ and $f_2(\eta) = 1/(\eta-z).$
\end{egs*}

\ms In the case of residue operators of the Virasoro field, the commutation identity has the following form.
\index{commutation identity!of residue operators of the Virasoro field}
\ms\begin{prop} \label{commutation4T}
Let $T$ be the Virasoro field, and let $v_1,$ $v_2$ be (local) meromorphic vector fields. Then in any given chart, we have
$$\left[T_{v_1}(z),T_{v_2}(z)\right]=-T_{[v_1,v_2]}(z)-\frac c{24}~\frac1{2\pi i}\oint_{(z)}(v_1v_2'''-v_1'''v_2)(\zeta)\,d\zeta$$
and
$$\left[T_{v_1}(z),\overline{T_{v_2}(z)}\right]=0.$$
\end{prop}

\ms\begin{proof}
Let us compute the residue $\oint_{(\eta)}$ in Proposition~\ref{commutation id}.
By Virasoro operator product expansion~\eqref{eq: Virasoro OPE}
and Taylor series expansion of $v_1,$
$$v_1(\zeta)=v_1(\eta)+v'_1(\eta)(\zeta-\eta)+\frac{v''_1(\eta)}2(\zeta-\eta)^2+\frac{v'''_1(\eta)}6(\zeta-\eta)^3+\cdots,$$
the residue is
$$\frac1{2\pi i}\oint_{(\eta)}v_1(\zeta)T(\zeta)T(\eta)\,d\zeta=(v_1\pa T+2v'_1T+\frac c{12}v'''_1)~(\eta).$$
Next we compute
\begin{align*}
\frac1{2\pi i}\oint_{(z)}&v_2(\eta)\,d\eta\,\frac1{2\pi i}\oint_{(\eta)}v_1(\zeta)T(\zeta)T(\eta)\,d\zeta\\
&=\frac1{2\pi i}\oint_{(z)}v_2(\eta)~(v_1\pa T+2v'_1T)(\eta)\,d\eta+\frac c{12}\frac1{2\pi i}\oint_{(z)}v_2(\eta)v'''_1(\eta)\,d\eta.
\end{align*}
The first integral in the right-hand side is
$$\frac1{2\pi i}\oint_{(z)} 2(v_2v'_1~T)(\eta)\,d\eta-\frac1{2\pi i}\oint_{(z)} (v_2v_1)'(\eta)~ T(\eta)\,d\eta=\frac1{2\pi i}\oint_{(z)} [v_2,v_1](\eta)~T(\eta)\,d\eta,$$
and clearly
$$\frac1{2\pi i}\oint_{(z)} (v_2v'''_1)(\eta)\,d\eta=-\frac1{2\pi i}\oint_{(z)} (v_1v'''_2)(\eta)\,d\eta.$$
Since the operator product expansion of $T(\zeta)\bar T(\eta)$ has no singular part, the second formula follows from Proposition~\ref{commutation id}.
\end{proof}

\ms\begin{rmk*}
In the special case of \emph{holomorphic} vector fields, we get
$$\left[T_{v_1},T_{v_2}\right]=-T_{[v_1,v_2]},$$
where the operators act on arbitrary Fock space fields.
This, of course, implies
$$\left[A_{v_1},A_{v_2}\right]=-A_{[v_1,v_2]},$$
which is a stronger statement than Proposition~\ref{commutation4A} where the action is restricted to fields in $\FF(W).$
This extension of Proposition~\ref{commutation4A} has been obtained under the assumption of the existence of the Virasoro field.
In Appendix~\ref{appx: T} we will reverse the argument and derive the existence of $T$ from Proposition~\ref{commutation4A}.
\end{rmk*}

\ms\section{Virasoro algebra} \label{sec: V algebra}

Let $\AA$ denote the (Witt's) Lie algebra of meromorphic vector fields in $\widehat\C$ with possible poles only at $0$ and $\infty,$
$${\AA}={\Lin}\{l_n:~n\in\Z\},\qquad [l_m,l_n]=(m-n)l_{m+n},$$
where
$$
(l_n\,\|\,\id_\C)(\zeta)=-\zeta^{1+n}.$$
Given a chart $\phi$ at $p\in D,$ we can embed $\AA$ into the algebra of local meromorphic vector fields at $p$:
$$l_n\mapsto -(\zeta-z)^{n+1} \qquad({\rm in~chart}\;\phi,\; \phi(p)=z).$$
Proposition~\ref{commutation4T} shows that the local operators $-T_v(z)$ (defined in chart $\phi$) provide a projective representation of $\AA,$ i.e., a Lie algebra homomorphisms
$$\AA\to \LL(H)/\C\cdot I,$$
where $H$ is the linear space of Fock space fields evaluated at $z$ in chart $\phi,$ and $\LL(H)$ is the algebra of linear maps.
Equivalently, a projective representation is a
linear map
$$\varrho:~\AA\to\LL(H)$$
such that
$$[\varrho v_1,\varrho v_2]=\varrho[v_1,v_2]-\omega(v_1,v_2)\cdot I,$$
where $\omega:\AA\times \AA\to\C$ satisfies the \emph{cocycle} equation
$$\omega([v_1,v_2],v_3)+\omega([v_2,v_3],v_1)+\omega([v_3,v_1],v_2)=0.$$
It is known \cite{GF68} that (essentially) the only possible form of such a cocycle is
$$\omega(v_1,v_2)=\frac c{24}~\frac1{2\pi i}\oint_{(0)}(v_2'''v_1-v_1'''v_2)(\zeta)\,d\zeta,$$
where $c$ is a constant factor. 
(Adding to $\omega(v_1, v_2)$ any linear function of the commutator
$[v_1, v_2]$ doesn't violate the cocycle condition, but such coboundary doesn't
affect the equivalence type of the representation.)
In terms of the basis $l_n$ this means
$$\omega(l_m, l_n)=-\frac c{12}m(m^2-1)\delta_{m+n,0}.$$
In this case we say that $\rho$ is a \emph{Virasoro algebra representation} \index{Virasoro!algebra representation} with central charge $c.$
Thus we can restate Proposition~\ref{commutation4T} as follows.

\begin{prop} \label{Ln Virasoro}
Let $T$ be the Virasoro field.
Then for each $p\in D$ and each local chart at $p,$ the operators
$$L_n(z):=\frac1{2\pi i}\oint_{(z)}(\zeta-z)^{n+1} T(\zeta)~d\zeta$$
represent the Virasoro algebra:
$$[L_m, L_n]=(m-n)L_{m+n}+\frac c{12}m(m^2-1)\delta_{m+n,0},$$
where $c=12\mu$ and $\mu$ is the order of $T$ as a Schwarzian form.
\end{prop}

\ms Assuming $T^-=\bar T$ we define the residue operators
$$T_v^-(z)=-\frac1{2\pi i}\oint_{(z)}\bar v\bar T,$$
so
$T_v^-X=\overline{(T_v\bar X)}.$
It is easy to check that the operators $L_n^-,$
$L_n^-X=\overline{(L_n\bar X)},$
represent the second copy of Virasoro algebra:
\begin{equation} \label{eq: 2nd copy}
[L^-_m, L^-_n]=(m-n)L^-_{m+n}+\frac c{12}m(m^2-1)\delta_{m+n,0},
\end{equation}
and satisfy
$$[L_m, L_n^-]=0.$$
(In the asymmetric case, we need to consider $L_n$ and $L_n^-$ separately.)

\ms \section{Virasoro generators} \label{sec: V generator}

We will now consider $L_n$'s as operators $X\mapsto L_nX$ acting on fields,
\begin{equation} \label{eq: Ln as op}
(L_nX)(z)=L_n(z)X(z)\qquad (\rm in~any~ given~chart).
\end{equation}
Of course, these operators are just OPE multiplications,
$$L_{n}X=T*_{(-n-2)}X,$$
i.e.,
\begin{equation} \label{eq: Ln modes T}
T(\zeta)X(z)=
\cdots+\frac{(L_0X)(z)}{(\zeta-z)^2}+\frac{(L_{-1}X)(z)}{\zeta-z}+(L_{-2}X)(z)+\cdots \qquad (\zeta\to z).
\end{equation}
By \eqref{eq: Ln as op} and Proposition~\ref{Ln Virasoro}, the operators $L_n$ provide a Virasoro algebra representation in the space of all Fock space fields in $D.$

\ms Let us restate some facts established in the previous lectures in terms of this representation.

\begin{prop} \label{Ln act on CFT}
Virasoro generators \index{Virasoro!generators} $L_n$ act on $\FF(W).$
\end{prop}

\begin{proof} By definition, $T\in\FF(W),$ and we know that OPE coefficients of fields in $\FF(W)$ belong to $\FF(W).$
\end{proof}

\ms For $n\ge-1,$ we can identify the action of $L_n$ on $\FF(W)$ with Lie derivative operators
$$(L_nX\,\|\,\phi)(z)~=~(\LL^+_{v_n}X\,\|\,\phi)(z), \qquad (v_n\,\|\,\phi)(\zeta)=(\zeta-z)^{n+1}.$$
Thus the Lie-subalgebra $\Lin\{L_{-1}, L_{0},\cdots\}$ (in the space of operators on Fock space fields) is isomorphic to the Lie-subalgebra $\Lin\{v_{-1},v_0,\cdots\}$ (in the space of locally holomorphic vector fields) with the bracket~\eqref{eq: Lv1v2}.

\ms \begin{prop} \label{primary field}
Let $X$ be a Fock space field. Any two of the following assertions imply the third one (but neither one implies the other two):
\renewcommand{\theenumi}{\alph{enumi}}
{\setlength{\leftmargini}{2.0em}
\begin{enumerate}
\ss \item \label{item: primary1} $X\in\FF(W)$;
\ss \item \label{item: primary2} $X$ is a $(\lambda,\lambda_*)$-differential;
\ss \item \label{item: primary3} $L_{\ge1}X=0,\; L_0X=\lambda X, \; L_{-1}X=\partial X,$ and similar equations hold for $\bar X.$
\end{enumerate}}
\end{prop}

Here and below, $L_{\ge k}X = 0$ means that $L_nX = 0$ for all $n\ge k.$

\begin{proof}
Under \eqref{item: primary2}, \eqref{item: primary1} and \eqref{item: primary3} are equivalent by Corollary~\ref{OPE4diff} and \eqref{eq: Ln modes T}.
Suppose that \eqref{item: primary1} and \eqref{item: primary3} hold.
Then by Proposition~\ref{Ward's OPEs} and \eqref{eq: Ln modes T}, we get
$$(\LL_{v_\zeta}^+X)(z) = (v_\zeta\pa+\lambda v_\zeta')X(z)
\qquad\textrm{and}\qquad
(\LL_{v_\zeta}^+\bar X)(z) = (v_\zeta\pa+\bar\lambda_*v_\zeta')\bar X(z),$$
which imply that the equation
$$\LL_{v}X= (v\pa+ \bar v\bp+\lambda v'+\lambda_*\overline{v'})X$$
holds in $D_{\mathrm{hol}}(v)$ for each vector field $v.$
By Proposition~\ref{LvX4diff}, we get \eqref{item: primary2}.
\end{proof}

\ss We call fields satisfying all three conditions \emph{(Virasoro) primary} \index{primary field!Virasoro} fields in $\FF(W).$ \index{Virasoro!primary}

\ms If $X\in \FF(W)$ is a Schwarzian form of order $\mu,$ then
$$L_{\ge3}X=0,\quad L_{2}X=6\mu I\,(I(z)\equiv1),\quad L_{1}X=0,\quad L_0X=2X,\quad L_{-1}X=\pa X.$$
For $n\le -2$ we have
$$L_{n}=\frac{\partial^{-n-2} T}{(-n-2)!} ~*,$$
so the Lie-subalgebra $\Lin\{L_{-2}, L_{-3},\cdots\}$ (in the space of operators on Fock space fields) is isomorphic to the Lie algebra $\Lin\{T, \pa T,\cdots\}$ (in the space of fields) with the bracket
$$[X,Y]=X*Y-Y*X.$$
Here we use the identity~\eqref{eq: XYZminusYXZ}.

\ms \section{Singular vectors} \label{sec: singular vec}

There is an extensive literature concerning Virasoro algebra representation theory.
For example, see \cite{IK11} and references therein. 
We will only mention an elementary fact that we will need later.

\begin{prop} \label{singular vec}
Let $V$ be a primary field in $\FF(W)$ with central charge $c,$ and let $\lambda,$ $\lambda_*$ be the conformal dimensions of $V.$ Then the field
$$X=[L_{-2}+\eta L_{-1}^2]V,$$
where $\eta$ is a complex number, is also a primary field (of dimensions $\lambda+2, \lambda_*$) if and only if
\begin{equation} \label{eq: level two singular condition}
3+2\eta+4\eta \lambda=0,\qquad \frac c2+4\lambda+6\eta \lambda=0.
\end{equation}
\end{prop}

\begin{proof} 
By assumption,
\begin{equation} \label{eq: LnV}
L_{\ge1}V=0,\quad L_0V=\lambda V,\quad L_{-1}V=\pa V.
\end{equation}
Observe that for any $\eta$ we have
$$L_{-1}X=\pa X.$$
Indeed,
differentiating the operator product expansion for $T(\zeta)V(z)$ in $z,$ we get (for all $n$)
$$L_n(\pa V)=\pa (L_nV)+(n+1)L_{n-1}V.$$
On the other hand, by \eqref{eq: LnV} we have
$$L_n(\pa V)=L_nL_{-1}V=(n+1)L_{n-1}V+L_{-1}L_nV,$$
and it follows that
$L_{-1}(L_nV)=\pa (L_nV),$ in~ particular $L_{-1}X=\pa X.$

\ms We also have (for any $\eta$)
\begin{equation} \label{eq: level2}
L_{0}X=(\lambda+2) X.
\end{equation}
This follows from the identity $L_0(L_nV)=(\lambda-n)L_nV,$ which is true because
$$L_0L_nV=[L_0,L_n]V+L_nL_0V=(-n+\lambda)L_nV.$$

\ms Let us now show that the equations in \eqref{eq: level two singular condition} are equivalent to the equations $L_{1}X=0$ and $L_2X=0$ respectively, and therefore to the condition $L_{\ge1}X=0.$

\ms Since $[L_1, L_{-1}]=2L_0,$ we have $L_1L_{-1}V=2L_0 V=2\lambda V,$ and
\begin{align*}
L_1L^2_{-1}V&=(L_1L_{-1})L_{-1}V=L_{-1}(L_1 L_{-1})V + 2L_0 L_{-1}V \\
&= 2\lambda L_{-1} V+(2\lambda L_{-1} V+ 2 L_{-1}V ) = (4\lambda+2)L_{-1}V.
\end{align*}
Since we also have
$$L_1L_{-2}V = [L_1,L_{-2}]V = 3L_{-1}V,$$
the first equation in \eqref{eq: level two singular condition} is equivalent to $L_{1}X=0.$

\ms Similarly, we show
 $$L_2L_{-2}V=(4\lambda+c/2)V, \qquad L_2L_{-1}^2V=6\lambda V,$$
so the second equation in \eqref{eq: level two singular condition} is equivalent to $L_{2}X=0.$

\ms Finally we notice that since $[L_m, L^-_n]=0,$ we have
$$L_{\ge1}^-X=0,\quad L_0^-X=\lambda_*X, \quad L_{-1}^-X=\bp X,$$
and the application of Propositions~\ref{Ln act on CFT} and \ref{primary field} completes the proof.
\end{proof}

\ss \begin{rmk*}
The field $X$ is called a \emph{level two singular vector} \index{singular vector} of the Virasoro algebra representation.
``Level two" means that $X$ is an ``eigenvector" of $L_0$ (see \eqref{eq: level2}) with eigenvalue $\lambda_X:$
$$\lambda_X = 2 + \lambda_V,$$
where $\lambda_V$ is the eigenvalue of $V.$
Level two singular vectors of the second copy of Virasoro algebra (see \eqref{eq: 2nd copy}) are described similarly to Proposition~\ref{singular vec}.
``Singular" means $X$ is ``primary."
We say that $V$ produces level two singular vector $X.$
The field $V$ is called degenerate (resp. non-degenerate) if $X =0$ (resp. $X\ne0$).

\ss At level one, $X=L_{-1}V$ is singular (i.e., primary) if and only if $L_0V=0,$ i.e., $\lambda=0$:
$$0=L_1L_{-1}V=[L_1,L_{-1}]V=2L_0V.$$
\end{rmk*}

\renewcommand\chaptername{Appendix}
\chapter{Existence of the Virasoro field} \label{appx: T}

Let $A$ be a holomorphic quadratic differential and $W=(A,\bar A).$
Let us assume that the family $\FF(W)$ is big enough in the following sense:
if $H$ is a holomorphic Fock space field, then
\begin{equation} \label{eq: FW big}
\forall n<0, \quad \forall X\in\FF(W),\quad H*_nX=0 \quad
\implies \quad H\textrm{~is~non-random.}
\end{equation}
In other words, there are no non-random $H$'s such that the operator product expansion of $H$ and $X$ has no singular terms.

\ms Interpreting a well-known postulate in the physical literature which says that scale (and translation) invariance at criticality implies (in 2D) the ``invariance with respect to the full conformal group", (i.e., the applicability of conformal field theory), see \cite{BPZ84}, we will show that the Virasoro field exists in a conformally invariant situation.

\begin{prop}
Let $D$ be a simply connected domain and let $q\in\pa D.$
Suppose $A$ is invariant with respect to the group $\Aut(D,q)$ and $\FF(W)$ satisfies \eqref{eq: FW big}.
Then there exists a field $T\in\FF(A,\bar A)$ such that $T-A$ is a non-random Schwarzian form.
\end{prop}

\begin{proof} It is sufficient to show that there is a number $c$ such that the operator product expansion
\begin{equation} \label{eq: OPE(A,A)}
A(\zeta)A(z)\sim \frac{c/2}{(\zeta-z)^4}+ \frac {2A(z)}{(\zeta-z)^2}+\frac {\pa A(z)}{\zeta-z}
\end{equation}
holds in some fixed half-plane uniformization of $D$ with $q=\infty.$
Indeed, if this is true, then we can define
$$T=A+\frac c{12}S_w,$$
where $w:D\to \H$ is a conformal map and $S_w = (w''/w')'- (w''/w')^2/2$ is the Schwarzian derivative of $w$ (expressed in local charts).
Clearly, $T$ is a Schwarzian form, and to claim that $T$ is the Virasoro field for $\FF(W)$ we must show that $T\in \FF(W).$
However, this follows from Corollary~\ref{OPE4form} because $T$ satisfies Ward's OPE in the half-plane uniformization, where we have $T=A,$ and as we mentioned earlier, it is enough to check Ward's OPE for a form in just one chart.

\ms To prove \eqref{eq: OPE(A,A)} we write (in $(D,q) = (\H,\infty)$)
\begin{equation} \label{eq*: OPE(A,A)}
A(\zeta)A(z)\sim \frac{C_1(z)}{\zeta-z}+\frac{C_2(z)}{(\zeta-z)^2}+\cdots
\end{equation}
with ``undetermined" coefficients $C_n,$ which are holomorphic Fock space fields.
Recall Proposition~\ref{commutation4A} -- for all $X\in\FF(W)$ and for all local holomorphic vector fields $v_1$ and $v_2,$ we have
$$[A_{v_1},A_{v_2}]X=-A_{[v_1,v_2]}X.$$
Applying the commutation identity (Proposition~\ref{commutation id}) and the operator product expansion \eqref{eq*: OPE(A,A)}, we see that
\begin{align*}
[A_{v_1},A_{v_2}]X&=\frac1{2\pi i} \oint_{(z)} v_2v_1C_1X(z)+\frac1{2\pi i} \oint_{(z)} v_2v_1'C_2X(z)\\
&+\frac1{2!}\frac1{2\pi i} \oint_{(z)} v_2v_1''C_3X(z)+\cdots.
\end{align*}
We now set $v_1=1.$ Then $[v_1, v_2]=v_2',$ so for all $v_2,$
$$\oint_{(z)} v_2C_1X(z)=-\oint_{(z)} v_2'AX(z),$$
i.e.,
$$\oint_{(z)} v_2(C_1-\pa A)X(z)=0.$$
According to our assumption \eqref{eq: FW big}, this gives $$C_1=\pa A+c_1,$$
where $c_1$ is a \emph{non-random} holomorphic field.
Next we set $v_1(\zeta)=\zeta.$ Then $[v_1, v_2]=\zeta v_2'-v_2,$ so
$$\oint_{(z)} \zeta v_2(\pa A)X(z)+\oint_{(z)} v_2C_2X(z)
=- \oint_{(z)} \zeta v_2'AX(z)+\oint_{(z)} v_2AX(z)$$
for all $v_2,$ which gives $$C_2=2 A+c_2,$$
where $c_2$ is a non-random holomorphic field.
Next steps with $v_1(\zeta)=\zeta^k$ give
$$C_3=c_3,\quad C_4=c_4, \cdots,$$
where $c_3,c_4,\cdots$ are non-random holomorphic fields.

\ms We claim that all $c_j$'s are zero except for $c_4$ which is constant. This is where we use conformal invariance.
Recall that conformal invariance of $A$ means that for all maps $\tau(z)=kz+a$ we have
$$\E\,A(z)\Phi(z_1)\Phi(z_2)\cdots=k^2\E\,A(\tau z)\Phi(\tau z_1)\Phi(\tau z_2)\cdots.$$
It is in the sense of such correlations that we can write
$$A(\zeta)A(z)=k^4A(\tau\zeta)A(\tau z),$$
or (according to the previous discussion) equate the operator product expansions
$$\frac{\pa A(z)+c_1(z)}{\zeta-z}+\frac{2A(z)+c_2(z)}{(\zeta-z)^2}+\frac{c_3(z)}{(\zeta-z)^3}+\cdots$$
and
$$k^3\frac{\pa A(\tau z)+c_1(\tau z)}{\zeta-z}+k^2\frac{2A(\tau z)+c_2(\tau z)}{(\zeta-z)^2}+k\frac{c_3(\tau z)}{(\zeta-z)^3}\cdots.$$
By conformal invariance of $A$ we can eliminate the terms with $A$ and $\pa A,$ so we end up with the identities
$$c_n(z)=k^{4-n}c_n(kz+a).$$
Clearly this implies that $c_n=0$ unless $n=4,$ in which case $c:=2c_4$ is a constant.
\end{proof}

\renewcommand\chaptername{Appendix}
\chapter{Operator algebra formalism} \label{appx: Op algebra}

Physical and algebraic literature uses the language of operator algebra formalism.
Here we will outline the relation of this formalism to the theory that we discuss in these lectures.

\ms \section{Construction of (local) operator algebras from holomorphic Fock space fields} \label{sec: construct}

\SS Fix a coordinate chart $\phi$ at some point $p\in D$ and assume $\phi(p)=0$; all our constructions will be in this chart.
Let $\mathfrak{F}$ be a linear set of quasi-polynomial Fock space fields with the following properties:

\renewcommand{\theenumi}{\alph{enumi}}
{\setlength{\leftmargini}{2.0em}
\begin{enumerate}
\ss \item \emph{all} fields in $\mathfrak{F}$ are \emph{holomorphic};
\ss \item $I\in \mathfrak{F},$ $(I(z) \equiv 1)$;
\ss \item \label{item: F*} $\mathfrak{F}$ is closed under OPE multiplications, i.e., all OPE coefficients of two fields in $\mathfrak{F}$ belong to $\mathfrak{F}$; in particular,
$\mathfrak{F}$ is closed under differentiation;
\ss \item the map $A\mapsto A(0)$ from $\mathfrak{F}$ to the space of correlation functionals in $D\sm\{p\}$ is 1-to-1.
\end{enumerate}}

\ms For example, we can take one or several holomorphic fields, such as $J,$ $T,$ $J\odot J,$ etc., and study the corresponding OPE families.

\ms Let us denote
$$V=\{A(0):~ A\in \mathfrak{F}\},$$
so $V$ is a linear subspace in the space of Fock space functionals, and we have a bijection
$\mathfrak{F}\mapsto V,$
$$ A\mapsto a=A(0).$$
(As a general rule, we will use upper and lower cases for the fields and their values, respectively.)
For each $A\in\mathfrak{F}$ we define the corresponding \emph{operator field} $\bfA$ \index{operator field} as a sequence of operators (the ``\emph{modes}" of $\bfA$)
\begin{equation} \label{eq: modes} \index{mode expansion!of operator field}
\bfa_n=\frac1{2\pi i} \oint_{(0)}\zeta^n A(\zeta)\,d\zeta,\qquad (n\in\mathbb{Z}),
\end{equation}
which we write as a \emph{formal} power series with operator-valued coefficients
\begin{equation} \label{eq: A[z]}
\bfA[z]=\sum_{-\infty}^\infty \frac{\bfa_n}{ z^{n+1}}.
\end{equation}
(We use bold letters for operators.)

\ms The indexing in the power series can be different from \eqref{eq: A[z]}.
It usually reflects the ``conformal weight" of $A$ (i.e., the eigenvalue of $A$ as eigenvector of $L_0$ whenever this makes sense, see e.g., Proposition~\ref{primary field}).
For example, the Virasoro operator field is \index{mode expansion!of the Virasoro operator field}
$$\bfT[z]=\sum_{-\infty}^\infty \frac{\bfl_n}{z^{n+2}},\qquad \bfl_n = \frac1{2\pi i} \oint_{(0)}\zeta^{n+1} T(\zeta)\,d\zeta,$$
see Section~\ref{sec: V generator} where we used the notation $L_n$ for $\bfl_n.$
As we mentioned, the operators $\bfl_n$ generate a Virasoro algebra representation:
\begin{equation} \label{eq: [l,l]}
[\bfl_m, \bfl_n]=(m-n)\bfl_{m+n}+\frac 1{12}m(m^2-1)\delta_{m+n,0}.
\end{equation}

\ms A simpler example is the operator field
$$\bfJ[z]=\sum_{-\infty}^\infty \frac{\bfj_n}{z^{n+1}}$$
corresponding to the current $J = \pa\Phi.$
The mode operators $\bfj_n$ satisfy the relations
\begin{equation} \label{eq: [j,j]}
[\,\bfj_m, \bfj_n]=n\delta_{m+n,0},
\end{equation}
which follow from the operator product expansion $J(\zeta)J(z)\sim-1/(\zeta-z)^2$ and the commutation identity in Proposition~\ref{commutation id}.
In a different language, the operators
$$\bfp_n = i\bfj_n$$
together with $\mathbf{1}$ generate a representation of the \emph{Heisenberg algebra}: \index{Heisenberg algebra}
\begin{equation} \label{eq: [p,p]}
[\,\bfp_m, \bfp_n]=m\delta_{m+n,0}.
\end{equation}

\ms\SS We can consider $\bfa_n$'s in \eqref{eq: modes} as operators $V\to V.$
Indeed, if $b\in V,$ then $b = B(0)$ for some field $B\in\mathfrak{F},$
and
$$\bfa_n b = (A*_{-n-1}B)(0)\in V$$
because $(A*_{-n-1}B)\in\mathfrak{F}$ by \eqref{item: F*}.
Thus we have a 1-to-1 map
$$Y: V\to\End(V)[[z]],\qquad a~\mapsto ~\bfA[z],$$
which can be called the \emph{operator-state correspondence} \index{operator-state correspondence}
(elements of $V$ are \emph{states}, \index{state} and operator fields are usually called \emph{operators}).
We denote by $\End(V)[[z]]$ the collection of formal power series with operator-valued coefficients $\bfa_n\in\End(V).$
Also, we have the ``\emph{translation}" operator \index{translation operator}
$\Pa:V\to V,$
$$\Pa a:=(\pa A)(0),$$
and a distinguished (``\emph{vacuum}") state $1=I(0).$ \index{state!vacuum} \index{vacuum state}

\ms The quadruple $(V, Y, \Pa, 1)$ is a \emph{chiral operator algebra}, \index{chiral!operator algebra} according to the definition in \cite{Kac98}.
In addition to some natural properties (axioms) involving $\Pa$ and $1,$
$$[\Pa, \bfa_n]=-n\bfa_{n-1},\quad \Pa1=0,\quad \bfa_{-1}1=a,\quad \bfa_{\ge0}1=0,$$
a chiral operator algebra must have the following properties:
for all $a, b\in V,$ there exists $N$ such that
\begin{equation} \label{eq: Axiom 1}
\bfa_n b=0,\quad (n>N),
\end{equation}
and
\begin{equation} \label{eq: Axiom 2}
(z-w)^n[\bfA[z],\bfB[w]]=0,\quad (n>N),
\end{equation}
(as a formal power series).
In our case, the first property is just a restatement of the fact that operator product expansions of quasi-polynomial Fock space fields have only finitely many singular terms.
The second axiom will be explained later.
We repeat that a chiral operator algebra is attached to the point $p$ and depends on $\phi;$
in other words, the operator fields are functions
$$\bfA = \bfA(p,\phi).$$
The letter $z$ in \eqref{eq: A[z]} for $\bfA[z]$ is just a dummy variable.

\ms \section{Radial ordering} \label{sec: radial ordering}

In some computations we can treat operator fields $\bfA[z],$ which are formal power series with operator coefficients, as operator-valued meromorphic functions
$$z\mapsto\bfA(z).$$
The precise meaning of the operator $\bfA(z)$ for any particular point $z$ is the following: if $v=X(0)\in V,$ then
$$\bfA(z)v=A(z)X(0)$$
in correlations with Fock space functionals whose nodes lie outside the disc $B(0,|z|).$
In other words, $\bfA(z)$ acts from $V,$ the space of correlation functionals in $D\sm\{\phi^{-1}0\},$ to the linear space of correlation functionals in $D\sm\phi^{-1}B(0,|z|).$

\ms The operator product expansion of Fock space fields (in chart $\phi$)
\begin{equation} \label{eq: OPE(A,B)}
A(z)B(w) = \sum_{\nu=-\infty}^\infty \frac{C_\nu(w)}{(z-w)^{\nu+1}},\qquad z\to w
\end{equation}
(note that the indexing in \eqref{eq: OPE(A,B)} is different from that in Section~\ref{sec: OPE def}, see \eqref{eq: OPE(X,Y)}; also recall that there are only finitely many non-zero $C_\nu$'s with positive $\nu$)
has the following operator analogue:
\begin{equation} \label{eq: R OPE}
\RR\,\bfA[z]\,\bfB[w]=\sum \frac{\bfC_\nu[w]}{(z-w)^{\nu+1}}.
\end{equation}
Here $\RR\,\bfA[z]\,\bfB[w]$ is the \emph{radial ordering} of $\bfA[z]$ and $\bfB[w]$ \index{radial ordering} defined as a pair of formal power series in 2 variables associated with the regions $|w|<|z|$ and $|z|<|w|,$
$$\bfA[z]\,\bfB[w] = \sum_m\sum_n \frac{\bfa_m\bfb_n}{z^{m+1}w^{n+1}},\qquad(|w|<|z|),$$
and
$$\bfB[w]\,\bfA[z] = \sum_n\sum_m \frac{\bfb_n\bfa_m}{z^{m+1}w^{n+1}},\qquad(|z|<|w|).$$
The meaning of the right-hand side is more complicated.
First, we replace the negative powers of $(z-w)$ by their Taylor expansions in the corresponding regions, e.g.,
$$\frac1{z-w} = \frac1z + \frac w{z^2} + \cdots, \qquad(|w|<|z|).$$
If $C_\nu=0$ for all but finitely many $\nu$'s, then the right-hand side in \eqref{eq: R OPE} gives us a formal power series in $w,z,$ say
$$\sum\sum\frac{\bft_{m,n}}{z^{m+1}w^{n+1}},\qquad(|w|<|z|),$$
and \eqref{eq: R OPE} means exactly the equality of coefficients: $$\bfa_m\bfb_n = \bft_{m,n}.$$

\ms However, if there are infinitely many $C_\nu\ne0,$ then the coefficients $\bft_{m,n}$ appearing in the double series in the right-hand side will be infinite sums of operators and therefore will not be elements of $\End(V)$ (unless we introduce some topology in $V$).
In fact, one can interpret \eqref{eq: R OPE} as an asymptotic expansion of formal power series;
$$\RR\,\bfA[z]\,\bfB[w]=\sum_{-N}^\infty \frac{\bfC_n(w)}{(z-w)^{n+1}}~+~O\left((z-w)^N\right)
~\textrm{ for all } N>0,$$
see \cite{Kac98} for the meaning of the error term.

\ms Let us give an interpretation of \eqref{eq: R OPE} in the setting of operator algebras arising from holomorphic Fock spaces fields.
Denote by $\bft_{m,n}$ the coefficient of $z^{-m-1}w^{-n-1}$ in the right-hand side of \eqref{eq: R OPE} for $\{|w|<|z|\}.$
(As we mentioned, $\bft_{m,n}$ is an infinite sum of operators in $V.$)

\ms\begin{claim*} For all strings $\XX,$ $p\notin S_\XX,$ and for all $f\in V,$ we have
\begin{equation} \label{eq: R OPE meaning}
\E(\bfa_m\bfb_nf)\XX = \E(\bft_{m,n}f)\XX.
\end{equation}
\end{claim*}

\ms Recall that $\bfa_m\bfb_nf$ is the value of some field in $\mathfrak{F}$ evaluated at $0.$
Similarly, $\bft_{m,n}f$ is an infinite sum of such values.

\ms\begin{proof}[Proof of Claim] Let us derive \eqref{eq: R OPE meaning} from the operator product expansion \eqref{eq: OPE(A,B)}.
We will use the following notation: if $0<\delta\ll1,$ then $\XX\in(\delta)$ means that all nodes of $\XX$ are outside $\phi^{-1}B(0,\delta).$

\ms If $\XX\in(\delta)$ and $0<|w|<\delta,$ then for all $f=F(0)\in V,$ we have
\begin{equation} \label{eq: radial1}
\E[(\bfB(w)f)\XX] = \E[B(w)F(0)\XX].
\end{equation}

The correlation function in the right-hand side is analytic in $0<|w|<\delta.$
The left-hand side is the numerical series
$$\sum\frac{\E(\bfb_nf)\XX}{w^{n+1}}=\sum\frac{\E(B*_{-n-1}F)(0)\XX}{w^{n+1}},$$
which converges in $0<|w|<\delta.$

\ms Suppose that $\XX\in (\delta)$ and $0<|w|<|z|<\delta.$
Then for all $f\in V,$
$$\E[\bfA(z)\bfB(w)f\XX]=\E[A(z)B(w)F(0)\XX].$$
Note that $A(z)\XX\in(\delta_1),$ where $\delta_1=|w|.$
By \eqref{eq: radial1}, we have
\begin{align*}
\E[A(z)B(w)F(0)\XX] &= \E[B(w)F(0)A(z)\XX] = \E[(\bfB(w)f)A(z)\XX] \\
&=\sum\frac{\E[(\bfb_nf)A(z)\XX]}{w^{n+1}} = \sum_n\frac1{w^{n+1}}\sum_m\frac{\E\,\bfa_m\bfb_nf\XX}{z^{m+1}}.
\end{align*}
The double series converges absolutely and represents the analytic function
$$(z,w)\mapsto \E[A(z)B(w)F(0)\XX]\quad\textrm{in } 0<|w|<|z|<\delta.$$
Similarly,
$$\E\sum\frac{\bft_{m,n}f\XX}{z^{m+1}w^{n+1}}$$
converges absolutely and represents the same function (by operator product expansion) in the region
$$\{|z-w|< \delta/2\} \cap \{\delta/2 < |w| < |z| < \delta\}.$$
This implies the equality of coefficients.

The proof of \eqref{eq: R OPE meaning} in the region $|z|<|w|<\delta$ is similar.
\end{proof}

\ms \section{Commutation identity and normal ordering} \label{sec: ::}

\SS The commutation identity in Proposition~\ref{commutation id} can be restated in operator terms as follows.
If we have the operator product expansion \eqref{eq: OPE(A,B)}, then
\begin{equation} \label{eq: commutator OPE}
[\bfA[z],\bfB[w]]=\sum_{k\ge0}\bfC_k[w]\,\frac{\delta^{(k)}(z-w)}{k!},
\end{equation}
where $\delta(z-w)$ denotes the power series
\begin{equation} \label{eq: delta}
\delta(z-w) = \frac1z\sum_{m=-\infty}^\infty\left(\frac wz\right)^m,
\end{equation}
and $\delta^{(k)}(z-w)$ are obtained from \eqref{eq: delta} by differentiating $k$ times with respect to $w.$

\ms To prove \eqref{eq: commutator OPE}, we write
$$[\bfa_m,\bfb_n] = \frac1{2\pi i}\oint_{(0)}\eta^n~\frac1{2\pi i}\oint_{(\eta)}\zeta^mA(\zeta)B(\eta)\,d\zeta\,d\eta.$$
The operator product expansion of $A(\zeta)B(\eta)$ and the binomial expansion of $\zeta^m$ at $\eta$ give us the (Borcherds') formula: \index{Borcherds' formula}
\begin{equation} \label{eq: Borcherds}
[\bfa_m,\bfb_n] =\sum_{k\ge0} \binom{m}{k}\bfc_{k,m+n-k},
\end{equation}
where $\bfc_{k,l}$'s are the mode of $\bfC_k.$
Thus we have
\begin{align} \label{eq: [A[z],B[w]]}
[\bfA[z],\bfB[w]]&=\sum_m\sum_n\frac{[\bfa_m,\bfb_n]}{z^{m+1}w^{n+1}}=\sum_{k\ge0}\sum_l\frac{\bfc_{k,l}}{w^{l+1}} \sum_m\binom mk z^{-m-1}w^{m-k}\\ &=\sum_{k\ge0}\bfC_k[w]\frac{\delta^{(k)}(z-w)}{k!}. \nonumber
\end{align}
Note that \eqref{eq: commutator OPE} implies the axiom~\eqref{eq: Axiom 2}
because the right-hand side in \eqref{eq: commutator OPE} is a finite sum.
(There are only finitely many non-zero $\mathbf{C}_k$'s. 
Compare \eqref{eq: OPE(A,B)} to \eqref{eq: commutator OPE}.)

\bs\SS The formula \eqref{eq: Borcherds} can be used to restate the radial ordering formula~\eqref{eq: R OPE} in terms of \emph{normal ordering}.
By definition \index{normal ordering}
$$:\!\bfA[z]\,\bfB[w]\!:~=\bfA_+[z]\,\bfB[w]+\bfB[w]\,\bfA_-[z],$$
where $\bfA_+[z] = \sum_{m<0}\bfa_mz^{-m-1},$ and $\bfA_-[z] = \bfA[z]-\bfA_+[z]$ is the principal part of the power series.
Then the operator product expansion~\eqref{eq: R OPE} takes the form
\begin{equation} \label{eq: OA OPE}
\RR~\bfA[z]\,\bfB[w]=~:\!\bfA[z]\,\bfB[w]\!:+\sum_{k\ge0} \frac{\bfC_k[w]}{(z-w)^{k+1}}.
\end{equation}
Note that all terms here are well-defined as formal power series.
To prove \eqref{eq: OA OPE} in the region $|z|>|w|,$ first we note that
$$\bfA[z]\,\bfB[w]\,\,-:\!\bfA[z]\,\bfB[w]\!:~ = [\bfA_-[z],\bfB[w]].$$
It follows from \eqref{eq: Borcherds} that
$$[\bfA_-[z],\bfB[w]]=\sum_{m\ge0}\sum_n\frac{[\bfa_m,\bfb_n]}{z^{m+1}w^{n+1}}=\sum_{k\ge0}\sum_l\frac{\bfc_{k,l}}{w^{l+1}} \sum_{m\ge0}\binom mk z^{-m-1}w^{m-k}.$$
On the other hand, the inner sum $\sum_{m\ge0}\binom mk z^{-m-1}w^{m-k}$ is the power series expansion of $(z-w)^{-k-1}$ in the region $|z|>|w|.$
Thus we have
$$[\bfA_-[z],\bfB[w]]=\sum_{k\ge0}\frac{\bfC_k[w]}{(z-w)^{k+1}}.$$
Similarly, to prove \eqref{eq: OA OPE} in the region $|z|<|w|,$ one can use
$$\bfB[w]\,\bfA[z]\,\,-:\!\bfA[z]\,\bfB[w]\!:~ = -[\bfA_+[z],\bfB[w]]$$
and \eqref{eq: Borcherds}.

\ms\SS Normal ordering of operator fields can be expressed in terms of their ``modes":
$$:\!\bfA[z]\,\bfB[w]\!:~ = \sum_m\sum_n \frac{:\!\bfa_m\,\bfb_n\!:}{z^{m+1}w^{n+1}},$$
where
$$:\!\bfa_m\,\bfb_n\!:~=
\begin{cases}
\bfa_m\,\bfb_n \qquad&\textrm{if }~ m < 0;\\
\bfb_n\,\bfa_m \qquad&\textrm{otherwise.}
\end{cases}
$$

\ms The definition can be extended to the case $w=z:$
$$:\!\bfA[z]\,\bfB[z]\!:~ \equiv \sum_m\sum_n \frac{:\!\bfa_m\,\bfb_n\!:}{z^{m+n+2}}.$$
The right-hand side is well-defined as a formal power series:
if $v \in V,$ then all but finitely many terms
in
$$\sum_m\sum_n \frac{:\!\bfa_m\,\bfb_n\!:v}{z^{m+n+2}}$$
are trivial.
Clearly, the operator field $:\!\bfA\,\bfB\!:$ corresponds to the Fock space field $A*B$ under the operator-state correspondence.

\begin{eg*}
The Virasoro field
$T = -\frac12 J*J$ (see Section~\ref{sec: T})
corresponds to the Virasoro operator field
$$\bfT = -\frac12:\!\bfJ\,\bfJ\!:.$$
The modes of $\bfT$ can be stated in terms of those of $\bfJ:$
\begin{equation} \label{eq: l=:jj:}
\bfl_n=-\frac{1}{2}\sum_{k=-\infty}^\infty:\!\bfj_{-k}\,\bfj_{k+n}\!:.
\end{equation}
Since $[\,\bfj_m,\bfj_n] = 0$ unless $m+n=0,$ we can understand the normal ordering
$$:\!\bfj_m\,\bfj_n\!:~=
\begin{cases}
\bfj_m\,\bfj_n \qquad&\textrm{if }~ m\le n;\\
\bfj_n\,\bfj_m \qquad&\textrm{otherwise,}
\end{cases}
$$
i.e., in \emph{Wick's sense}: we put all ``creation" operators on the left, so we apply ``annihilation" operators first.
For example,
$$\bfl_0 = -\frac12\,\bfj_0^2-\sum_{n=1}^\infty \,\bfj_{-n}\bfj_n.$$
\end{eg*}

\ms\SS The construction \eqref{eq: l=:jj:} in the last example is purely algebraic: if we define
$$\bfl_n=\frac{1}{2}\sum_{k=-\infty}^\infty:\!\bfp_{-k}\bfp_{k+n}\!:$$
for any representation of Heisenberg's algebra, then we get a Virasoro algebra representation with $c=1.$
We assume \eqref{eq: Axiom 1} for $\mathbf{p}$ so that the action of $\mathbf{l}_n$ is well defined.
For example, one can use the generators
$$\bfp_0=\mathbf{1}~(\mathbf{1}v=1\textrm{ for all } v\in V),\quad \bfp_n=\pa_{q_n},\quad \bfp_{-n}=nq_n,\qquad (n>0)$$
which give us a Heisenberg's algebra representation in the space of quasi-polynomials $f(q_1,q_2,\cdots)e^{q_0}.$
We also have
\begin{equation} \label{eq: [l,p]}
[\bfl_m,\bfp_n] =-n\bfp_{m+n}.
\end{equation}

\ms This algebraic approach can be applied to construct Virasoro algebra representations with $c\ne1.$
For example, it is easy to verify that the generators
\begin{equation} \label{eq: c-modified ln}
\ti{\bfl}_n = \bfl_n -ib(n+1)\bfj_n
\end{equation}
give a representation with $c=1-12b^2,$ see \cite{KR87} where this modification is called Fairlie's construction. \index{Fairlie's construction}

\ms To prove \eqref{eq: c-modified ln}, we use \eqref{eq: [l,l]}, \eqref{eq: [p,p]} and \eqref{eq: [l,p]}:
\begin{align*}
[\ti{\bfl}_m,\ti{\bfl}_n] &= [\bfl_m,\bfl_n] + b(m+1)[\bfl_n,\bfp_m] -b(n+1)[\bfl_m,\bfp_n] + b^2(m+1)(n+1)[\bfp_m,\bfp_n]\\
&=[\bfl_m,\bfl_n] -b(m-n)(m+n+1)\bfp_{m+n} -b^2 m(m^2-1)\delta_{m+n,0}\\
&=(m-n)(\bfl_{m+n}-b(m+n+1)\bfp_{m+n})+\frac {1-12b^2}{12}m(m^2-1)\delta_{m+n,0}\\
&=(m-n)\ti{\bfl}_{m+n}+\frac{c}{12}m(m^2-1)\delta_{m+n,0}.
\end{align*}

\ms We will discuss these central charge modifications in the context of Fock space fields in the next lecture.

\renewcommand\chaptername{Lecture}
\chapter{Modifications of the Gaussian free field} \label{ch: modifications}

In this lecture we discuss central charge modifications \index{central charge!modification} of the Gaussian free field in a simply connected domain $D$ with a marked boundary point $q.$
These modifications appeared in \cite{RBGW07} in the context of chordal SLE theory.
Similar constructions had been well known in the physics/algebraic literature, in particular in Coulomb gas formalism, see \cite{DFMS97}, Chapter~9.

\ms The modifications of the Gaussian free field are parametrized by real numbers $b.$
We denote by $\FF_{(b)}$ the corresponding OPE families.
The families $\FF_{(b)}$ are $\Aut(D,q)$-invariant and have the central charge $c = 1-12b^2.$
Their Virasoro fields are exactly the algebraic modifications \eqref{eq: c-modified ln} mentioned in Appendix~\ref{appx: Op algebra}.

\ms Certain vertex fields in $\FF_{(b)}$ have the fundamental property of degeneracy at level two -- they produce singular null vectors.
Combining the degeneracy equations with Ward's identities we obtain the equations of Belavin-Polyakov-Zamolodchikov type (BPZ equations), which play an important role in conformal field theory.
Cardy's boundary version of BPZ equations will be used in Lecture~\ref{ch: SLE theory} to relate chordal SLE theory to conformal field theory.

\ms \textbf{Change of notation.}
From now on, we add the subscript $(0)$ to the notation of fields in the OPE family of the Gaussian free field.
(This subscript will indicate the value of the modification parameter $b.$)
Thus $\Phi_{(0)}$ is the new notation for the Gaussian free field in $D,$ and
$$J_{(0)}=\pa \Phi_{(0)}, \quad T_{(0)}=-\frac12J_{(0)}*J_{(0)}, \quad \textrm{etc.}$$

\ms \section{Construction} \label{sec: c}
For a simply connected domain $D$ with a marked boundary point $q\in \pa D,$ we consider a conformal map $$w\equiv w_{D,q}:~(D,q)\to(\mathbb{H},\infty),$$
from $D$ onto the upper half-plane $\H = \{z\in\C:\operatorname{Im} z > 0\}.$
It is important that the function $\arg w':~D\to\mathbb{R}$ does not depend on the choice of the conformal map.
Let us fix a parameter $b\in\mathbb{R}$ and define
\begin{equation} \label{eq: GFFb}
\Phi\equiv \Phi_{(b)}=\Phi_{(0)}-2b\arg w', \qquad J\equiv J_{(b)}=\pa \Phi=J_{(0)}+ib\frac{w''}{w'}.
\end{equation}
Note that $J$ is a pre-Schwarzian form of order $ib,$ and as a form it is conformally invariant with respect to $\Aut(D,q).$

\ss \begin{prop}
The field $\Phi_{(b)}$ has a stress tensor, and its Virasoro field is
$$T\equiv T_{(b)}=-\frac12 J*J+ib\pa J.$$
The central charge of $\Phi_{(b)}$ is
$$c=c(b)=1-12 b^2.$$
\end{prop}

\begin{proof} Let us define
$$A\equiv A_{(b)}=A_{(0)}+(ib\pa -j)J_{(0)},\qquad j:=\E[J]=ib{w''}/{w'}. $$
Then $A$ is a holomorphic quadratic differential.
Indeed, $ib\pa J_{(0)}$ and $jJ_{(0)}$satisfy the following transformation laws:
\begin{align*}
ib\pa J_{(0)} &= ibh'' \tilde J_{(0)}\circ h + ib(h')^2 \pa \tilde J_{(0)}\circ h, \\ 
jJ_{(0)} &= ib \Big(\frac{h''}{h'}\Big)h' \tilde J_{(0)}\circ h +(h')^2 (\tilde j\tilde J_{(0)}) \circ h.
\end{align*}

\ms We claim that $W=(A,\bar A)$ is a stress tensor for $\Phi.$
Since $\Phi$ is a pre-pre-Schwarzian form, by Corollary~\ref{OPE4form} all we need is to check Ward's OPE in $\H,$
\begin{equation} \label{eq: OPE(A,Phi)}
A(\zeta) \Phi(z) = (A_{(0)}(\zeta) + ib\pa J_{(0)}(\zeta)-j(\zeta)J_{(0)}(\zeta))\Phi(z)\sim \frac{J(z)}{\zeta-z}+ib\frac{1}{(\zeta-z)^2}.
\end{equation}
However, this is immediate from Proposition~\ref{OPE4T}~\eqref{item: OPE(T,Phi)} and \eqref{eq: OPE(J,Phi)}.

\ms Finally, let us show that
\begin{equation} \label{eq: TA}
T = A + \frac{1-12b^2}{12} S_w,
\end{equation}
where $S_w = (w''/w')'- (w''/w')^2/2$ is the Schwarzian derivative of $w.$
This will conclude the proof: $T$ is a Schwarzian form of order $c/12,$ and $T\in\FF(W)$ by Proposition~\ref{*n in CFT}.
From the expressions of $T$ and $J$ we find
$$T=A+\frac1{12}S_w-\frac{j^2}2+ibj'.$$
The last two terms can be written as
${b^2}N_w^2/2-b^2N_w'=-b^2S_w,$ so we get \eqref{eq: TA}.
\end{proof}

\ms \begin{rmk*} The Virasoro field $T$ is real and has no singularities on $\pa D$ including the point $q.$
This is easy to verify using the formula
$$\E[\pa J(\zeta)\Phi(z)]= \frac1{(\zeta-z)^2}- \frac1{(\zeta-\bar z)^2}\qquad(\textrm{in }\id_\H).$$
In particular, it follows that we can apply Ward's equations as they are stated in Section~\ref{sec: Ward's in H}.
\end{rmk*}

\ms \textbf{Notation.} Denote by $\FF_{(b)}$ the OPE family of the ``bosonic" field \index{bosonic field $\Phi_{(b)}, \widehat\Phi$} $\Phi_{(b)},$ the algebra (over $\C$) spanned by the generators $1,\pa^j\bp^k \Phi_{(b)}$ and $\pa^j\bp^k e^{*\alpha \Phi_{(b)}}\,(\alpha\in\C)$ \index{OPE!family $\FF_{(b)}$ of the bosonic field} under OPE multiplication. 
For example, OPE family $\FF_{(b)}$  contains 
$$1,\quad \Phi_{(b)}*\Phi_{(b)}, \quad J_{(b)}*J_{(b)}, \quad \pa J_{(b)}*(\Phi_{(b)}*\Phi_{(b)}) , \quad  J_{(b)}*e^{*\alpha \Phi_{(b)}},\quad \textrm{etc.}$$
Since the OPE coefficients of two conformally invariant fields are conformally invariant, the fields in $\FF_{(b)}$ are invariant with respect to $\Aut(D,q),$ and $T_{(b)}$ is their Virasoro field.

\ms \section{Vertex fields} \label{sec: vertex field}

Vertex fields \index{vertex field!non-chiral $\VV,\widehat\VV$} in $\FF_{(b)}$ are defined as OPE-exponentials of the bosonic field $\Phi=\Phi_{(b)}.$
If $\alpha\in\C,$ then by definition
$$\VV^\alpha\equiv \VV^\alpha_{(b)}=e^{*\alpha\Phi}=\sum_{n=0}^\infty\frac{\alpha^n}{n!}\Phi^{*n}.$$
We have
$$\VV^\alpha_{(b)}=e^{\alpha\varphi}\VV^\alpha_{(0)}=e^{\alpha\varphi}C^{\alpha^2}e^{\odot\alpha\Phi_{(0)}},$$
where $\varphi:=\E\Phi=-2b\arg w'$ is an imaginary part of a pre-pre-Schwarzian form and $C$ is the conformal radius, which is a $(-1/2,-1/2)$-differential.
This gives the following statement.

\begin{prop}
$\VV^\alpha_{(b)}$ is a primary field of $\FF_{(b)}$ with conformal dimensions
$$\lambda=-\frac{\alpha^2}2+i\alpha b,\qquad \lambda_*=-\frac{\alpha^2}2-i\alpha b.$$
\end{prop}

\ms Note that the expression for $\VV^\alpha_{(b)}$ in the upper half-plane does not depend on $b.$
For example, the 1-point function is $\E\VV^\alpha=(2y)^{\alpha^2},$ and the 2-point function is
\begin{equation} \label{eq: 2ptfcn4V}
\E\,\VV^\alpha(z_1)\VV^\alpha(z_2)=(4y_1y_2)^{\alpha^2}e^{2\alpha^2G(z_1,z_2)}=(4y_1y_2)^{\alpha^2}\left|\frac{z_1-\bar z_2}{z_1- z_2}\right|^{2\alpha^2}.
\end{equation}
On the other hand, the conformal properties of the vertex fields, as well as their Virasoro fields $T,$ depend on the central charge.

\ms In what follows, we will only consider vertex fields with $\alpha=i\sigma$ purely imaginary ($\sigma\in\R)$ and therefore with real conformal dimensions
$$\lambda=\frac{\sigma^2}2-\sigma b,\quad \lambda_*=\frac{\sigma^2}2+\sigma b.$$

\ms The difference $\lambda-\lambda_*=-2\sigma b$ is called the conformal \emph{spin} of the vertex field. \index{spin}
If the spin is $-1,$ then the direction of the field (in correlations with real Fock space fields) transforms as the direction of a vector field, and so the orbits of the ordinary differential equation
$$\dot z=\VV^{i\sigma}(z)$$ (if this can be defined appropriately)
are natural conformally invariant objects, see \cite{Sheffield05} and \cite{RBGW07}.

\ms Vertex fields with $\sigma=2b$ have conformal dimension $\lambda=0 $; they produce non-zero level one singular vectors $\pa \VV.$
See Remark in Section~\ref{sec: singular vec}.

\ms \section{Level two degeneracy and BPZ equations} \label{sec: degeneracy}

Let $\VV=\VV^{ia}_{(b)}.$
From the algebraic description of level two singular vectors in Proposition~\ref{singular vec} it is easy to see (use $\lambda=a^2/2-ab$ and $c=1-12b^2$) that the field
$$X=T*\VV+\eta\pa^2\VV$$
is a differential (or primary) if $\eta=-1/(2a^2)$ and $2a(a+b)=1.$

\begin{prop} \label{degeneracy} \index{level two degeneracy equation}
If $\VV=\VV^{ia}_{(b)}$ and if $$2a(a+b)=1,$$ then
\begin{equation} \label{eq: degeneracy}
T*\VV=\frac1{2a^2}\pa^2\VV.
\end{equation}
\end{prop}

\begin{proof}
Since the difference is a differential, it is sufficient to verify \eqref{eq: degeneracy} in the upper half-plane.
The proof is an easy exercise in Wick's calculus.
All computations below refer to the identity chart of the upper half-plane.
In this uniformization we have
$$\Phi=\Phi_{(0)},\quad \pa \Phi=J=J_{(0)},\quad T=T_{(0)}+ib\partial J,\quad T_{(0)}=-\frac12 J\odot J.$$
Let us first compute the $T*\VV$ in the special case $b=0$:
\begin{equation} \label{eq: T0VV}
T_{(0)}*\VV(z)=T_{(0)}\odot \VV(z)+ia(\partial J)\odot \VV(z) -ia\frac{J\odot\VV(z)}{z-\bar z} + \frac32a^2\frac{\VV(z)}{(z-\bar z)^2}.
\end{equation}
Indeed,
\begin{align*}
T_{(0)}(\zeta)\VV(z)&=-\frac 12\sum_{n\ge0}\frac{(ia)^n}{n!}(J(\zeta)\odot J(\zeta))\,\Phi(z)\odot\cdots\odot\Phi(z) C^{-a^2}(z) \\&=~ {\mathrm{I}+\mathrm{II}}+T_{(0)}\odot \VV(z) +o(1),
\end{align*}
where the terms I and II come from 1 and 2 contractions, respectively.
Since
\begin{align*}
{\mathrm{II}} &= \frac12a^2(\E[J(\zeta) \Phi(z)])^2\VV(z) = \frac12a^2 \left(\frac1{(\zeta-z)^2}-\frac2{(\zeta-\bar z)(\zeta-z)}+\frac1{(\zeta-\bar z)^2}\right)\VV(z) \\
& = \frac12a^2 \left(\frac1{(\zeta-z)^2}-\frac2{(z-\bar z)(\zeta-z)}+\frac3{(z-\bar z)^2}\right)\VV(z) + o(1),
\end{align*}
its contribution to $T_{(0)}*\VV$ is
$$\frac32a^2\frac{\VV(z)}{(z-\bar z)^2}.$$
On the other hand, it follows from $\E[J(\zeta) \Phi(z)]=1/{(\zeta-\bar z)}-1/{(\zeta-z)}$ that
$${\mathrm{I}}=ia\left(\frac{1}{\zeta-z}-\frac{1}{\zeta-\bar z} \right)J(\zeta)\odot \VV(z).$$
Thus its contribution to $T_{(0)}*\VV$ is
$$ia(\partial J)\odot \VV(z) -ia\frac{J\odot\VV(z)}{z-\bar z}.$$

\ms
To compute $T*\VV$ in the general case $b\ne0,$ we note that
\begin{equation} \label{eq: dJVV}
(\partial J)*\VV(z)=(\partial J)\odot \VV(z)-ia\frac{\VV(z)}{(z-\bar z)^2},
\end{equation}
which follows from $\E[\partial J(\zeta) \Phi(z)]=-1/{(\zeta-\bar z)^2}+1/{(\zeta-z)^2}.$
From \eqref{eq: T0VV} and \eqref{eq: dJVV}, we get
$$T*\VV(z)=T_{(0)}\odot \VV(z)+i(a+b)(\partial J)\odot \VV(z) -ia\frac{J\odot\VV(z)}{z-\bar z} + \left(\frac32a^2+ab\right)\frac{\VV(z)}{(z-\bar z)^2}.$$
The computation of the right-hand side in \eqref{eq: degeneracy} is easy:
$$\pa \VV(z) = -a^2\frac{ \VV(z)}{z-\bar z} + ia J\odot\VV(z),$$
$$\frac1{2a^2}\pa^2 \VV(z)=T_{(0)}\odot \VV(z)+\frac i{2a}(\pa J)\odot \VV(z) -ia\frac{J\odot\VV(z)}{z-\bar z} + \frac{1+a^2}2 \frac{\VV(z)}{(z-\bar z)^2}.$$
It follows that
$$T*\VV^{ia}(z)-\frac1{2a^2}\pa^2 \VV^{ia}(z) = i\left(a+b- \frac 1{2a} \right) (\pa J)\odot \VV(z) +\left(a^2+ab-\frac12\right)\frac{\VV(z)}{(z-\bar z)^2} = 0
$$
provided that
$2a(a+b)=1.$
\end{proof}

Degenerate singular vectors give rise to (BPZ) equations \index{Belavin-Polyakov-Zamolodchikov (BPZ) equations} for certain correlation functions.

\begin{prop} \label{BPZ eqs}
Let $\VV = \VV^{ia}_{(b)}$ and $2a(a+b)=1.$
Then in the $(\H,\infty)$-uniformization, we have
\begin{align*}
\frac1{2a^2}\pa_z^2\E\,\VV(z) X_1(z_1)\cdots X_n(z_n)~&=~\E\,\VV(z) \LL^+_{v_z}\left[X_1(z_1)\cdots X_n(z_n)\right]
\\&+~\E\, \LL^-_{v_{\bar z}}\left[\VV(z)X_1(z_1)\cdots X_n(z_n)\right],
\end{align*}
where $v_z(\zeta)=1/(z-\zeta)$ and the fields $X_j$ belong to $\FF_{(b)}.$
\end{prop}

\begin{proof}
Denote $X = X_1(z_1)\cdots X_n(z_n).$
Since $T*\VV=\dfrac1{2a^2}\pa^2\VV,$ we have
$$\frac1{2a^2}\pa_z^2\E[\VV(z)X] =\E\left[\frac{\pa^2\VV}{2a^2}(z)X\right]=\E\left[(T*\VV)(z)X\right],$$
so we can apply Proposition~\ref{Ward equation}.
\end{proof}

\ms \begin{eg*} The function
$$f(z,z_1,\cdots, z_n)=\E\,\VV^{ia}(z) \VV^{i\sigma_1}(z_1)\cdots \VV^{i\sigma_n}(z_n),\qquad (z,z_j\in\H),$$
e.g., the 2-point function \eqref{eq: 2ptfcn4V},
satisfies the following 2nd order linear PDE for all values of $b$:
\begin{align*}
\frac1{2a^2}\pa^2_z f&=\left(\frac{\bp_z}{z-\bar z}+\frac{\lambda_*}{(z-\bar z)^2}\right)f\\
&+\sum_{j=1}^n\left(\frac{\pa_j}{z-z_j}+\frac{\lambda_j}{(z-z_j)^2}+\frac{\bp_j}{z-\bar z_j}+\frac{\lambda_{*j}}{(z-\bar z_j)^2}\right)f,
\end{align*}
where $\lambda_j,\lambda_{*j}$ are conformal dimensions of $\VV^{i\sigma_j}$ and $\lambda_*= a^2/2+ab.$
\end{eg*}

\ms If the fields $X_j$ in Proposition~\ref{BPZ eqs} are not differentials (e.g., if they are forms), then the BPZ equations are not necessarily of PDE type.
See e.g., the Friedrich-Werner formula in Section~\ref{sec: eg MO}.

\ms \section{Boundary conditions and insertions} \label{sec: BC} \index{boundary condition modification}

We can further modify our bosonic fields $\Phi_{(b)}$ by conditioning them to have certain (non-random) boundary values.

\begin{prop}\label{sing^4non-chiral}
Let $u$ be a real non-random harmonic function in $D.$
Define
$$\widehat\Phi= \Phi_{(b)}+u,\qquad \widehat J = \pa\widehat\Phi.$$
Then

\ss (a) the field $\widehat \Phi$ has a stress tensor and its Virasoro field is
\begin{equation} \label{eq: T hat}
\widehat T=-\frac12 \widehat J*\widehat J+ib\partial \widehat J=T-(\pa u) J+ib \pa^2u-\frac12(\pa u)^2;
\end{equation}

\ss (b) the vertex fields $\widehat \VV=e^{*ia\widehat\Phi}$ with $2a(a+b)=1$ produce degenerate level two singular vectors. \index{level two degeneracy equation}
\end{prop}

\ss Conditioning changes neither the conformal nature of the fields nor the central charge.
See Remark in Section~\ref{sec: singular vec} for the meaning of the expression ``the vertex fields produce degenerate level two singular vectors."

\ss \begin{proof} (a) The field
\begin{equation} \label{eq: A hat}
\widehat A = A -(\pa u) J+ib \pa^2u-\frac12(\pa u)^2
\end{equation}
is a holomorphic quadratic differential.
Indeed, denote $f=\pa u$; it is a holomorphic 1-differential.
Since both $ibf'$ and $fJ$ are quadratic differentials with the same cocycle $ibf(\log h')',$
their difference is a quadratic differential.

\ms We claim that $\widehat W=(\widehat A,\overline{\widehat A})$ is a stress tensor for $\widehat\Phi.$
Since $\widehat\Phi$ is still a pre-pre-Schwarzian form, by Corollary~\ref{OPE4form} all we need is to check Ward's OPE (in the $(\H,\infty)$-uniformization) of $\widehat A$ and $\widehat\Phi:$
$$A(\zeta) \Phi(z) -f(\zeta)J(\zeta) \Phi(z) \sim ib\frac{1}{(\zeta-z)^2}+\frac{\widehat J(z)}{\zeta-z}.$$
However, this is immediate from \eqref{eq: OPE(A,Phi)} and \eqref{eq: OPE(J,Phi)}.

\ms From \eqref{eq: TA}, \eqref{eq: T hat}, and \eqref{eq: A hat} we derive
\begin{equation} \label{eq: TA hat}
\widehat T =\widehat A + \frac{1-12b^2}{12} S_w.
\end{equation}
It follows that $\widehat T$ is the Virasoro field because $\widehat T$ is a Schwarzian form of order $c/12,$ and $\widehat T\in \FF(\widehat W)$ by \eqref{eq: T hat} and Proposition~\ref{*n in CFT}.

\ms (b) Denote $\widehat\VV = e^{ia u}\VV,$ $\VV\equiv\VV_{(b)}^{ia},$ $f=\pa u,$ and let
$$\widehat X:=\widehat T*\widehat\VV -\frac1{2a^2}\pa^2 \widehat\VV.$$
It follows from the operator product expansion (see \eqref{eq: OPE(J,Phi)})
$$J(\zeta)\Phi(z)=-\frac1{\zeta-z} + \pa c(z) + (J\odot\Phi)(z) + o(1)$$
that
$$(fJ)*\widehat\VV = f J\odot\widehat\VV + iaf\pa c \widehat\VV -ia f'\widehat\VV.$$
Thus by the relation~\eqref{eq: T hat}, we have
$$\widehat T*\widehat\VV = T*\widehat\VV-f J\odot\widehat\VV -iaf\pa c \widehat\VV + i(a+b) f'\widehat\VV -\frac12f^2 \widehat\VV.$$
Differentiating $\widehat\VV = e^{ia u}\VV,$ we get
$$\frac1{2a^2}\pa^2 \widehat\VV = \frac1{2a^2}\left(iaf'\widehat\VV -a^2f^2\widehat\VV +2iafe^{iau}\pa\VV +e^{iau}\pa^2\VV\right).$$
The degeneracy equation~\eqref{eq: degeneracy} gives
$$\widehat X = -f J\odot\widehat\VV -iaf\pa c \widehat\VV -\frac{i}{a}fe^{iau}\pa\VV.$$
However, $\pa\VV = \pa(C^{-a^2}e^{\odot ia \Phi})= iaJ\odot\VV-a^2\pa c\VV$ in $\H.$
Thus $\widehat X = 0.$

\ms (One can argue that conditioning does not change the singular parts of the operator product expansions of $J(\zeta)\VV(z)$ and $T(\zeta)\VV(z).$
In Appendix~\ref{appx: KZ} we will explain this implies that the degeneracy equation survives under conditioning.)
\end{proof}

\ms \begin{rmks*}
(a) Ward's and BPZ equations for fields with non-trivial boundary conditions are in general not the same as the equations in Propositions~\ref{Ward's in H} and \ref{BPZ eqs}.
This is because the Virasoro field $\widehat T,$ see \eqref{eq: T hat}, may be non-real and/or may have singularities on the boundary.
For example, if
\begin{equation}\label{eq: u}
u = \mathrm{const}\cdot\arg w,\qquad w:(D,p,q)\to(\H,0,\infty),
\end{equation}
then $\widehat T$ has a double pole at the origin, and Ward's equations for differentials take the form
\begin{align*}
\E\, \widehat T(\zeta) \widehat X=\E\, \widehat T(\zeta)\,\E\, \widehat X
&+\sum_j\Big[\big(-\frac1\zeta+\frac1{\zeta-z_j} \big)\pa_j+\frac{\lambda_j }{(\zeta-z_j)^2}\Big]~\E\,\widehat X \\
&+\sum_j\Big[\big(-\frac1\zeta+\frac1{\zeta-\bar z_j}\big)\bp_j+\frac{\lambda_{*j}}{(\zeta-\bar z_j)^2}\Big]~\E\,\widehat X, \quad (\textrm{in } \id_\H),
\end{align*}
where $\widehat X = \widehat X_1(z_1)\cdots\widehat X_n(z_n)$ is the tensor product of differentials in $\FF(\widehat W).$
The BPZ equations can be adjusted accordingly.

\bs (b) In the special case (\cite{SS10})
\begin{equation} \label{eq: BCbySS}
u=2a \arg w, \qquad 2a(a+b) = 1
\end{equation}
of boundary conditions \eqref{eq: u}, we have a different type of BPZ equations -- this will be important for the SLE theory.
The nature of the equations is the following.
We can realize the boundary conditions \eqref{eq: BCbySS} by inserting a chiral vertex which produces a degenerate singular vector.
Chiral vertex fields will be defined in the next lecture.
It is probably worthwhile to explain the idea in a simpler, non-chiral situation.

\ms Fix a point $z_0\in D$ and define
$$\widehat\Phi= \Phi+2aiG_{z_0},$$
where $\Phi=\Phi_{(b)},$ the constants $a$ and $b$ satisfy $2a(a+b)=1,$ and $G_{z_0}$ is the Green's function with pole at $z_0.$
As in Proposition~\ref{sing^4non-chiral}, we can build many other fields from $\widehat\Phi,$ e.g., $\widehat J:=\pa \widehat\Phi$ or $\widehat\VV^\beta:= e^{*\beta\widehat\Phi}.$
As we explained in Appendix~\ref{sec: insert}, we can interpret such hat-fields in terms of an insertion:
$$\E\,\widehat\XX=\E[e^{\odot ia\Phi(z_0)}\XX].$$
If $\XX$ is a string of differentials, then we can apply Ward's and degeneracy equations to derive 2nd order PDE for
$$\E\,\widehat\XX = C^{-a^2}(z_0)\E[\VV_{(b)}^{ia}(z_0)\XX].$$
This equation will involve the insertion point $z_0.$
\end{rmks*}

\renewcommand\chaptername{Appendix}
\chapter{Current primary fields and KZ equations} \label{appx: KZ}

In this appendix we give an algebraic proof of Proposition ~\ref{degeneracy} (characterization of level two degenerate vertex fields).
The proof is based on the fact that vertex fields in $\FF_{(b)}$ are primary fields of the corresponding current algebra.
We also derive the so-called Knizhnik-Zamolodchikov equations (KZ equations) for correlators of current primary fields.

\section{Current primary fields} \label{sec: J primary}

Let $\{J_n\}$ and $\{L_n\}$ denote the modes of the current field $J$ and the Virasoro field $T$ in $\FF_{(b)}$ theory, respectively: \index{mode expansion! of the current field}
$$J_n(z):=\frac1{2\pi i}\oint_{(z)}(\zeta-z)^{n} J(\zeta)~d\zeta$$
and
$$L_n(z):=\frac1{2\pi i}\oint_{(z)}(\zeta-z)^{n+1} T(\zeta)~d\zeta.$$
(We consider them as operators acting on fields in $\FF_{(b)}.$)
Then we have the following equations:
\begin{equation}\label{eq: [J,J]}
[J_m,J_n] = n\delta_{m+n,0};
\end{equation}
\begin{equation}\label{eq: [L,J]}
[L_m,J_n] = -nJ_{m+n}+ibm(m+1)\delta_{m+n,0};
\end{equation}
cf. \eqref{eq: [j,j]} and \eqref{eq: [l,p]}, and also
\begin{equation}\label{eq: LJ}
L_n = -\frac12\sum_{k=-\infty}^\infty\!:J_{-k}J_{k+n}:\!~-ib(n+1)J_n,
\end{equation}
cf. \eqref{eq: l=:jj:} and \eqref{eq: c-modified ln}.

\ms Recall that $X\in\FF_{(b)}$ is (Virasoro) primary if $X$ is a $(\lambda,\lambda_*)$-differential; equivalently:
\begin{equation} \label{eq: T primary}
L_{\ge1}X=0,\qquad L_0X=\lambda X, \qquad L_{-1}X=\partial X,
\end{equation}
and similar equations hold for $\bar X$ (see Proposition~\ref{primary field}).
A (Virasoro) primary field $X$ is called \emph{current primary} \index{current!primary} if \index{primary field!current}
\begin{equation} \label{eq: J primary1}
J_{\ge1}X = J_{\ge1}\bar X = 0,
\end{equation}
and
\begin{equation} \label{eq: J primary2}
J_0X = -iqX, \quad J_0\bar X = i\bar q_*\bar X
\end{equation}
for some numbers $q$ and $q_*$ (``charges" of $X$). \index{charge!of a current primary field}
Charges determine dimensions (see \eqref{eq: L-1}):
\begin{equation} \label{eq: q lambda}
\lambda = \frac12 q^2 -bq,\qquad \lambda_*= \frac12 q_*^2 +bq_*.
\end{equation}

\begin{egs*}
\renewcommand{\theenumi}{\alph{enumi}}
{\setlength{\leftmargini}{2.0em}
\begin{enumerate}
\item \label{eg: V is current primary} \ms The vertex field $\VV^\alpha$ is current primary with charges $q=q_*=-i\alpha,$ see \eqref{eq: OPE(J,V)}.
\item\ms The current $J$ in the case $b=0$ and the Virasoro field $T$ in the case $c=0$ are Virasoro primary, but not current primary.
\end{enumerate}}
\end{egs*}

\begin{prop} \label{J primary}
If $X\in\FF_{(b)}$ is a current primary field, then
\begin{equation} \label{eq: J primary3}
J_{-1}X = -\frac iq\,\pa X, \quad J_{-1}\bar X = \frac i{\bar q_*}\,\pa \bar X.
\end{equation}
\end{prop}

\begin{proof}
First we note by \eqref{eq: LJ},
\begin{equation}\label{eq: L-1}
L_{-1}X = -J_{-1}J_0X,\qquad L_0X = -\frac12J_0^2X-ibJ_0X.
\end{equation}
It follows from \eqref{eq: T primary}, \eqref{eq: L-1}, and \eqref{eq: J primary2} that
$$\pa X = L_{-1}X= iqJ_{-1}X$$
and the similar equation holds for $\bar X.$
\end{proof}

\begin{prop} \label{degeneracy2} \index{level two degeneracy equation}
Let $V$ be a current primary field in $\FF_{(b)},$ and let $q,q_*$ be charges of $V.$
Then
$$[L_{-2}+\eta L_{-1}^2]V =0 $$
if
$$2q(b+q) = 1,\qquad \eta = -\frac1{2q^2}.$$
\end{prop}

\begin{proof}
Since $V$ is Virasoro primary, $L_{-1}V = \pa V,$ see \eqref{eq: T primary}.
By \eqref{eq: J primary3}, we have
$$L_{-1}^2V = L_{-1}\pa V = iqL_{-1}J_{-1}V.$$
It follows from \eqref{eq: LJ}, \eqref{eq: J primary2}, and \eqref{eq: J primary3} that
$$L_{-2}V = -J_{-2}J_0V -\frac12J_{-1}^2V+ibJ_{-2}V = i(b+q)J_{-2}V+\frac i{2q}J_{-1}L_{-1}V.$$

Let $X =[L_{-2}+\eta L_{-1}^2]V.$
Combining the above two equations we see that $X$ is a linear combination of $J_{-2}V, J_{-1}L_{-1}V$ and $L_{-1}J_{-1}V$:
$$X = i(b+q)J_{-2}V+\frac i{2q}J_{-1}L_{-1}V +i \eta qL_{-1}J_{-1}V.$$
On the other hand, by \eqref{eq: [L,J]}, $J_{-2},J_{-1}L_{-1}$ and $L_{-1}J_{-1}$ are not linearly independent:
$$J_{-2}+J_{-1}L_{-1}-L_{-1}J_{-1}=0.$$
Thus $X=0$ if
$$b+q=\frac1{2q}=-\eta q.$$
\end{proof}

\section{KZ equations} \label{sec: KZ} \index{Knizhnik-Zamolodchikov (KZ) equations}

\begin{prop}
Let $X_j = \VV^{i\sigma_j}$ and $X = X_1(z_1)\cdots X_n(z_n)$ be the tensor product of $X_j$'s.
Then the equation
$$\pa_{z_j}\E X=\Big(-\frac{\sigma_j^2}{z_j-\bar z_j}+\sum_{k\ne j}\sigma_j\sigma_k \big( \frac1{z_j-z_k}-\frac1{z_j-\bar z_k}\big)\Big)~\E X$$
holds in the $(\H,\infty)$-uniformization.
\end{prop}

\begin{proof}
It follows from Proposition~\ref{J primary} that
$$\pa_{z_j}\E X = i\sigma_j \E[X_1(z_1)\cdots J*X_j(z_j)\cdots X_n(z_n)]
=i\sigma_j\frac1{2\pi i}\oint_{(z_j)}\frac{\E[J(\zeta)X]}{\zeta-z_j}\,d\zeta.$$
By Schwarz reflection principle ($J$ is purely imaginary on the boundary),
$$\zeta\mapsto \E[J(\zeta) X]$$
has an analytic continuation $f$ to $\C\sm\{z_k,\bar z_k:k=1,\cdots,n\}.$
It is easy to check that the integral of $f(\zeta)/(\zeta-z_j)$ over the circle $|\zeta|=R$ tends to $0$ as $R\to\infty.$
Thus by Green's formula
\begin{align*}
\pa_{z_j}\E X &= -i\sigma_j \frac1{2\pi i}\oint_{(\bar z_j)} \frac{f(\zeta)}{\zeta-z_j}\,d\zeta \\ &-i\sigma_j\sum_{k\ne j} \left(\frac1{2\pi i}\oint_{(z_k)} \frac{f(\zeta)}{\zeta-z_j}\,d\zeta + \frac1{2\pi i}\oint_{(\bar z_k)} \frac{f(\zeta)}{\zeta-z_j}\,d\zeta\right).
\end{align*}
Proposition now follows from \eqref{eq: J primary1} and \eqref{eq: J primary2}.
See Example~\eqref{eg: V is current primary}.
\end{proof}

\ms In contrast to the BPZ equations, KZ equations do not depend on the central charge.

\renewcommand\chaptername{Lecture}
\chapter{Multivalued conformal Fock space fields} \label{ch: chiral}

In the physical literature, chiral fields are described as elements of the ``holomorphic part" of conformal field theory.
We will try to interpret these objects in our ``statistical" setting, in terms of conformal Fock space fields.

\ms The  following example is meant to illustrate the idea of a chiral field.
We define
$$\Phi_n(z) = \sqrt2 \sum_{j=1}^n (G(z,\lambda_j) -G(z,\mu_j)),$$
where $G$ is the Green's function in the unit disc $\D$ and $\{\lambda_j\}_{j=1}^n,$ $\{\mu_j\}_{j=1}^n$ are two independent copies of the eigenvalues in the Ginibre ensemble (the case $Q(z) = |z|^2$ in the random normal matrix model mentioned in Appendix~\ref{appx: Gibbs}), see Figure~\ref{figure: GFFn}.
The random function $\Phi_n$ approximates the Gaussian free field in $\D$ (with zero boundary conditions);
$$\Phi_n\to\Phi$$
as distributional fields.

\begin{figure}[ht]
\begin{center}
\begin{minipage}{.4\textwidth}
\centering
\includegraphics[height=.95\textwidth,width=.95\textwidth]{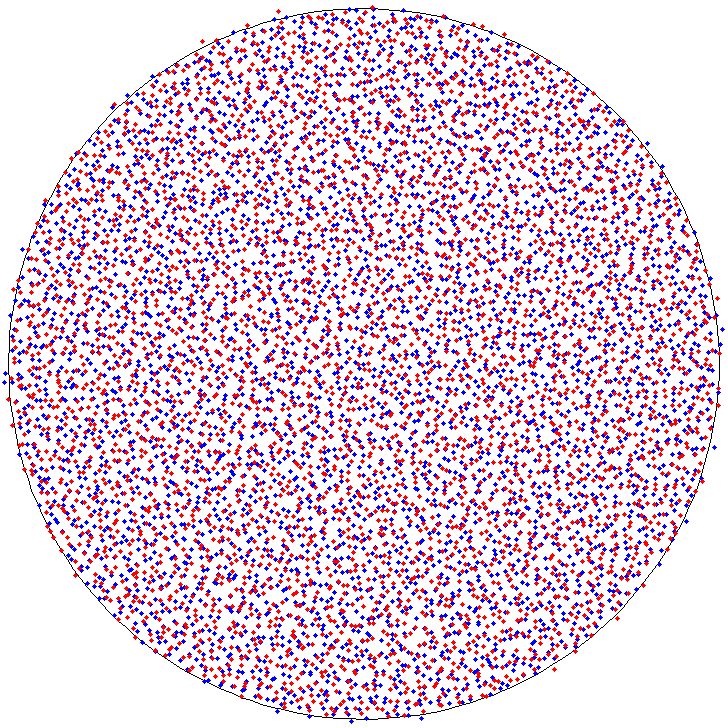}
\end{minipage}
\begin{minipage}{.15\textwidth}
\quad
\end{minipage}
\begin{minipage}{.4\textwidth}
\centering
\includegraphics[height=1.4\textwidth,width=\textwidth]{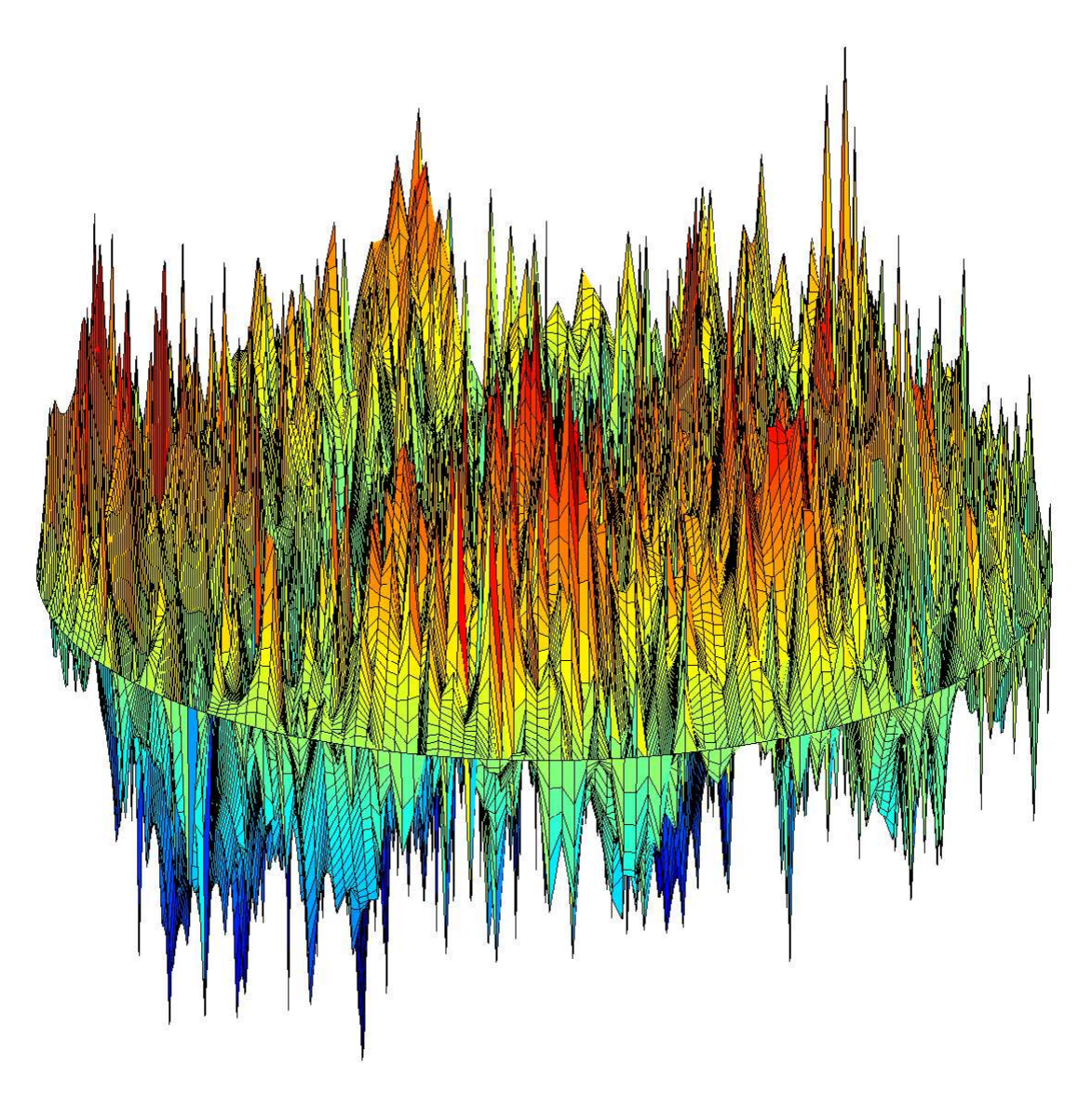}
\end{minipage}
\end{center}
\caption{Eigenvalues and graph of $\Phi_n$}
\label{figure: GFFn}
\end{figure}

Let $\widetilde\Phi_n$ be the harmonic conjugate of $\Phi_n$ and set
$$\Phichiral_n=\frac{\Phi_n+i\widetilde\Phi_n}2.$$
This random function is holomorphic and ramified at many points.
On approximate level, chiral vertex fields are properly normalized exponentials of $\Phichiral_n.$
In the limit $n\to\infty,$ such fields will be ramified everywhere, but in correlations with any particular Fock space functionals, their monodromy group will be finitely generated.

\ms Some of the reasons to study chiral fields will become clear in the last two lectures.

\ms \section{Chiral bosonic fields} \label{sec: chiral Phi} \index{chiral!bosonic field $\Phichiral$} \index{bosonic field $\Phi_{(b)}, \widehat\Phi$!chiral $\Phichiral$}

\SS In this section we will define multivalued fields $\Phichiral\equiv\Phichiral_{(b)}$ in $(D,q),$ a simply connected domain $D$ with a marked boundary point $q.$
As in the previous lecture, $b\in\R$ is a fixed parameter (so $c=1-12b^2$ will be the ``central charge"), and
$$\Phi\equiv \Phi_{(b)}=\Phi_{(0)}-2b\arg w', \qquad J\equiv J_{(b)}=\pa \Phi=J_{(0)}+ib\frac{w''}{w'},$$
where $\Phi_{(0)}$ is the Gaussian free field in $D,$ and $w:(D,q)\to(\H,\infty)$ is a conformal map.
Recall that $J$ is a pre-Schwarzian form of order $ib.$

\ms \textbf{Notation.}
If $\gamma$ is a path in $D$ (or in the closure $\bar D$), then we will write
$$\Phichiral(\gamma)=\int_\gamma J(\zeta)~d\zeta.$$
We think of this ``generalized" Gaussian random variable as a correlation functional in the complement of the curve (we can define correlations with Fock space functionals $\XX$ as long as the set of nodes $S_\XX$ is in $ D\sm\gamma$).
The integral $\int_\gamma J_{(0)}(\zeta)~d\zeta$ does not depend on local coordinates but the extra term in the $b\ne0$ theories requires a specification of coordinate charts at the endpoints of $\gamma.$
In the following discussion we will use the identity chart $\id_D$ unless the opposite is explicitly stated.

\ms If $z,z_0\in D$ (not necessarily distinct), then
$$\Phichiral(z,z_0)=\{\Phichiral(\gamma):\;\gamma\textrm{ is a curve from } z_0\textrm{ to }z\}$$
is a \emph{multivalued} correlation functional \index{Fock space correlation functional!multivalued} \index{multivalued Fock space correlation functional}
$$\XX\mapsto \E[\Phichiral(z,z_0)\XX],$$
where we only consider curves in the complement of $S_\XX.$
Since $J$ is holomorphic, homotopic curves give the same values.
Varying the endpoints, we obtain a \emph{bi-variant} field $\Phichiral$ whose values are multivalued functionals
$$\Phichiral(z,z_0)=\Phichiral_{(0)}(z,z_0)+ib\log w'\big|^{z}_{z_0}.$$

\ms We can ``freeze" the point $z_0$ and consider $\Phichiral(z)\equiv \Phichiral(z,z_0)$ as a function of one variable; more about it later, see Section~\ref{sec: V*}.
Similarly, we can consider the ``monodromy field" $z\mapsto \Phichiral(z,z).$

\ms We can also define the \emph{harmonic conjugate} $\widetilde\Phi$ \index{bosonic field $\Phi_{(b)}, \widehat\Phi$!harmonic conjugate $\widetilde\Phi$} of bosonic field $\Phi$ by the equation
\begin{equation} \label{eq: dual boson}
2\Phichiral=\Phi+i\widetilde\Phi,
\end{equation}
where $\Phi=\Phi(z,z_0)$ is of course $\Phi(z)-\Phi(z_0).$
Then we have
$$\widetilde\Phi=2\,\Im\, \Phichiral=i\Phi-2i\Phichiral=\int*d\Phi=\int i\bar J d\bar z-iJdz.$$
If both endpoints are on the boundary, then $\widetilde\Phi=-2i\Phichiral.$

\ms \SS We can talk about ``branches" of the multivalued field $\Phichiral$ in the sense of correlations with a \emph{fixed} functional $\XX,$ e.g., we have single-valued branches of $\E\Phichiral(z)\XX$ in any simply connected domain $U\subset D\sm S_\XX.$
It is in terms of such branches that we understand conformal properties and define derivatives of the multivalued field $\Phichiral.$
In particular:

\ms (a) the branches of $\Phichiral$ depend on local coordinates, e.g., $\Phichiral$ is a pre-pre-Schwarzian form of order $\pm ib$ with respect to the endpoints;

\ms (b) the branches of $\Phichiral$ have the usual derivatives,
\begin{equation} \label{eq: dPhichiral}
\pa_z\Phichiral(z, z_0)=J(z), \qquad \pa_{z_0}\Phichiral(z, z_0)=-J(z_0),
\end{equation}
and of course $\bar\pa_z\Phichiral=\bar\pa_{z_0}\Phichiral=0,$ so $\Phichiral$ is a
``holomorphic" field;

\ms (c) if a vector field $v$ is analytic at $z$ and $z_0,$ then the branches of $\Phichiral$ at $z$ have the Lie derivative
$$\LL_v(z)\Phichiral(z, z_0)=\LL_v^+(z)\Phichiral(z, z_0)=v(z)J(z)+ib v'(z),$$
and the branches of $\Phichiral$ at $z_0$ have the Lie derivative
$$\LL_v(z_0)\Phichiral(z, z_0)=\LL_v^+(z_0)\Phichiral(z, z_0)=-v(z_0)J(z_0)-ib v'(z_0),$$
so the functional
\begin{equation} \label{eq: LPhichiral}
\LL_v[\Phichiral(z, z_0)]
= v(z)J(z)- v(z_0)J(z_0)+ib v'(z)-ibv'(z_0)
\end{equation}
is single-valued.

\ms \SS As we explained, the correlations of $\Phichiral(z,z_0)$ with single-valued Fock space fields are multivalued analytic functions in $z$ and $z_0$ in the complement of the nodes.
\ss \begin{egs*}

\ss (a) We have
\begin{equation} \label{eq: G^+}
\E[\Phichiral_{(0)}(z, z_0)\Phi_{(0)}(z_1)]=2(\Gchiral(z,z_1)-\Gchiral(z_0,z_1)),
\end{equation}
where $\Gchiral$ is the \emph{complex} Green's function, \index{complex Green's function}
$$2\Gchiral(z,z_1)=G(z,z_1)+i\widetilde G(z,z_1).$$
Here $\widetilde G$ is the harmonic conjugate of the Green's function.
The multivalued holomorphic function $\Gchiral(\cdot,z_1)$ is defined up to a constant.
In the case when we have a marked boundary point $q,$ we usually choose the constant so that $0\in \Gchiral(q, z_1),$ which makes $\Gchiral$ conformally invariant with respect to $\Aut(D,q).$
In terms of a conformal map $w:(D,q)\to (\H,\infty),$ we have
$$\Gchiral(z,z_1)=\frac12\log\frac{w(z)-\overline{w(z_1)}}{w(z)-w(z_1)}.$$
Note that $z$ and $z_1$ appear in $\Gchiral$ asymmetrically, and
\begin{equation} \label{eq: dG^+}
\partial_z\Gchiral(z,z_1)=\partial_zG(z,z_1).
\end{equation}

\ms (b) Differentiating \eqref{eq: G^+} with respect to $z_1$ we obtain the equation
\begin{equation}\label{eq: EPhichiralJ}
\E[\Phichiral(z, z_0)J(z_1)]=\frac{1}{z-z_1}-\frac{1}{z_0-z_1}\qquad\textrm{in }\; \H.
\end{equation}
It follows that the correlations $\E[\Phichiral(z, z_0)J(z_1)]$ are single-valued as functions of all three variables.
Similarly, we can show that the correlations of $\Phichiral$ with the Virasoro field $T$ are single-valued, and this fact is important because it allows us to consider Ward's OPE and apply Virasoro generators to $\Phichiral,$ see below.
\end{egs*}

\ms It is easy to describe the fields which have single-valued correlations with $\Phichiral.$
Consider the ``charge" operator \index{charge!operator}
$$Q=\frac1{2\pi i}\oint J$$
(it is just the 0-th mode of the current and it does not depend on $b$).
For example,
$$QI=0~(I(z)\equiv1),\qquad QJ=0,\qquad QT=0,$$
but
$$ Q\Phi=-I,\qquad Q\VV^\alpha=-\alpha \VV^\alpha.$$
A single-valued Fock space field $X$ has single-valued correlations with
$\Phichiral$ if and only if
$$X\in\ker Q$$
(express the correlations of $X$ with the monodromy field $\Phichiral(z,z)$ in terms of $QX$).

\ms \begin{prop*}
If $X, Y\in\ker Q,$ then all their OPE coefficients, in particular $\pa X$ and $X*Y,$ are in $\ker Q.$
\end{prop*}

\begin{proof}
It follows from Proposition~\ref{commutation id} that
$$Q(X*_nY) = Q(X)*_nY + X*_nQ(Y).$$
\end{proof}

\ms \SS We define correlations of two or more multivalued fields only for \emph{non-intersecting paths}.
For instance, in the case of the $2$-point function of $\Phichiral,$ we consider non-intersecting paths $\gamma$ and $\gamma',$ and set
\begin{equation} \label{eq: cross-ratio}
\E[\Phichiral(\gamma)\Phichiral(\gamma')]=\log\frac{(z-z'_0)(z_0-z')}{(z-z')(z_0-z'_0)}\qquad\textrm{in } \H,
\end{equation}
where the logarithm of the cross-ratio depends on the number of times one curve winds around the other.

\ms\SS Let us finally explain in what sense $\Phichiral_{(b)}$ belongs to the theory $\FF_{(b)},$ the conformal family of Fock space fields generated by $\Phi_{(b)},$ see the previous lecture.
As in the case of single-valued fields, this can be expressed in the form of Ward's OPE
\begin{equation}\label{eq: OPE4Phi+}
T(\zeta)\Phichiral(z, z_0)\sim \frac{ib}{(\zeta-z)^2}+\frac{J(z)}{\zeta-z}, \qquad \zeta\to z,
\end{equation}
where $T\equiv T_{(b)}$ and $\Phichiral\equiv\Phichiral_{(b)},$ (recall that $\Phichiral$ is a pre-pre-Schwarzian form)
or equivalently in the residue form of Ward's identities
\begin{equation}\label{eq: OPE4Phi+2}
\frac1{2\pi i}\oint_{(z)}vT\Phichiral(z, z_0)=\LL_v(z)\Phichiral(z, z_0).
\end{equation}
(Similar statements hold for $z_0.$)

\ms According to our discussion above, the meaning of \eqref{eq: OPE4Phi+} and \eqref{eq: OPE4Phi+2} is the following.
For every single-valued Fock space functional $\XX$ and every curve $\gamma$ connecting $z_0$ and $z$ in $D\sm S_\XX,$ the function
$\zeta\mapsto \E\left[T(\zeta)\Phichiral(\gamma)\XX\right]$ has an analytic continuation to $D\sm( S_\XX\cup\{z, z_0\})$ and the above relations hold for this analytic continuation.
Note that the continuation depends on the homotopy class of $\gamma$ in $D\sm S_\XX$ but the singular parts of the Laurent series do not depend on $\gamma.$

\ms To prove \eqref{eq: OPE4Phi+} we simply integrate with respect to $\eta$ the operator product expansion
$$T(\zeta)J(\eta)\sim \frac{2ib}{(\zeta-\eta)^3}+\frac{J(\zeta)}{(\zeta-\eta)^2},$$
which is equivalent to Ward's OPE for $J,$ see Corollary~\ref{OPE4form}, and interpret the result in the sense explained above.

\ms \section{Chiral bi-vertex fields} \label{sec: chiral bi-vertex} \index{chiral!bi-vertex field $V,\widehat V$}\index{vertex field!chiral bi- $V,\widehat V$}

\subsec{Definition} Our next goal is to construct normalized exponentials of the chiral bosonic fields in such a way that these multivalued fields will be $\Aut(D,q)$-invariant holomorphic differentials and will belong to $\FF_{(b)}$ in the sense explained above.

\ms Recall that in the non-chiral case, the use of OPE multiplication automatically produces fields with these properties.
We cannot directly extend this approach to chiral fields because of the difficulties with the definition of their correlation functions and therefore with operator product expansions.
Without going into details, we will just state the definition of chiral vertex fields and then verify the properties.
(See also Section~\ref{sec: star}.)

\ms So, by definition (for $z\ne z_0, \alpha\in\C$),
\begin{align*}
V^\alpha(z,z_0)&\equiv V_{(b)}^\alpha(z,z_0)=\{V_{(b)}^\alpha(\gamma)\}=\left(\frac{w'(z)w'(z_0)}{(w(z)-w(z_0))^2}\right)^{-\alpha^2/2} e^{\odot \alpha\Phichiral_{(b)}(z,z_0)}\\&=
\left(\frac{w'(z)w'(z_0)}{(w(z)-w(z_0))^2}\right)^{-\alpha^2/2}~\left(\frac{w'(z)}{w'(z_0)}\right)^{i\alpha b}~e^{\odot \alpha\Phichiral_{(0)}(z,z_0)},
\end{align*}
where $w$ is any conformal map $(D,q)\to (\H,\infty)$ and $\gamma$ is a curve from $z_0$ to $z.$

\ms There is no difficulty in interpreting Wick's exponentials $e^{\odot \alpha\Phichiral_{(0)}(z,z_0)}.$
Their correlations with single-valued Fock space fields are given by Wick's calculus and by correlations of $\Phichiral,$ see the previous section.
Clearly, $e^{\odot \alpha\Phichiral_{(0)}(z,z_0)}$ is a scalar field and it is invariant with respect to $\Aut(D).$
It is also clear that the field $e^{\odot \alpha\Phichiral_{(0)}(z,z_0)}$ is holomorphic in both variables.

\ms The function
$$\left(\frac{w'(z)w'(z_0)}{(w(z)-w(z_0))^2}\right)^{-\alpha^2/2}$$
does not depend on the choice of $w.$
As we mentioned in Section~\ref{sec: conf inv}, this is a non-random holomorphic differential of dimensions $-{\alpha^2}/2$ with respect to both variables; it is invariant with respect to $\Aut(D).$
It follows that $V_{(b)}^\alpha(z,z_0)$ is a holomorphic differential of conformal dimension
$$\lambda=-\frac{\alpha^2}2+i\alpha b$$
with respect to $z$ and of conformal dimension
$$\lambda_0=-\frac{\alpha^2}2-i\alpha b$$
with respect to $z_0$; it is $\Aut(D,q)$-invariant.

\ms\subsec{Ward's OPE}
Let us now establish Ward's OPEs for chiral vertex fields.
This will allow us to say that the field $V_{{(b)}}^\alpha$ belongs to $\FF_{(b)}.$
The meaning of this statement was explained in Section~\ref{sec: chiral Phi} -- we need to consider correlations with single-valued Fock space functionals $\XX.$
It is crucial that such correlations are not ramified at $\zeta=z,$ which is a consequence of the corresponding property of the chiral bosonic fields, see \eqref{eq: OPE4Phi+}.

\begin{prop}
We have
\begin{equation} \label{eq: Ward OPEs4V+}
T(\zeta)V^\alpha(z,z_0) \sim \lambda \frac{V^\alpha(z,z_0)}{(\zeta-z)^2} + \frac{\pa_{z}V^\alpha(z,z_0)}{\zeta-z},\qquad (\zeta\to z),
\end{equation}
where $T \equiv T_{(b)}$ and $V \equiv V_{(b)}.$
Similar operator product expansion (with $\lambda_0$) holds as $\zeta\to z_0.$
\end{prop}

\begin{proof}
The proof is by Wick's calculus.
Since chiral vertex fields are differentials, it is sufficient to perform the computation in the half-plane uniformization, which simplifies the argument.
All computations below refer to the global chart of $(\H,\infty)$ as in the proof of Proposition~\ref{degeneracy}.
Thus we have
$$\Phi=\Phi_{(0)},\quad \Phichiral=\Phichiral_{(0)},\quad\pa \Phi=J=J_{(0)},\quad T=T_{(0)}+ib\partial J,\quad T_{(0)}=-\frac12 J\odot J.$$
Let us first consider the case $b=0$:
$$T_{(0)}(\zeta)e^{\odot\alpha\Phichiral}= -\frac 12\sum_{n\ge0}\frac{\alpha^n}{n!}~(J(\zeta) \odot J(\zeta)) \, (\Phichiral\odot\cdots\odot \Phichiral) ~\sim~ \mathrm{I}+\mathrm{II},$$
where the terms I and II come from 1 or 2 contractions, respectively. Thus
$$\mathrm{I}~\sim~-\alpha \E[J(\zeta)\Phichiral] J(\zeta)\odot e^{\odot\alpha\Phichiral}\quad
\textrm { and } \quad
\mathrm{II}~\sim~-\frac12\alpha^2 \E[J(\zeta)\Phichiral]^2 e^{\odot\alpha\Phichiral}.$$
By \eqref{eq: EPhichiralJ},
the singular part of operator product expansion of $T_{(0)}$ and $e^{\odot\alpha\Phichiral}$ at $z$ reads
$$\alpha\frac{J(z)}{\zeta-z}\odot e^{\odot\alpha\Phichiral} -\frac{\alpha^2}{2} \frac{e^{\odot\alpha\Phichiral}}{(\zeta-z)^2} +\frac{\alpha^2}{z-z_0}~\frac{e^{\odot\alpha\Phichiral}}{\zeta-z}.$$
This proves the relation~\eqref{eq: Ward OPEs4V+} when $b=0.$
For $b\ne 0,$ all we need to show is that
$$\Sing_{\zeta\to z}\pa J(\zeta)e^{\odot\alpha\Phichiral}= \frac{\alpha}{(\zeta-z)^2} e^{\odot\alpha\Phichiral}.$$
Since we can differentiate the singular part of operator product expansion, we just need to verify that
$$\Sing_{\zeta\to z}J(\zeta)e^{\odot\alpha\Phichiral}=-\alpha\frac{e^{\odot\alpha\Phichiral}}{\zeta-z}.$$
However, this is immediate from \eqref{eq: EPhichiralJ}.
\end{proof}

\ms\SS Global Ward's identities involving chiral vertex fields have the following meaning.
Let $\XX$ be a string of single-valued Fock space fields in the family $\FF_{(b)}.$
Then for all curves $\gamma$ in $D\sm S_\XX$ and all vector fields $v$ (which are holomorphic at the endpoints $z,z_0$ of $\gamma$ and at the nodes of $\XX$) we have
$$\E[\LL_v(V^\alpha(\gamma)\XX)]=\E[W(v)V^\alpha(\gamma)\XX].$$
This can be established following the same argument as in Lecture~\ref{ch: Ward} and using the fact that the function
$$\zeta\mapsto \E\left[T(\zeta)V^\alpha(\gamma)\XX\right]$$
extends to a single-valued analytic function in $D\sm S_\XX.$

Ward's equations, see Section~\ref{sec: Ward's in H}, also hold for chiral bi-vertex fields with similar interpretation.

\ms \section{Rooted vertex fields} \label{sec: V*} \index{rooted vertex field $V_\star,\widehat V_\star$} \index{vertex field!rooted $V_\star,\widehat V_\star$}

It is important to understand that chiral vertex fields exist only as bi-variant objects.
Nevertheless, we can freeze one of the variables, say $z_0,$ and consider $V^\alpha(z,z_0)$ as a field that depends only on $z.$
This of course results in the appearance of a new marked point unless we use the marked point $q$ that we already have from the central charge modifications ($b\ne0$).
The problem with the choice of $z_0=q$ is that it leads to a divergence that has to be taken care of by some normalization procedure.

\ms In any case, for each $b,$ we construct a one-parameter family of $\Aut(D,q)$-invariant, primary fields $V_\star^\alpha(z)$ in the ``holomorphic part" of $\FF_{(b)}$ theory.
Below we explain the definition and properties of $V_\star^\alpha,$ and clarify some points related to the fact that these fields are boundary differentials.

\ms\subsec{Boundary differentials}
Let $q\in\pa D.$
According to the discussion in Lecture~\ref{ch: conf geom}, we say that a field $F= F(z)$ in $D$ is a \emph{boundary} differential \index{differential!boundary} with respect to $q$ if it depends on the choice of a standard boundary chart (see Section~\ref{sec: conf F-field}) $\phi$ at $q,$ (in addition to local coordinates at $z$) and transforms as follows: \index{boundary differential}
\begin{equation} \label{eq: bdry diff}
(F(z)\,\|\,\phi) = (F(z)\,\|\,\widetilde\phi) (h'(0))^{\lambda_q},
\end{equation}
where $h$ is the transition map between charts $\phi$ and $\widetilde\phi$ satisfying $\phi(q) = \widetilde\phi(q)=0.$
The exponent $\lambda_q$ is called the dimension of $F$ with respect to $q.$
The following example of a boundary differential should help to clarify this concept.

\ms \begin{eg*}
Define a non-random field $F = F(z)$ in $D$ by the equation
$$F = w_\phi^\alpha(w_\phi')^\beta,$$
where $w_\phi:D\to\H$ is a conformal map normalized in a standard boundary chart $\phi, \phi(q) = 0,$ by the condition
$$w_\phi(z) = -\frac1{\phi(z)} + \cdots,\qquad z\to q.$$
Then $F$ is a boundary differential of dimension $\alpha+\beta$ with respect to $q$ (and a $(\beta,0)$-differential with respect to $z$).
\end{eg*}
\ss\begin{proof}
Clearly, $w_\phi = k w_{\widetilde\phi}$ for some $k>0.$
We have
$$k = \lim_{z\to q}\frac{w_\phi(z)}{w_{\widetilde\phi}(z)} = \lim_{z\to q}\frac{\widetilde\phi(z)}{\phi(z)} = \lim_{z\to q}\frac{h\circ \phi(z)}{\phi(z)} = h'(0),$$
so
$F = (h'(0))^{\alpha+\beta} \widetilde F.$
\end{proof}

\ms\subsec{Normalization}
Let us first consider the case $b=0.$
Fix any point $q\in \pa D$ and define
$$V_\star^\alpha(z) = V^\alpha(z,q).$$
In terms of a uniformizing map $w: (D, q)\mapsto (\H,0)$ we have
$$V_\star^\alpha=w^{\alpha^2}(w')^{-\alpha^2/2}~e^{\odot \alpha\Phichiral_{(0)}}\cdot(w'(q))^{-\alpha^2/2}.$$
The derivative $w'(q)$ requires a specification of a standard boundary chart at $q,$ so
$V_\star^\alpha(z)$ is a differential with respect to $z$ and a boundary differential with respect to $q$; both dimensions are equal to $-{\alpha^2}/2.$
If we fix a chart at $q$ and require $w'(q)=1$ in this chart, then the formula simplifies:
$$V_\star^\alpha=w^{\alpha^2}(w')^{-\alpha^2/2}~e^{\odot \alpha\Phichiral_{(0)}}.$$
However, it somewhat hides the fact that we are dealing with a boundary differential.
We obtain an even simpler expression in terms of a uniformization $w: (D, q)\mapsto (\H,\infty)$:
\begin{equation} \label{eq: V*0}
V_\star^\alpha=(w')^{-\alpha^2/2}~e^{\odot \alpha\Phichiral_{(0)}},
\end{equation}
where we require that $w(\zeta)\sim -1/(\zeta-q)$ as $\zeta\to q$ in a fixed boundary chart at $q.$
The formula~\eqref{eq: V*0} of course means that we have
$$V_\star^\alpha(z,\gamma)=\big(w'(z)\big)^{-\alpha^2/2}~e^{\odot \alpha\Phichiral_{(0)}(\gamma)}$$
for all paths $\gamma$ connecting $z$ and $q.$

\ms Let us now consider the case $b\ne0.$
As usual, $q$ denotes the central charge modification point.
By definition,
$$V_\star^\alpha \equiv V_{\star,(b)}^\alpha = (w')^{-\alpha^2/2+i\alpha b} e^{\odot\alpha\Phichiral_{(0)}},$$
where
$w: (D,q)\mapsto (\H,\infty)$ is such that $w(\zeta)\sim -1/(\zeta-q)$ as $\zeta\to q$ in a fixed boundary chart $\phi$ at $q.$
Equivalently,
\begin{equation} \label{eq: V*}
V_\star^\alpha \equiv V_\star^\alpha (z; \gamma)=\lim_{\ve\to 0}~w'(z_\ve)^{i\alpha b}~ V^\alpha(\gamma_\ve),
\end{equation}
where the curve $\gamma_\ve$ is the part of $\gamma $ from $z_\ve$ to $z,$ and the point $z_\ve$ is at (spherical) distance $\ve$ from $q$ in the chart $\phi;$ $V^\alpha \equiv V_{(b)}^\alpha$ is the bi-vertex field.

\ss This last expression represents $V_\star^\alpha$ in terms of a rescaling procedure, and allows us to derive the properties of $V_\star^\alpha$ from those of bi-vertex fields.
In particular, we obtain the following version of Ward's identities.

\ms\subsec{Ward's identities for rooted fields}
\begin{prop}
If $v$ is a non-random vector field smooth up to the boundary, and if $v(q)=v'(q)=0,$ then the Ward's identities
\begin{equation} \label{eq: Ward4V*}
\E\,\LL_v\left[V_\star^\alpha(z) X_1(z_1)\cdots X_n(z_n) \right]~=~\E\left[W(v)V_\star^\alpha(z) X_1(z_1)\cdots X_n(z_n) \right]
\end{equation}
hold true provided that $X_j\in \FF_{(b)},$ and $z$ and $z_j$'s are in $D_{\rm hol}(v).$
\end{prop}

\begin{proof}
We will use the representation~\eqref{eq: V*} and the fact that Ward's identities hold for bi-vertex fields.
All we need is to show that
$$\lim_{\ve\to0} \LL_v(z_\ve)[w'(z_\ve)^{i\alpha b}V^\alpha(\gamma_\ve)]=0.$$

\ms The field $X(z_\ve)\equiv X(z,z_\ve;\gamma_\ve)=[w'(z_\ve)^{i\alpha b}V^\alpha(\gamma_\ve)]$ is a holomorphic differential in $z_\ve,$ so its Lie derivative has two terms, with $v(z_\ve)$ and $v'(z_\ve).$
At the same time
$$\E\left[X(z_\ve)X_1(z_1)\cdots X_n(z_n)\right] = O(1),\qquad \E\left[\pa X(z_\ve)X_1(z_1)\cdots X_n(z_n)\right] = O(1)$$
as $\ve\to0.$
Since $v(q)=v'(q)=0,$ the Lie derivative at $z_\ve$ tends to zero.
\end{proof}

\ms Since the vector fields $$v_z(\zeta) = \frac1{\zeta-z} \qquad ({\rm in}\;\H)$$ have a triple zero at infinity, we can apply Ward's identities to such fields and derive, as in Section~\ref{sec: Ward's in H}, the corresponding Ward's equations. \index{Ward's!equations}
In particular, we have the following

\begin{prop} \label{Ward4VX}
If $X=X_1(z_1)\cdots X_n(z_n)$ with $X_j\in\FF_{(b)},$ then the equation
$$\E\,(T*V_\star^{\alpha})(z)X~=~\E\, V_\star^{\alpha}(z)\LL^+_{v_z}X~+~ \E\,\LL^-_{v_{\bar z}}[V_\star^{\alpha}(z)X]$$
holds in the $(\H,\infty)$-uniformization.
\end{prop}

\ms The equation obviously extends to the case when $z$ is a boundary point.
It is in this form that we will use Ward's equations in the SLE theory.

\ms \subsec{Level two degeneracy}
In terms of the action of Virasoro generators $L_n,$ see Section~\ref{sec: V generator}, and current generators $J_n,$ see Section~\ref{sec: J primary}, the rooted chiral vertex fields $V_\star^{ia}$ are Virasoro primary holomorphic fields of conformal dimension $\lambda= a^2/2-ab$ and current primary with charge $q = a.$
As we explained in Appendix~\ref{appx: KZ} this implies the following:

\begin{prop} \label{TVd2V} \index{level two degeneracy equation}
We have
$$T_{(b)}*V_\star^{ia}=\frac1{2a^2}\partial^2 V_\star^{ia},$$
provided that $2a(a+b)=1.$
\end{prop}

\ms This degeneracy equation is the reason for the close relation that exists between SLE and conformal field theory.
Combining this with Ward's equation, we will derive Cardy's equations (see Section~\ref{sec: Cardy}) that correlation functions of any fields in $\FF_{(b)}$ under the insertion of $V_\star^{ia}(\xi), (\xi\in\R)$ are annihilated by the second order differential operators
$$\frac1{2a^2}\pa_\xi^2 - \LL_{v_\xi},$$
which appear in It\^o's calculus (see Section~\ref{sec: SLE MOs}) in the context of SLE martingale-observables. 
We don't claim that $V_\star^{ia}$ is the only field that satisfies such equations and therefore gives rise to relation to SLE.

\renewcommand\chaptername{Appendix}
\chapter{CFT and SLE numerology} \label{appx: numerology}

For convenience of the reader, we give here a summary of useful formulas concerning level two singular and degenerate vertex fields.
We will also introduce an alternative parametrization which is used in the SLE theory.

\ms Recall that the conformal field theories $\FF_{(b)}$ generated by modifications of the Gaussian free field are parametrized by real numbers $b\in\R.$
The central charge of $\FF_{(b)}$ is
$$
c\equiv c(b) = 1-12 b^2.
$$
Parameters $b$ and $b' = -b$ are called \emph{dual} -- they have the same central charge.

\ms Given $b,$ there is a unique positive exponent
$$a\equiv a(b) = \frac{\sqrt{b^2+2}-b}2$$
satisfying the equation
$$2a(a+b)=1.$$
The dual exponent $a'(b):=a(b')$ is
$$a' = \frac1{2a} = a + b.$$
The SLE parameter $\kappa = \kappa(b)$ is defined by the equation
$$\kappa = \frac2{a^2},\quad\textrm{or}\quad \frac ba = \frac\kappa4-1.$$
The dual SLE parameter $\kappa' := \kappa(b')$ satisfies
$$\kappa\kappa'=16.$$
(Duplantier conjectured in \cite{Duplantier00} that $\SLE(\kappa')$ trace should describe the boundary of the hull of $\SLE(\kappa)$ when $\kappa>4.$
Duality of variants of SLE was established by Zhan in \cite{Zhan08b} and Dub\'edat in \cite{Dubedat09b} independently.)

\begin{figure}[ht]
\begin{center}
\includegraphics[width=\textwidth]{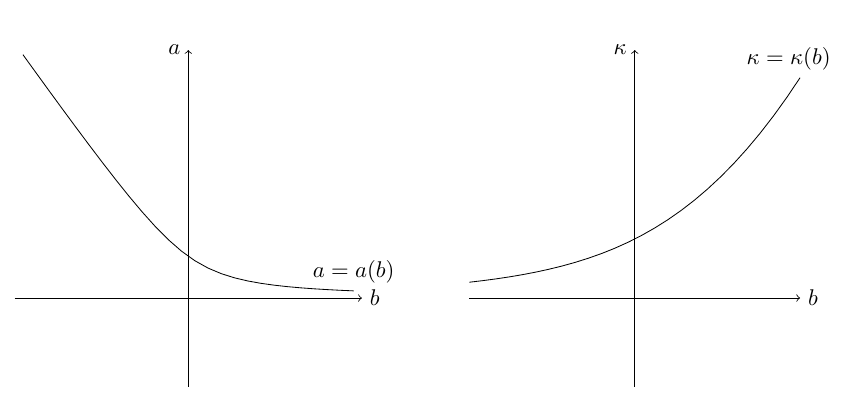}
\end{center}
\caption{Graphs of $a(b)$ and $\kappa(b)$}
\label{figure: ak}
\end{figure}

\ms The correspondence $b \leftrightarrow \kappa$ is a bijection $\R \leftrightarrow\R_+,$ see Figure~\ref{figure: ak}.
In terms of $\kappa$ we have
$$a = \sqrt{2/\kappa},\qquad b = \sqrt{\kappa/8}-\sqrt{2/\kappa},$$
and
$$c =1-\frac32\frac{(\kappa-4)^2}{\kappa} = 1 -6\Big(\sqrt{\kappa/4}-\sqrt{4/\kappa}\Big)^2 = \frac{(3\kappa-8)(6-\kappa)}{2\kappa}. $$

\ms Let us restate the algebraic description of level two singular vectors in Proposition~\ref{singular vec}.

\ms\subsec{Singular conformal dimensions} \index{singular conformal dimension} \index{conformal dimension!singular}
Define the \emph{singular conformal dimensions} $h = h(b)$ and $h' = h'(b)$ as follows:
$$h := \frac{3a^2}2-\frac12 = \frac{6-\kappa}{2\kappa}, \qquad h' := \frac{3}{8a^2}-\frac12 = \frac{3\kappa-8}{16}.$$
Note that $h'(b) = h(b')$ and $h'(b') = h(b),$ so the unordered pair $\{h,h'\}$ is the same for $b$ and $b',$ which means that it depends only on the central charge:
$$\{h,h'\} = \{\frac1{16}(5-c\pm\sqrt{(1-c)(25-c)})\}$$
(use $h+h' = (5-c)/8,$ and $hh' = c/16$).
The traditional notation for $\{h,h'\}$ is $\{h_{1,2}, h_{2,1}\}.$
Let us also denote
$$\eta \equiv \eta(b) := -\frac1{2a^2} = -\frac\kappa4,\qquad \eta' \equiv \eta'(b) := \eta(b') = -\frac4\kappa.$$
Note
\begin{equation} \label{eq: eta}
\eta+\eta' =\frac{c-13}{6},\qquad\eta\eta' = 1,
\end{equation}
and
\begin{equation} \label{eq: eta and h}
\eta = -\dfrac{3}{2(2h+1)}.
\end{equation}

\begin{table}[ht] \renewcommand{\arraystretch}{2.75}
\caption{Selected special cases}
\begin{center}
\begin{tabular}{c|ccccccc}
\toprule \vspace{-1.75em} $\kappa$ & 2 & 8 & 8/3 &6 & 3 & 16/3 & 4 \\
\, & (LERW) & (UST) & (SAW?) & (percolation) & (Ising) & (FK Ising) & (GFF) \\
\midrule $b$ & $-\dfrac12$ & $\dfrac12$ & $-\dfrac{\sqrt3}{6}$ & $\dfrac{\sqrt3}{6}$ & $-\dfrac1{\sqrt{24}}$ & $\dfrac1{\sqrt{24}}$ & 0 \\
$a$ & 1 & $\dfrac12$ & $\dfrac{\sqrt3}2$ & $\dfrac1{\sqrt3}$ & $\dfrac{\sqrt2}{\sqrt3}$ & $\dfrac{\sqrt3}{2\sqrt2}$ & $\dfrac1{\sqrt2}$\\
$h$ & 1 & $-\dfrac18$ & $\dfrac58$ & 0 & $\dfrac12$ & $\dfrac1{16}$ & $\dfrac14$ \\
$-\eta$ & $\dfrac12$ & 2 & $\dfrac23$ & $\dfrac32$ & $\dfrac34$ & $\dfrac43$ & 1 \\
\bottomrule
\end{tabular}
\label{table: numerology}
\end{center}
\end{table}

We can now restate Proposition~\ref{singular vec} as follows.

\begin{prop} \label{singular vec2}
A primary field $\OO\in\FF_{(b)}$ of dimension $\lambda$ produces a level two singular vector if and only if $\lambda = h(b)$ or $\lambda = h'(b).$
The corresponding singular vector is
$[L_{-2} +\eta L_{-1}^2]\OO$ or $[L_{-2} +\eta' L_{-1}^2]\OO,$ respectively.
\end{prop}

\begin{proof}
From \eqref{eq: level two singular condition}, we find
$$\eta^2 -\frac{c-13}{6}\eta + 1 = 0,$$
so by \eqref{eq: eta}, $\eta = \eta(b)$ or $\eta(b').$
It follows from \eqref{eq: level two singular condition} and \eqref{eq: eta and h} that
$$\lambda = -\frac12 -\frac3{4\eta(b)} = h(b), \textrm{ or }\lambda = -\frac12 -\frac3{4\eta(b')} = h(b').$$
\end{proof}

\ms Let us apply this description to vertex fields in $\FF_{(b)}.$
We will write $\OO^{(\sigma)}$ for a vertex field with exponent $i\sigma,$ e.g.,
$$\OO^{(\sigma)} = V_\star^{i\sigma}.$$
(We can equally consider non-chiral vertex fields $\VV^{i\sigma},$ bi-vertex chiral fields $V^{i\sigma},$ or the fields $\widehat V_\star^{i\sigma}$ which we define in the next lecture.)
The conformal dimension of $\OO^{(\sigma)}$ is
\begin{equation} \label{eq: conf dim}
\lambda=\frac{\sigma^2}2-\sigma b.
\end{equation}

\ms \SS
We call a primary field (level two) degenerate if it produces a null singular vector.
See Remark in Section~\ref{sec: singular vec}.

\begin{prop} \label{degenerate O}
For each $b,$ there are exactly 4 exponents $\sigma$ (except for the case $b=0,\pm1/2$ when some of $\sigma$'s coincide) such that the vertex fields $\OO^{(\sigma)}$ produce level two singular vectors:
$$\phantom{--b}\sigma = a,\quad \sigma = 2b -a, \quad(\lambda= h(b)),$$
and
$$\sigma = -a-b,\quad \sigma = 3b + a, \quad(\lambda= h'(b)).$$
The vertex fields $\OO^{(a)}$ and $\OO^{(-a-b)}$ are degenerate, but $\OO^{(2b-a)}$ and $\OO^{(a+3b)}$ are not (unless $b=0,\pm1/2$).
\end{prop}

\begin{proof}
The first statement follows from Proposition~\ref{singular vec2} and \eqref{eq: conf dim}.
The second statement follows from Proposition~\ref{degeneracy2} and from the Wiener chaos decomposition of singular vectors
(cf. the proof of Proposition~\ref{degeneracy}) in the following expressions:
\begin{align*}
(L_{-2} + \eta\phantom{'} L_{-1}^2) \OO^{(2b-a)} &=\phantom{1(}\frac{8-\kappa}{2}\phantom{)}\left[A_{(b)}+\phantom{3}\frac{\kappa-6}{2} A_{(0)}\right]\odot \OO^{(2b-a)};\\
(L_{-2} + \eta' L_{-1}^2) \OO^{(3b+a)}&=\frac{4(\kappa-2)}{\kappa}\left[A_{(b)}+\frac{8-3\kappa}{\kappa} A_{(0)}\right]\odot \OO^{(3b+a)},
\end{align*}
where $A_{(0)}$ and $A_{(b)}$ are holomorphic quadratic differentials from the definition of the stress tensor of $\FF_{(0)}$ and $\FF_{(b)}$ theories, see the beginning of Lecture~\ref{ch: modifications}.
\end{proof}

\ms The simplest example of a non-trivial singular vector is the Virasoro field $T$ in the case $c = 0$ (i.e., $\kappa = 8/3$ or $6$).
In this case $T$ is primary and, as a level two singular vector $T$ corresponds to $\OO^{(\sigma=0)}.$
In the case $b=0,$ the differential $J$ appears as a level one current-singular vector.

\begin{rmk*}
Primary fields are of particular importance from the physical point of view. 
Non-zero singular vectors are examples of non-vertex primary fields.
\end{rmk*}

\renewcommand\chaptername{Lecture}
\chapter{Connection to SLE theory}\label{ch: SLE theory}

In this lecture we discuss SLE theory from the point of view of conformal field theory.
Chordal Schramm-Loewner evolution (SLE) equation in $(D,p,q)$ describes conformally invariant random curves from $p$ to $q$ satisfying the so-called domain ``Markov property."
The corresponding probability laws are parametrized by a single parameter $\kappa >0.$
We will establish the following fact and explain some of its consequences: if the numbers $a$ and $b$ are related to $\kappa$ as
$$a = \sqrt{2/\kappa},\qquad b = a(\frac\kappa4-1) = \sqrt{\kappa/8}-\sqrt{2/\kappa}$$
(see Appendix~\ref{appx: numerology}), then under the insertion of a boundary vertex $V_\star^{ia}(p),$ all fields in the theory $\FF_{(b)}$ satisfy the field ``Markov property" with respect to the SLE filtration.

\ms In Section~\ref{sec: Chordal SLE} we recall some basic definitions and facts of SLE theory.
In Sections~\ref{sec: BC changing} - \ref{sec: Cardy} we interpret the insertion of $V_\star^{ia}(p)$ as a ``boundary condition changing operator" and derive Cardy-type equations for correlation functions under this insertion.
In Section~\ref{sec: SLE MOs} we prove the field ``Markov property": all correlation functions
$$\frac{\E[V_\star^{ia}(p) X_1(z_1)\cdots X_n(z_n)]}{\E[V_\star^{ia}(p)]},\qquad (X_j\in\FF_{(b)})$$
are SLE martingale-observables.
In the last section we consider several example of SLE observables.
Further examples, related to vertex fields are discussed in the next lecture.

\ms \section{Chordal SLE} \label{sec: Chordal SLE} \index{SLE}
\ms\SS Let $(D,p,q)$ be a simply connected domain with two marked boundary points and $\kappa$ be a positive parameter.
The chordal Schramm-Loewner evolution with parameter $\kappa,$ $\SLE(\kappa),$ is a conformally invariant law on random curves in $D$ from $p$ to $q.$
It is described by the Loewner equation driven by the standard one-dimensional Brownian motion $B_{\kappa t}.$
More precisely, for each $z \in D,$ let $g_t(z)$ be the solution (which exists up to a time $\tau_z\in(0,\infty]$) of the equation
$$\partial_t g_t(z) = \frac{2}{g_t(z) -\sqrt{\kappa}B_t}, \quad B_0=0,$$
where $g_0:(D, p,q)\to (\H, 0,\infty)$ is a given conformal map.
Then it is known that for all $t,$
$$w_t(z):=g_t(z) -\sqrt{\kappa}B_t$$
is a well-defined conformal map from the domain
\begin{equation} \label{eq: Dt}
D_t := \{z \in D: \tau_z>t\}
\end{equation}
onto the upper half-plane $\mathbb{H},$
where the SLE stopping time $\tau_z$ (the solution $g_t(z)$ of the Loewner equation exists for $t<\tau_z$) satisfies
\begin{equation} \label{eq: SLE stopping time}
\lim_{t\uparrow \tau_z} w_t(z) = 0.
\end{equation}
The SLE curve $\gamma$ is defined by the equation
\begin{equation} \label{eq: SLE curve}
\gamma_t \equiv \gamma(t) := \lim_{z\to0} g_t^{-1}(z + \sqrt{\kappa}B_t);
\end{equation}
the limit exists for all $t$ almost surely.
Also almost surely, the SLE curve is a continuous path $\gamma:[0,\infty) \to \overline{D}$
(assuming local connectivity of $\pa D$)
such that $D_t$ is the unbounded component of $D\setminus\gamma[0,t]$ for all $t\ge 0.$
The SLE path $\gamma$ is simple for $\kappa\in[0,4]$;
self-intersecting but non-self-crossing for $\kappa\in(4,8)$; and
space-filling for $\kappa\ge8.$
For the proof of these facts and other basic properties of SLE, see \cite{RS05}, and also \cite{Lawler05}, \cite{Lawler09}, \cite{Werner04}, \cite{Werner09}.
Beffara proved that the Hausdorff dimension of $\SLE(\kappa)$ trace is almost surely
$\operatorname{min}(1+\kappa/8,2),$ see \cite{Beffara08}.

\begin{figure}[ht]
\begin{center}
\begin{minipage}{.4\textwidth}
\vspace{0pt}
\centering
\includegraphics[width=\textwidth]{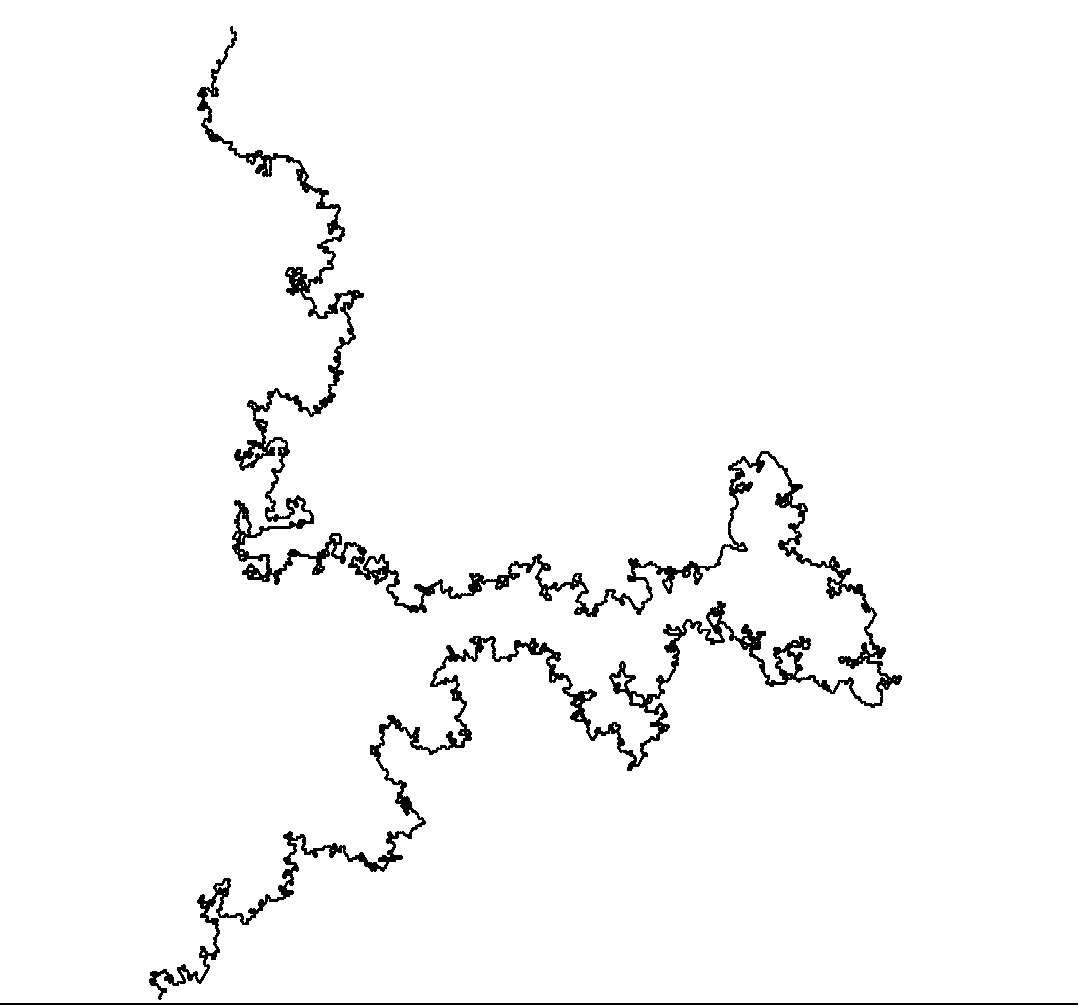} \\
$\kappa\le4$
\end{minipage}
\begin{minipage}{.1\textwidth}
\quad
\end{minipage}
\begin{minipage}{.4\textwidth}
\vspace{0pt}
\centering
\includegraphics[width=\textwidth]{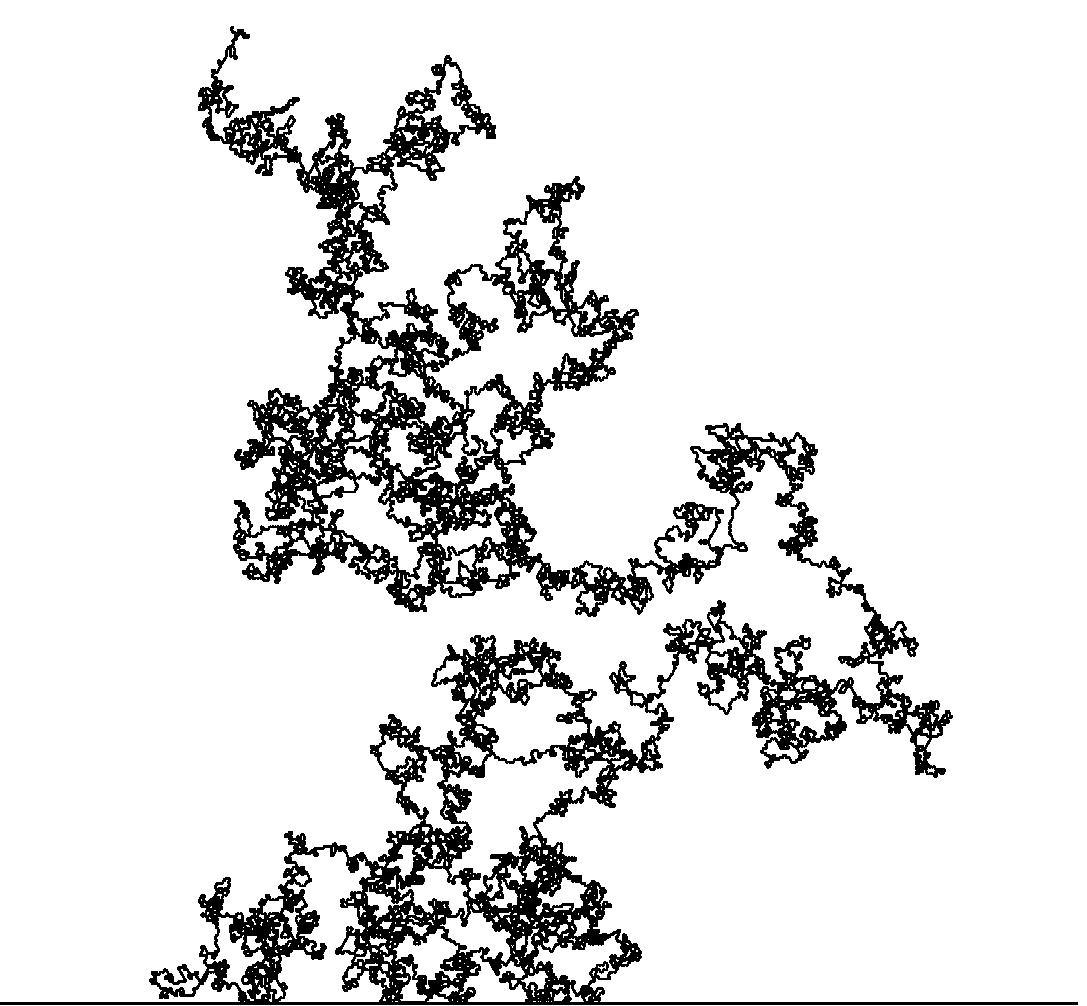} \\
$4<\kappa<8$
\end{minipage}
\end{center}
\caption{The phases of SLE; from \cite{Kang07a}}
\label{figure: phase}
\end{figure}

\ms \subsec{Domain ``Markov property"} \index{domain Markov property}
If $\kappa=0,$ the (non-random) curve $\gamma$ is just the hyperbolic geodesic from $p$ to $q$ in $D.$
It satisfies the equation
$$\gamma[t,\infty)=\gamma_{D_t,\gamma_t,q}[0,\infty).$$
If $\kappa>0,$ then the same equation holds for the laws on random curves, which is a direct consequence of the SLE construction.
Alternatively, one can express this property by the equation
\begin{equation} \label{eq: domain Markov}
\mathrm{Law}~\left( \gamma[t,\infty)~|~\gamma[0, t]\right)~=~\mathrm{Law}~\gamma_{D_t,\gamma_t,q}[0,\infty).
\end{equation}
(One should understand this and similar statements in the sense of conditional expectations with respect to the filtration by the Brownian motion in the definition of SLE.)
Schramm's principle states that $\SLE(\kappa)$ are the only conformally invariant laws on non-self-crossing curves satisfying \eqref{eq: domain Markov}.

\ms \subsec{Field ``Markov property"} \index{field Markov property}
Again we first look at the ``classical" case $\kappa=0.$
Consider the (non-random) field
$$f(z)\equiv f_{D,p,q}(z)=\arg \psi'(z),\qquad \psi:(D,p,q)\to(\mathbb{C}\sm \R_+, 0,\infty).$$
(Note that this field is a pre-pre-Schwarzian form.)
Since $\R_-$ is the hyperbolic geodesic in $\mathbb{C}\sm \R_+,$ it is clear that the field $f$ has the (``Markov") property
$$f\big|_{D_t}~=~f_t\equiv f_{D_t, \gamma_t,q}.$$
Moreover, many other fields, e.g., $f^2,$ $e^f,$ $\pa f,$ etc., will have the same property with respect to $\SLE(0).$

\ms In the case $\kappa>0,$ we want to define a similar property of a collection $\FF=\{F_j\}$ of \emph{random} conformal Fock space fields:
\begin{equation} \label{eq: Field Markov}
\mathrm{Law}~\left( \FF~|~\gamma[0, t]\right)~=~\mathrm{Law}~\FF_t,\qquad \FF_t:=\FF_{D_t,\gamma_t,q}.
\end{equation}
(Here we assume invariance of $\FF$ with respect to $\Aut(D, p,q),$ and define the fields $F_t$ as explained in Section~\ref{sec: conf inv}. 
While the field $X_t$ in Section~\ref{sec: Lie} means the pull-back of $X$ with respect to the local flow, $F_t$ in this Lecture indicates the field in the SLE triple $(D_t, \gamma_t,q).$
See Section~\ref{sec: SLE MOs}.)
On the level of correlation functions, the equation \eqref{eq: Field Markov} should mean
\begin{equation} \label{eq: SLE mart1}
``\E\left[F_1(z_1)\cdots F_n(z_n)~|~\gamma[0, t]\,\right]"=\E\left[F_{1t}(z_1)\cdots F_{nt}(z_n)\right],
\end{equation}
but in order to interpret the left-hand side we need to have both random fields and SLE curves be defined on the same probability space.
One way to proceed is to couple $\SLE(\kappa)$ and the Gaussian free field, see \cite{SS10} and \cite{Dubedat09}, but instead of going into the analytic details of such a coupling we just note that if this coupling is defined properly, then the processes
\begin{equation} \label{eq: SLE mart2}
f_t(z_1,\cdots, z_n)= \E\left[F_{1t}(z_1)\cdots F_{nt}(z_n)\right]
\end{equation}
are local SLE martingales, and take this last property as a definition.

\ms \subsec{Martingale-observables}
A stochastic process $M_t$ is called a martingale with respect to a filtration $\AA_t$ (an increasing family of $\sigma$-algebras, e.g., the $\sigma$-algebras generated by the Brownian motion up to time $t$) if $M_t$ is $\AA_t$-measurable for all $t,$ $\E\,|M_t|<\infty$ for all $t$ and if
$$\E\,[M_t\,|\,\AA_s] = M_s\quad\textrm{for all } t\ge s.$$
For an $L^1$ random variable $M,$ the process $M_t = \E[M\,|\,\AA_t]$ is a martingale.
Thus if $F_{1}(z_1)\cdots F_{n}(z_n)$ in \eqref{eq: SLE mart1} could be replaced by an $L^1$ random variable, then the processes \eqref{eq: SLE mart2} would be martingales.

\ms We refer to any textbook on stochastic calculus, e.g., \cite{RY99} for the definition of local martingales.
In particular, for a smooth function $h,$ the stochastic integral
$$\int_0^t h(B_s)\,dB_s$$
is a local martingale.
In this respect, recall It\^o's formula.

\ms \textbf{It\^o's formula.} \index{It\^o's formula} \emph{If $f$ is in $C^{1,2},$ then almost surely}
\begin{equation} \label{eq: Ito's Formula}
f(t,B_t) -f(0,B_0) = \int_0^t f'(s,B_s)\,dB_s + \int_0^t \dot f(s,B_s)\,ds + \frac{1}{2} \int_0^t f''(s,B_s)\,ds.
\end{equation}
The term 
$$\int_0^t \dot f(s,B_s)\,ds + \frac12 \int_0^t f''(s,B_s)\,ds$$
is called the drift term of  $f(t,B_t)$ and the process $f(t,B_t)$ is called a local martingale if its drift term vanishes.

\ms By definition, a collection $\FF$ of fields has the ``Markov property" with respect to $\SLE(\kappa)$ if for all $F_j\in\FF$ and all $z_j\in D,$ the processes \eqref{eq: SLE mart2} are local martingales.
We say that the non-random fields $f(z_1,\cdots,z_n)=\E\left[F_{1}(z_1)\cdots F_{n}(z_n)\right]$ are $\SLE(\kappa)$ \emph{martingale-observables}. \index{martingale-observables}

\ms It is easy to verify by It\^o's calculus that any particular correlation function is a martingale-observable, but our goal is to describe a large collection of SLE observables by means of conformal field theory (Ward's and level two degeneracy equations).

\ms \section{Boundary condition changing operators} \label{sec: BC changing}
We use this term for the correspondence $\XX \mapsto \widehat\XX$ resulting from the insertion of a chiral bi-vertex field with endpoints on the boundary.
This operation changes the boundary values of Fock space fields.
The term is borrowed from physics, see e.g., \cite{Cardy05}, and in many (but not all) cases we have a good match with the physical formulas (as we understand them).

\bs\SS Let us recall the set-up:
{\setlength{\leftmargini}{2.0em}
\begin{itemize}
\ms\item $(D,q)$ is a simply connected domain $D$ with a marked boundary point $q\in\pa D;$
\ss\item $b\in\R$ is a fixed parameter and $\FF_{(b)}=\FF_{(b)}(D,q)$ is the OPE family of the bosonic field $\Phi_{(b)},$ see Section~\ref{sec: c}; the notation for the standard fields $\Phi,J,T,V^{i\sigma},$ etc. refers to the family in $\FF_{(b)};$
\ms\item the vertex field $V_\star^{ia}$ with $a>0,2a(a+b)=1,$ rooted at the modification point $q,$ is a holomorphic differential of dimension
\begin{equation} \label{eq: h12}
h:=\frac{a^2}2-ab,
\end{equation}
with respect to $z$ and a boundary differential with respect to $q,$ see Section~\ref{sec: V*};
\ss\item we normalize this field in a \emph{fixed} boundary chart at $q;$
\ms\item the vertex field $V_\star^{ia}$ produces a degenerate singular vector:
$$T*V_\star^{ia} = \frac1{2a^2}\pa^2V_\star^{ia},$$ see Section~\ref{sec: V*}.
\end{itemize}}
\ms Let now $p\in\pa D,$ $p\ne q,$ and denote by $\tau$ the arc of $\pa D$ from $q$ to $p$ oriented in the counterclockwise direction.
We use $\tau$ to define the value $V_\star^{ia}(p)=V_\star^{ia}(p;\tau).$
In the half-plane uniformization consistent with the fixed boundary chart at $q=\infty,$ $p=\xi\in\R,$ we have
$$V_\star^{ia}(\xi)=e^{\odot ia\Phichiral_{(0)}[\tau]},$$
where $\tau$ is the half-line $(-\infty, \xi).$

\bs\SS The insertion of $V_\star^{ia}(p)$ is an operator
$$\XX\mapsto\widehat\XX$$
on Fock space functionals/fields.
By definition, this correspondence is given by the formula
\begin{equation} \label{eq: BC} \index{bosonic field $\Phi_{(b)}, \widehat\Phi$}
\widehat\Phi=\Phi+2ia\Gchiral(p,z)
\end{equation}
and the rules
\begin{equation} \label{eq: BCrules}
\pa\XX\mapsto\pa\widehat\XX,\qquad \bp\XX\mapsto\bp\widehat\XX,\qquad\alpha\XX+\beta\YY\mapsto \alpha\widehat\XX+\beta\widehat\YY,\qquad \XX\odot\YY\mapsto\widehat\XX\odot\widehat\YY.
\end{equation}
If $w:(D, p,q)\to (\H, 0,\infty)$ is a conformal map, then
$$2ia\Gchiral(p,z)=2a\arg w(z).$$

\ms \textbf{Notation.} We denote by $\widehat \FF_{(b)}$ the image of $\FF_{(b)}$ under this correspondence.

\ms Fields in $\widehat \FF_{(b)}$ are $\Aut(D,p,q)$-invariant because $\arg w$ is $\Aut(D,p,q)$-invariant and fields in $\FF_{(b)}$ are invariant with respect to $\Aut(D,q).$

\ms Denote
$$\widehat \E[\XX]:=\frac{\E [V_\star^{ia}(p)\XX]}{\E [V_\star^{ia}(p)]}=\E [e^{\odot ia \Phichiral_{(0)}(p)}\XX],$$
see Section~\ref{sec: insert} for the motivation of this notation.
As in Section~\ref{sec: insert}, we prove by induction the following:
\ms \begin{prop}
Let $\widehat\XX\in \widehat \FF_{(b)}$ correspond to the string $\XX\in \FF_{(b)}$ under the map given by \eqref{eq: BC} and \eqref{eq: BCrules}.
Then
\begin{equation} \label{eq: hat E=E hat}
\widehat \E[\XX]=\E[\widehat\XX].
\end{equation}
\end{prop}

\ms\SS\textbf{Examples:} \quad
\renewcommand{\theenumi}{\alph{enumi}}
{\setlength{\leftmargini}{2.0em}
\begin{enumerate}
\ms \item The current $\widehat J$ is a pre-Schwarzian form of order $ib,$ \index{current!field $J,\widehat J$}
$$\widehat J = J -ia\dfrac{w'}{w} = J_{(0)} -ia\dfrac{w'}{w} + ib\dfrac{w''}{w'}.$$
In the $(\H,0,\infty)$-uniformization, $\hat\jmath(z):=\widehat \E\, J(z)=-\dfrac{ia}z;$
\ms\item The Virasoro field $\widehat T$ is a Schwarzian form of order $c/{12},$ \index{Virasoro!field $T, \widehat T$}
\begin{align*}
\widehat T &= -\dfrac12\widehat J*\widehat J + ib\pa\widehat J = T + ia\frac{w'}{w}J_{(0)}+ h\Big(\dfrac{w'}{w}\Big)^2\\
&= A_{(0)} -\hat\jmath J_{(0)}+ib\pa J_{(0)}+ \dfrac{c}{12}Sw + h\Big(\dfrac{w'}{w}\Big)^2,
\end{align*}
where $h=a^2/2-ab$ is the conformal dimension of $V_\star^{ia},$ see \eqref{eq: h12}.
In the $(\H,0,\infty)$-uniformization, $\widehat \E\, T(z)=h\dfrac1{z^2};$
\ms\item The non-chiral vertex field $\widehat\VV^\alpha$ is a differential of conformal dimensions $(-\alpha^2/2+i\alpha b,-\alpha^2/2-i\alpha b),$ \index{vertex field!non-chiral $\VV,\widehat\VV$}
\begin{equation} \label{eq: hat VV}
\widehat\VV^\alpha=e^{2\alpha a\arg w}\VV^\alpha = e^{2\alpha a\arg w -2\alpha b\arg w'}C^{\alpha^2}e^{\odot\alpha\Phi_{(0)}}.
\end{equation}
In the $(\H,0,\infty)$-uniformization,
$\widehat \E\, \VV^\alpha=(2y)^{\alpha^2}e^{2\alpha a\arg z};$
\ms\item The bi-vertex field $\widehat V^\alpha(z, z_0)$ is a $-\alpha^2/2\pm i\alpha b$-differential in both variables, \index{vertex field!chiral bi- $V,\widehat V$} \index{chiral!bi-vertex field $V,\widehat V$}
$$\widehat V^\alpha(z, z_0)=
\left(\dfrac{w'(z)w'(z_0)}{(w(z)-w(z_0))^2}\right)^{-\alpha^2/2}\,\left(\dfrac{w'(z)}{w'(z_0)}\right)^{i\alpha b}\,\left(\dfrac{w(z)}{w(z_0)}\right)^{-i\alpha a}\,e^{\odot \alpha\Phichiral_{(0)}(z,z_0)}.$$
In the $(\H,0,\infty)$-uniformization,
$\widehat \E \,V^\alpha(z, z_0)=(z-z_0)^{\alpha^2}z^{-i\alpha a} z_0^{i\alpha a}.$
\end{enumerate}}

\ms \section{Cardy's equations} \label{sec: Cardy} \index{Cardy's equations}
In this section we will derive equations for correlation functions of fields in $\widehat\FF_{(b)}.$
These equations are similar to those in Proposition~\ref{BPZ eqs}.
As usual, we will state them in the case $(D,q) = (\H,\infty).$

\ms For $\xi\in\R$ and the tensor product
$$X=X_1(z_1)\cdots X_n(z_n),$$
of fields $X_j\in\FF_{(b)}$ ($z_j\in\H$), we denote
$$\widehat\E_\xi X=\E [e^{\odot ia \Phichiral_{(0)}(\xi)}X],$$
so $\widehat\E X=\widehat{\E}_\xi X\big|_{\xi=0}.$

\begin{prop} \label{PDE's4bdry}
If $2a(a+b)=1,$ then in the identity chart of $\H,$ we have
\begin{equation} \label{eq: PDE's4bdry}
\frac1{2a^2}\partial^2_\xi \widehat\E_\xi X= \widehat\E_\xi[\LL_{v_\xi}X],\qquad v_\xi(z):=\frac1{\xi-z},
\end{equation}
where $\pa_\xi=\pa+\bar\pa$ is the operator of differentiation with respect to the real variable $\xi.$
\end{prop}

\begin{proof} Denote
$$R_\xi\equiv R_\xi(z_1,\cdots, z_n)=\widehat{\E}_\xi X.$$
Since $R_\xi$ does not depend on the boundary chart (at $\infty$) in which we normalize $V_\star^{ia},$ we can assume $V_\star^{ia}(\xi)=e^{\odot ia \Phichiral_{(0)}[\tau]},$ where $\tau$ is the half-line $(-\infty, \xi),$ and therefore
$$R_\xi=\E\left[V_\star^{ia}(\xi)X\right].$$
Denote
$$R_z\equiv R(z;z_1,\cdots, z_n)=\E\left[V_\star^{ia}(z)X\right],\qquad z\in\H,$$
where for $z$ close to $\xi$ we use a path from $\infty$ to $z$ close to $\tau$ so that $R_\xi=\lim_{z\to\xi}R_z.$
By the level two degeneracy equation (Proposition~\ref{TVd2V}), we have
$$\frac1{2a^2}\pa^2\E\left[V_\star^{ia}(z)X\right]= \E\left[(T*V_\star^{ia})(z)X\right].$$

\ms Applying Ward's equation (Proposition~\ref{Ward4VX}) to the right-hand side (we can apply this because $T=A=T_{(0)} + ib\pa J_{(0)}$ satisfies the conditions in Proposition~\ref{represent A}), we conclude
\begin{align*}
\frac1{2a^2}\pa^2\E\left[V_\star^{ia}(z)X\right]&=\E\left[ V_\star^{ia}(z)\LL^+_{v_z} X\right]+ \E\left[\LL^-_{v_{\bar z}}\left( V_\star^{ia}(z) X\right)\right] \\&= \E\left[ V_\star^{ia}(z)\LL^+_{v_z} X\right]+ \E\left[V_\star^{ia}(z)\LL^-_{v_{\bar z}} X\right],
\end{align*}
where we use Leibniz's rule and the fact that $\LL^-_{v}V_\star^{ia}(z)=0$ (because $V_\star^{ia}$ is a holomorphic differential).

\ms Let us now take the limit $z\to\xi.$
Since $\pa_\xi=\pa+\bar\pa$ and the field $V_\star^{ia}$ is holomorphic, the $\pa^2$-derivative in the left-hand side converges to $\partial^2_\xi R_\xi.$
On the other hand, since $\xi=\bar\xi,$ the right-hand side converges to
$$\E\left[ V_\star^{ia}(\xi)\LL_{v_\xi} X\right]=\widehat\E_\xi[\LL_{v_\xi}X].$$
\end{proof}

\begin{cor*} We have 
\begin{equation} \label{eq: BPZ-C eq}
[(\sum\pa_j+\bp_j)^2-2a^2\LL_{v_{_0}}]\widehat\E\,X=0.
\end{equation}
\end{cor*}

\ms \begin{proof}
We will write $R$ for $R_{\xi=0}.$
By translation invariance,
$$R_\xi = R(z_1-\xi,\cdots,z_n-\xi)$$
and therefore
$$\pa_\xi\big|_{\xi=0}R_\xi = -\sum\pa_j R -\sum\bp_j R = -\sum\pa_{x_j}R.$$
\end{proof}

\ms We will refer to equations \eqref{eq: BPZ-C eq} as Cardy's equations, cf. \cite{Cardy84}; they are Cardy's boundary version of BPZ-type equations in Proposition~\ref{BPZ eqs}.

\ms \textbf{Examples.} (a) If $X_j$'s are differentials of conformal dimensions $(\lambda_j,\lambda_{*j})$ (e.g., $X_j$'s are vertex fields), then the Lie derivative $\LL_{v_{_0}}$ acts on $\widehat\E X$ as a differential operator
$$\LL_{v_{_0}}\widehat\E X= \sum\Big(-\frac{\pa_j}{z_j}+\frac{\lambda_j}{z_j^2}-\frac{\bp_j}{\bar z_j}+\frac{\lambda_{*j}}{\bar z_j^2}\Big)~\widehat\E X,$$
see \eqref{eq: Lie4diff}, and Cardy's equation~\eqref{eq: BPZ-C eq} is a linear 2nd order PDE:
$$\sum\Big(-\frac{\pa_j}{z_j}+\frac{\lambda_j}{z_j^2}-\frac{\bp_j}{\bar z_j}+\frac{\lambda_{*j}}{\bar z_j^2}\Big)~\widehat\E X =\frac1{2a^2}(\sum\pa_j+\bp_j)^2\widehat\E X.$$

\ms (b) The 1-point function $R(z) = \E\,\widehat T(z)$ is a Schwarzian form of order $c/12,$ see Section~\ref{sec: BC changing}.
Recall that
$$\LL_v X = (v\pa + 2v')X + \mu v'''$$
for a Schwarzian form $X$ of order $\mu,$ see \eqref{eq: Lie4S-form}.
Thus
$$\LL_{v_{_0}}R= \Big(-\frac{\pa}{z}+\frac{2}{z^2}\Big)R + \frac{c/2}{z^4},$$
and Cardy's equation~\eqref{eq: BPZ-C eq} is
$$\pa^2 R = 2a^2 \LL_{v_{_0}}R.$$
Since $R(z) = h/z^2,$ we have the identity
$$6h = 2a^2\big(4h + \frac c2\big).$$
One can directly check this identity from $h = a^2/2-ab, c= 1-12b^2,$ and $2a(a+b)=1.$

\ms (c) The last example can be generalized to the $n$-point function of $\widehat T,$
$$R(z_1\cdots,z_n) = \E[\widehat T(z_1) \cdots \widehat T(z_n) \,\|\,\id_\H].$$
Denote $\bfs{z} = (z_1,\cdots,z_n),$ $\bfs{z}_j = (z_1,\cdots,\hat z_j,\cdots,z_n).$
By \eqref{eq: Lie4S-form} and Leibniz's rule (Proposition~\ref{Leibnitz4strings}), we have
$$\LL_{v_{_0}}R=\sum_{j=1}^n\Big(-\frac{\pa_j}{z_j}+\frac2{z_j^2}\Big)R(\bfs{z}) + \frac{c}2 \sum_{j=1}^n\frac{R(\bfs{z}_j)}{z_j^4}.$$
Cardy's equation~\eqref{eq: BPZ-C eq} gives us the following recursive formula:
\begin{equation} \label{eq: Cardy FW}
\frac1{2a^2}\big(\sum\partial_j\big)^2 R(\bfs{z}) = \sum_{j=1}^n\Big(-\frac{\pa_j}{z_j}+\frac2{z_j^2}\Big)R(\bfs{z}) + \frac{c}2 \sum_{j=1}^n\frac{R(\bfs{z}_j)}{z_j^4}.
\end{equation}
See \cite{FW02} for this equation in the case $c=0.$

\ms \section{SLE martingale-observables} \label{sec: SLE MOs} \index{martingale-observables}
\ms\SS We defined martingale-observables in Section~\ref{sec: Chordal SLE}.
Let us discuss this definition in more detail.
Suppose $M$ is a non-random field of $n$ variables in the half-plane, and suppose that $M$ is invariant with respect to $\Aut(\H,0,\infty).$
As we explained in Section~\ref{sec: conf inv}, conformal invariance allows us to define $M$ for any triple $(D,p,q):$
$$(M_{D,p,q}\,\|\,\id) = (M\,\|\,w^{-1}),$$
where $w$ is a conformal map from $(D,p,q)$ to $(\H,0,\infty),$
so we can think of $M$ as a function or as a collection $\{M_{D,p,q}\}$ of fields.

\ms Consider now chordal SLE curve $\gamma$ in $D$ from $p$ to $q$ so we have a family of domains $(D_t,\gamma_t,q)$ with marked boundary points, see \eqref{eq: Dt} and \eqref{eq: SLE curve}.
A non-random field $M$ is a martingale-observable if for any $z_1,\cdots ,z_n\in D,$ the process
\begin{equation}\label{eq: MO}
M_t(z_1,\cdots, z_n)=M_{D_t,\gamma_t,q}(z_1,\cdots, z_n)
\end{equation}
(stopped when any $z_j\not\in D_t$) is a local martingale (on SLE probability space).
It is important that we compute $M_t(z_1,\cdots,z_n)$ in \eqref{eq: MO} in local coordinates chart that do not change with $t.$
For instance we can use the identity chart of $D,$ and then for $(\lambda,0)$-differentials, we have
$$M_t(z) = (w_t'(z))^\lambda M(w_t(z)).$$
Similarly, if $M$ is a Schwarzian form of order $\mu,$ then
$$M_t(z) = (w_t'(z))^2M(w_t(z)) + \mu S_{w_t}(z).$$
To verify the local martingale condition, it is enough to check that the stochastic differential $dM_t$ has no drift (i.e., no $dt$-term).

\begin{eg*}
The simplest example of an SLE martingale-observable is the $1$-point function of the bosonic field in the case $\kappa=4,$
$$M(z) = \widehat\E[\Phi_{(0)}(z)] = \sqrt2 \arg w(z).$$
We have
$$d\arg w_t(z) = -\sqrt\kappa\,\Im\frac1{w_t}\,dB_t + \big(2-\frac\kappa2\big)\Im\frac1{w_t^2}\,dt,$$
so the drift disappears if $\kappa = 4.$
\end{eg*}

\ms\SS Special cases of the following statement appeared in \cite{BB03} and \cite{RBGW07}.

\begin{prop} \label{MO}
If $X_j\in \FF_{(b)},$ then the non-random fields
$$M(z_1,\cdots, z_n)=\widehat\E [X_1(z_1)\cdots X_n(z_n)]$$
are martingale-observables for $\SLE(\kappa).$
\end{prop}

\ss \begin{proof}
Let $g_t$ be the SLE conformal maps, $g_t(\gamma_t)=\xi_t,$ and $\xi_t=\sqrt\kappa B_t.$
Denote
$$R_\xi(z_1,\cdots, z_n)\equiv\widehat\E_\xi[X_1(z_1)\cdots X_n(z_n)]=\E[e^{\odot ia\Phichiral_{(0)}(\xi)}X_1(z_1)\cdots X_n(z_n)].$$
Then
$$M_t = m(\xi_t,t), \qquad m(\xi,t) =\big(R_\xi\,\|\,g_t^{-1}\big).$$
Note that the function $m(\xi,t)$ is smooth in both variables.
By It\^o's formula we have
$$dM_t=\pa_\xi\Big|_{\xi=\xi_t}~m(\xi,t)\,d\xi_t+\frac\kappa2\pa^2_\xi\Big|_{\xi=\xi_t}~m(\xi,t)\,dt+L_t\, dt,$$
where
$$
L_t:=\frac d{ds}\Big|_{s=0}\big(R_{\xi_t}\,\|\,g_{t+s}^{-1}\big)
=\frac d{ds}\Big|_{s=0}\big(R_{\xi_t}\,\|\,g_t^{-1}\circ f_{s,t}^{-1}\big).
$$
The time-dependent flow $f_{s,t} = g_{t+s}\circ g_t^{-1}$ satisfies
$$\frac{d}{ds}f_{s,t}(\zeta) = \frac{2}{f_{s,t}(\zeta)-\xi_{t+s}}$$
or
$$ f_{s,t} = \id -2sv_{\xi_t} + o(s)\qquad(\textrm{as } s\to0),$$
where $v_{\xi}(z)=1/(\xi-z).$
Since the fields in $\FF_{(b)}$ depend smoothly on local charts, it follows from \eqref{eq: Lie} that
$$L_t=\big(\LL_vR_{\xi_t}\,\|\,g_{t}^{-1}\big),$$
where $v=-2v_{\xi_t}.$
By Ward's equation~\eqref{eq: PDE's4bdry}, we get
$$L_t=-2\big(\LL_{v_{\xi_t}}R_{\xi_t}\,\|\,g_{t}^{-1}\big)=-\frac1{a^2}\big(\pa^2_\xi\Big|_{\xi=\xi_t} R_{\xi}\,\|\,g_{t}^{-1}\big).$$
Thus the drift term of $dM_t$ vanishes.
\end{proof}

\begin{rmk*}
If we insert the degenerate vertex field $V_{\star,(b)}^{-i(a+b)}$ (see Proposition~\ref{degenerate O}) instead of $V_{\star,(b)}^{ia},$ then we get martingale-observables for the dual SLE theory, i.e., $\SLE(\kappa'), (\kappa\kappa'=16).$
\end{rmk*}

\section{Examples} \label{sec: eg MO}

\ms \SS\textbf{Example 1 (Schramm-Sheffield's observables)} \index{martingale-observables!Schramm-Sheffield's} \index{Schramm-Sheffield's observables}
The 1-point functions of the bosonic fields
$$\widehat\varphi(z) = \widehat\E[\Phi_{D,p,q}(z)] = 2a\arg w(z) -2b\arg w'(z), \quad w:(D,p,q)\to(\H,0,\infty),$$
were introduced as SLE martingale-observables by Schramm and Sheffield, see \cite{SS10}.
By It\^o's calculus,
\begin{align*}
\widehat\varphi_t(z) &= \widehat\E[\Phi_{D_t,\gamma_t,q}(z)]= 2a\arg w_t(z) -2b\arg w_t'(z) \\
&= 2a\arg w(z) -2b\arg w'(z) -2\sqrt2\int_0^t\Im \frac1{w_s(z)}~dB_s.
\end{align*}
The fact that the $2$-point functions
$$ \widehat\E[\Phi(z_1)\Phi(z_2)]=2G(z_1,z_2)+\widehat\varphi(z_1)\widehat\varphi(z_2)$$
are martingale-observables is essentially equivalent to the following special case of Hadamard's variation formula
\begin{equation} \label{eq: Hadamard}
d G_{D_t}(z_1,z_2)=-4\Im\frac1 {w_t(z_1)}\Im\frac1 {w_t(z_2)}\,dt=-\frac12\,d\langle \widehat\varphi(z_1),\widehat\varphi(z_2)\rangle_t.
\end{equation}

Schramm and Sheffield used \eqref{eq: Hadamard} to construct a coupling of SLE and the Gaussian free field such that
$$\E[\,\widehat\Phi_{D,p,q}\,|\,\gamma[0,t]\,]=\widehat\Phi_{D_t,\gamma_t,q}.$$

\ms Let us outline the main idea and explain how  Schramm-Sheffield's coupling is related to the fact that all $n$-point functions
$$M(z_1,\cdots, z_n)=\widehat\E [\Phi(z_1)\cdots\Phi(z_n)]$$
are martingale-observables.
For simplicity, we consider the case $\kappa\le4$ (then SLE curves are Jordan, and for all $z\in D,$ $z\in D_t$ almost surely).
For a fixed $t,$ we define a random field $\Psi_t$ in $D$ as follows.
Let $\GFF(D_t)$ denote the Gaussian free field in $D_t$  independent of SLE (e.g., consider the pull back of $\GFF(\H)$ by some conformal map from $D_t$ to $\H$)
and set
$$\Psi_t = \widehat\varphi_t + \GFF(D_t).$$
It is in fact easy to define $\Psi_t$ as a distributional field in $D,$ so that the probability space of $\Psi_t$ is the product of probability spaces of SLE and GFF.

\begin{claim*}
For all t, the correlation functions of $\Psi_t$ are identical to those of $\widehat\Phi.$
\end{claim*}

\begin{proof}
We first show that
$$\E[\Psi_t(z_1)\cdots\Psi_t(z_n)] = \E_{_\SLE}\, M_t(z_1,\cdots,z_n)$$
by applying Wick's calculus to the GFF component of
$\E = \E_{_\SLE}\otimes\E_{_\GFF}.$
Then we verify that $M_t$ is a global martingale (this requires some simple estimates from complex analysis and stochastic calculus).
It follows that
$$\E_{_\SLE}\, M_t(z_1,\cdots,z_n) = \, M_0(z_1,\cdots,z_n) = \E [\widehat\Phi(z_1)\cdots\widehat\Phi(z_n)].$$
\end{proof}

See \cite{SS10} for the version of this statement in the case $4<\kappa<8,$ and for the limiting case $t=\infty.$
A more subtle question of uniqueness of SLE/GFF coupling was settled by Dub\'edat~\cite{Dubedat09}.

\ms \SS\textbf{Example 2 (Friedrich-Werner's formula)} \index{martingale-observables!Friedrich-Werner's formula} \index{Friedrich-Werner's formula}
Let us apply Ward's equations to the function
$$\E\,[\,\widehat T(z_1)\cdots \widehat T(z_n)\,\|\,\id_\H\,] = \E\,[\,V_\star^{ia}(0)T(z_1)\cdots T(z_n)\,\|\,\id_\H\,]$$
by replacing one of $T(z_j)$'s in the right-hand side by the corresponding Lie derivatives, see Propositions~\ref{Ward identity} and \ref{represent A}.
(As usual, $\id_\H$ is the identity chart in the upper half-plane $\H.$)
Denote $\bfs{z} = (z_1,\cdots,z_n),$ $\bfs{z}_j = (z_1,\cdots,\hat z_j\cdots,z_n),$ and
$$R(\xi;\bfs{z}) = \E\,[\,V_\star^{ia}(\xi)\,T(z_1)\cdots T(z_n)\,];$$
this non-random field is a boundary differential in $\xi.$
Then we have
\begin{equation} \label{eq: recursion4T}
R(\xi;z,\bfs{z})=\LL_{v_z}R(\xi;\bfs{z}),\quad (\textrm{in }\id_\H),
\end{equation}
where $v_z(\zeta) = 1/(z-\zeta).$
In particular, at $\xi=0,$ setting $R(z,\bfs{z})\equiv R(0;z,\bfs{z}),$ we get a recursive formula
$$R(z,\bfs{z})= \frac{h}{z^2}R(\bfs{z})+ \sum_{j=1}^n\Big(\big(-\frac1z+\frac1{z-z_j}\big)\pa_j+\frac2{(z-z_j)^2}\Big)R(\bfs{z}) +\frac{c}2 \sum_{j=1}^n\frac{R(\bfs{z}_j)}{(z-z_j)^4}.$$
In the case $\kappa = 8/3, (c=0, h = 5/8)$ and $z_j = x_j \in \R,$ this equation coincides with Friedrich-Werner's formula (\cite{FW03}) for the function

\ms
$$B(\bfs{x}) = \lim_{\ve\to0}\frac{\P(\textrm{SLE}(8/3)\textrm{ hits all }[x_j,x_j+i\ve\sqrt2])} {\ve^{2n}}.$$

\ms

(Also, the equation~\eqref{eq: Cardy FW} coincides with their ``dynamical" formula in \cite{FW02}.)
Since $B = R = 1$ for $n=0,$ we conclude that
$$B(x_1,\cdots,x_n)=\E\,[\,\widehat T(x_1)\cdots \widehat T(x_n)\,\|\,\id_\H\,].$$
One can in fact interpret the argument in \cite{FW03} in terms of Lie derivatives and directly relate it to the equation \eqref{eq: recursion4T}.

\ms We will use the \emph{restriction property} \index{restriction property} of SLE(8/3); see \cite{LSW03} or Example 3 below:
{\setlength{\leftmargini}{2.0em}
\begin{itemize}
\ms \item the law of SLE(8/3) in $\H$ conditioned to avoid a fixed hull $K$ is identical to the law of SLE(8/3) in $\H\sm K;$
\ms \item equivalently, there exists $\lambda$ such that for all $K,$
$$\P(\SLE(8/3) \textrm{ avoids }K) = [\Psi_K'(0)]^\lambda,$$
where $\Psi_K$ is the conformal map $(\H\sm K,0,\infty)\to (\H,0,\infty)$ satisfying $\Psi_K'(\infty)=1.$
(The restriction exponent $\lambda$ of $\SLE(8/3)$ is equal to $5/8.$)
\end{itemize}}

\ms Define the non-random field $A(\xi,x_1,\cdots,x_n)$ of $n+1$ variables as follows:
{\setlength{\leftmargini}{2.0em}
\begin{itemize}
\ms \item $A$ is a boundary differential of conformal dimension $\lambda$ with respect to $\xi$ and of conformal dimension $2$ with respect to $x_j$;
\ms \item $(A(\xi,x_1,\cdots,x_n)\,\|\,\id_\H) = B(x_1-\xi,\cdots,x_n-\xi).$
\end{itemize}}

\ms \begin{claim*} We have 
$$A(0;x,\bfs{x}) = \LL_{v_x} A(0;\bfs{x}).$$
\end{claim*}
(So $B$ and $R$ satisfy the same recursive equation and are therefore equal.)

\ms \begin{proof}[Proof of Claim]
Denote $\bfs{x} = (x_1,\cdots,x_n).$ We write $\P(\bfs{x})$ for the probability that SLE(8/3) path hits all segments $[x_j,x_j+i\ve\sqrt2]$ and $\P(\bfs{x}\,|\,\neg x)$ for the same probability conditioned on the event that the path avoids $[x,x+i\ve\sqrt2].$
By construction,
\begin{equation} \label{eq: FW1}
\P(x,\bfs x) \approx \ve^{2(n+1)} A(0;x,\bfs x).
\end{equation}
On the other hand, by the restriction property of SLE(8/3), we have
\begin{align} \label{eq: FW2}
\P(&x,\bfs x) = \P(\bfs x)- \P(\bfs x\,|\,\neg x)(1-\P(x)) \\
&\approx \ve^{2n} (A(0;\bfs{x}) -\Psi'(0)^\lambda \Psi'(x_1)^2\cdots\Psi'(x_n)^2 A(\Psi(0);\Psi(x_1),\cdots, \Psi(x_n))), \nonumber
\end{align}
where
$$\Psi(z) = \sqrt{(z-x)^2+2\ve^2}+x.$$
It follows from \eqref{eq: FW1} and \eqref{eq: FW2} that
\begin{align*}
A(0;x,\bfs x) &= \frac{1}{\ve^2}(A(0;\bfs x)-\Psi'(0)^\lambda\Psi'(x_1)^2\cdots\Psi'(x_n)^2 A(\Psi(0);\Psi(x_1),\cdots, \Psi(x_n)))\\
&\approx \LL_v A(0;\bfs x),
\end{align*}
where $v$ is the vector field of flow
$\Psi_t(z) = \sqrt{(z-x)^2-2t}+x.$
Clearly, $v = v_x.$
\end{proof}

\ms \SS\textbf{Example 3 (Lawler-Schramm-Werner's restriction formula)} \index{restriction formula}
Restriction property of $\SLE(8/3)$ follows, by optional stopping theorem, from the fact that
for each compact hull $K,$
$$M_t = (\E\,V_\star^{ia}(\gamma_t;D_t\sm K)\,\|\,g_t)$$
is a local martingale.
This is a special case of Lawler-Schramm-Werner's formula

\begin{equation} \label{eq: LSW}
\textrm{the drift term of } dM_t= \frac c6 S_{h_t}(\xi_t) M_t\,dt,\quad (\xi_t = \sqrt\kappa B_t),
\end{equation}
which holds for all $\kappa\le4.$
On the event $\gamma[0,\infty)\cap K = \emptyset,$ a conformal map $h_t:\Omega_t = g_t(D_t\sm K)\to\H$ is defined by
$$h_t = \widetilde g_t \circ \Psi_K \circ g_t^{-1},$$
where $\widetilde g_t$ is a Loewner map from $\widetilde D_t = D\sm\widetilde\gamma[0,t]$ onto $\H,$
$\widetilde\gamma(t) = \Psi_K \circ \gamma(t),$ and
$\Psi_K$ is the conformal map $(\H\sm K,0,\infty)\to (\H,0,\infty)$ satisfying $\Psi_K'(\infty)=1.$
Let $\widetilde \xi_t = h_t(\xi_t).$
Then $h_t$ satisfies
\begin{equation} \label{eq: ht}
\dot h_t(z) = \frac{2h_t'(\xi_t)^2}{h_t(z)-\widetilde \xi_t} - \frac{2h_t'(z)}{z-\xi_t},
\end{equation}
see e.g., \cite{LSW03}.

\begin{figure}[ht]
\begin{center}
\includegraphics[width=\textwidth]{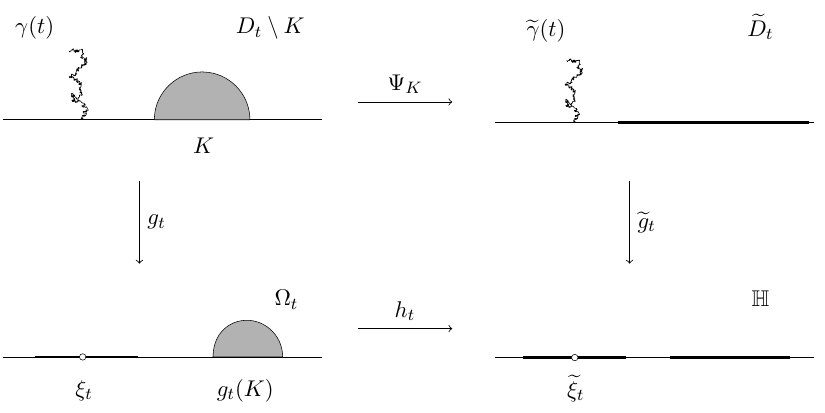}
\end{center}
\caption{The conformal maps, $\Psi_K,g_t,\widetilde g_t,$ and $h_t$}
\label{figure: restriction}
\end{figure}

\ms We now explain why the central charge and the Schwarzian derivative appear in the formula \eqref{eq: LSW}.
To prove this formula, denote
$$F(z,t):=(\E\,V_{\Omega_t}(z)\,\|\,\id).$$
Then $F(z,t)=(\E\,V_{\H}\,\|\,h_t^{-1})(z)$ and
$$M_t= F(\xi_t,t) = (\E\,V_{\Omega_t}(\xi_t)\,\|\,\id).$$
The function $F$ is smooth in both variables, so by It\^o's formula,
$$\textrm{the drift term of }dM_t = \dot F(\xi_t,t)\,dt + \frac\kappa2F''(\xi_t,t)\,dt.$$
For the first term of the right-hand side, we represent $\dot F$ in terms of the Lie derivatives: 
\begin{align*}
\dot F(z,t) &= \frac{d}{ds}\Big|_{s=0} (\E\, V_{\H}\,\|\, h_{t+s}^{-1}) (z)= \frac{d}{ds}\Big|_{s=0} (\E\, V_{\H}\,\|\, h_{t}^{-1}\circ f_{s,t}^{-1}) (z) \\
&= (\E\,\LL_v V_{\H}\,\|\,h_t^{-1})(z),
\end{align*}
where $f_{s,t} = h_{t+s}\circ h_{t}^{-1}$ and
$$(v\,\|\,\id_\H) = \frac{d}{ds}\Big|_{s=0} f_{s,t} = \dot h_t \circ h_t^{-1}.$$
It follows from \eqref{eq: ht} that
$$(v\,\|\,\id_\H)(\zeta)=-2h_t'(\xi_t)^2\frac1{\widetilde\xi_t-\zeta} + 2h_t'(h_t^{-1}(\zeta))\frac{1}{\xi_t-h_t^{-1}(\zeta)}.$$
By Proposition~\ref{Ward's in H} we have
$$\dot F(z,t) = -2h_t'(\xi_t)^2 h_t'(z)^\lambda(\E\, A_\H(\widetilde \xi_t)\,V_\H (h_t(z))\,\|\,\id_\H) +2 (\E\,\LL_{v_{\xi_t}} V_{\Omega_t}\,\|\,\id_{\Omega_t}\big)(z),$$
where $\lambda$ is the conformal dimension of $V.$
It follows from conformal invariance that
$$\dot F(z,t) =-2(\E\, A_{\Omega_t}(\xi_t)\,V_{\Omega_t}(z)\,\|\,\id) + 2 (\E\,\LL_{v_{\xi_t}} V_{\Omega_t}\,\|\,\id\big)(z).$$
Let us now apply Proposition~\ref{Ward's OPEs} to the right-hand side of the above equation:
$$\dot F(\xi_t,t)=\lim_{z\to\xi_t}\dot F(z,t)=-2(\E\,A_{\Omega_t}*V_{\Omega_t}(\xi_t)\,\|\,\id).$$
At the same time, we have
$$\frac\kappa2F''(\xi_t,t) = \frac\kappa2(\E\,\pa_\xi^2\Big|_{\xi=\xi_t}V_{\Omega_t}(\xi)\,\|\,\id) = 2(\E\,T_{\Omega_t}*V_{\Omega_t}(\xi_t)\,\|\,\id).$$

\renewcommand\chaptername{Lecture}
\chapter{Vertex observables}\label{ch: VO}

In this last lecture we will look at some examples of ``primary" SLE observables, i.e., observables that have a covariant dependence on local coordinates.
This is the type of dependence that appears as a result of rescaling or normalization of probabilities and expectations, in particular in lattice models.
As we explained in the previous lecture, correlators of primary fields in $\widehat\FF_{(b)}$-theories are examples of such observables.
In this lecture we will expand our collection of primary fields by considering normalized tensor products of chiral vertex fields and their conjugates.
Further ``primary" observables can be obtained from singular vectors and, in some cases, by such operations as differentiation, integration, and ``screening."
By It\^o's calculus, ``primary" SLE observables are solutions to 2nd order linear differential equations, which in general are not easy to solve.
The knowledge of a large collection of (multi-point) primary fields allows us, in some cases, to identify particular solutions by calculus of conformal dimensions.
This is somewhat related to ``Coulomb gas" methods in the physical literature.

\bs\section{Holomorphic 1-point vertex fields} \label{sec: 1-pt V}

\ms \subsec{Definition}
We want to construct holomorphic single variable differentials in $\widehat\FF_{(b)}$-theory.
Chiral vertex fields rooted at $q,$ the central charge modification point, considered under the boundary condition changing operation (see Section~\ref{sec: BC changing}) seem to be natural candidates.
The problem with this construction is the divergence at $q$ so, as in Section~\ref{sec: V*}, where we defined rooted vertex fields, we will use a certain normalization procedure.
(We will use it again to define correlators and, more generally, tensor product of such fields in the next section.)

\ms The idea is to start with a (well-defined) chiral {\it bi}-vertex field,
\begin{equation} \label{eq: hat bi-vertex}
\widehat V^{i\sigma}(z,z_0)=w^{\sigma a}(w')^{\sigma^2/2-\sigma b}\,w_0^{-\sigma a}(w'_0)^{\sigma^2/2+\sigma b} (w-w_0)^{-\sigma^2}\,e^{\odot i\sigma\Phichiral_{(0)}(z,z_0)},
\end{equation}
see the last example in Section~\ref{sec: BC changing} and normalize it so that the limit exists as $z_0\to q.$
We denote this limit by $\widehat V_\star^{i\sigma}(z).$
Then it satisfies the equations of $\widehat\FF_{(b)}$-theory, in particular Cardy's equations.
It turns out that by simply ignoring all terms in \eqref{eq: hat bi-vertex} involving the point $z_0,$ we arrive to the following correct definition: \index{vertex field!rooted $V_\star,\widehat V_\star$} \index{rooted vertex field $V_\star,\widehat V_\star$}
$$\OO^{(\sigma)}(z)\equiv\widehat V_\star^{i\sigma}(z)=w^{\sigma a}(w')^\lambda\,e^{\odot i\sigma\Phichiral_{(0)}(z,q)},$$
where $\lambda = \sigma^2/2-\sigma b.$
(Instead of $\widehat\OO^{(\sigma)}$ we write $\OO^{(\sigma)}$ for $\widehat V_\star^{i\sigma}$ in this lecture. See Appendix~\ref{appx: numerology}.)
More accurately, the expression for $\OO^{(\sigma)}$ can be described as a limit similar to the formula \eqref{eq: V*} in Section~\ref{sec: V*}.
The vertex fields $\OO^{(\sigma)}$ are invariant with respect to $\Aut(D,p,q).$

\begin{prop} \label{VM}
The 1-point function $M=\E\,\OO^{(\sigma)}$ is an SLE martingale-observable.
\end{prop}

\ss It is of course very easy to verify this statement by It\^o's calculus but the point is that we can get it as a limit of Cardy's equations for bi-vertex fields.
Indeed,
let $R_t(z,z_0)$ denote the 1-point function martingale of the bi-vertex field.
We have
$$R_t(z,z_0)=M_t(z) E_t(z,z_0),$$
where
$$E_t(z,z_0)=w_t(z_0)^\alpha(w_t'(z_0))^\beta(w_t(z)-w_t(z_0))^\gamma $$
with appropriate exponents.
Since $R_t$ is a local martingale, the drift of $M_t$ is equal to that of
$$-M_t\frac{dE_t}{E_t}-dM_t\cdot\frac{dE_t}{E_t}.$$
It is trivial to see that
$$ \frac{dE_t}{E_t}\to 0\qquad( z_0\to q),$$
for any combination of the exponents, and this proves the statement.
(We will refer to this argument again in the next section.)

\ms\subsec{Conformal dimensions}
The normalization procedure in the definition of $\OO^{(\sigma)}$ produces covariance with respect to $q.$

\begin{prop} \label{hol 1-pt MO}
The 1-point function $M=\E\,\OO^{(\sigma)}=w^{\sigma a}(w')^\lambda$ is an ${\Aut}(D,p,q)$-invariant holomorphic $\lambda$-differential with respect to $z$
and a boundary $\lambda_q$-differential with respect to $q,$ where
$$\lambda=\frac{\sigma^2}2-\sigma b,\qquad \lambda_q=\lambda+\sigma a.$$
Every martingale-observable with the stated properties is the correlation function of some vertex field $\OO^{(\sigma)}.$
\end{prop}

\begin{proof}
See Example in Section~\ref{sec: V*} for the value of $\lambda_q.$
The last part, first mentioned in Smirnov, see \cite{Smirnov06}, follows from It\^o's formula.
Let $M_t = f(w_t)(w_t')^\lambda.$
Driftlessness of $M_t$ reduces to Cauchy-Euler equation in $(\H,0,\infty).$
Thus $M(z) = C_1z^{\alpha_1} + C_2z^{\alpha_2}$ or $C_1z^{\alpha} + C_2z^{\alpha}\log z$ in the $(\H,0,\infty)$-uniformization.
We can take $C_2 = 0$ because $M$ is a boundary differential at $q.$
\end{proof}

\ms Note that

\ms (a) $\lambda=0$ if and only if $\sigma=2b$ (or $\sigma=0$);

\ms (b) $\lambda=1$ if and only if $\sigma=-2a$ or $\sigma=\dfrac1a=2(a+b)$;

\ms (c) $\lambda_q=0$ if and only if $\sigma=2(b-a)$ (or $\sigma=0$).

\ms \subsec{Special cases} Suppose $\sigma\ne0.$

\ms (a) If $\lambda = 0$ (i.e., $M$ is a scalar), then
$$M = w^\beta,\qquad\beta=\lambda_q=1-4/\kappa,\qquad (\kappa\ne4)$$
(a conformal map onto a wedge).
These ``wedge" observables are some of the simplest (and well-known) SLE martingale-observables.
For example, if $\kappa=2,$ then the observable
\begin{equation} \label{eq: LERW observable}
\Im\,M_{D,p,q} = \frac{P(z,p)}{N_qP(\cdot,p)}
\end{equation}
plays an important role in the theory of loop-erased random walks \index{loop-erased random walk (LERW)} (LERWs), see \cite{LSW04}.
Here $P$ is the Poisson kernel, and $N_qP(\cdot,p)$ is its normal derivative at $q.$

\ms \begin{rmk*}
For $\kappa>4,$ the wedge observables $M$ have the following probabilistic interpretation.
Let $\eta_\ve$ denote a point on $\pa D$ at distance $\ve$ from $q$ (in a local chart $\phi$).
Then
$$\lim_{\ve\to 0}\frac{\P(\tau_z=\tau_{\eta_\ve})}{\ve^\beta} = \const~\Im\, (M^{(2b)}(z)\,\|\,\phi),$$
where $\const$ is the normalization constant, and also
$$\lim_{\ve\to 0}\frac{\P(\tau_z>\tau_{\eta_\ve})}{\ve^\beta}$$
is a linear combination of real and imaginary parts of $(M^{(2b)}(z)\,\|\,\phi).$
These facts follow from Cardy's formula \eqref{eq: Cardy's formulae} which we discuss later in this lecture;
$\tau_z$ and $\tau_{\eta_\ve}$ are the SLE stopping times \eqref{eq: SLE stopping time}.
In particular, it is clear from this interpretation why wedge observables are scalars with respect to $z$ and have non-trivial conformal dimensions with respect to $q.$
\end{rmk*}

\ms (b) If $\lambda_q=0,$ then
$$M =(w'/w)^\lambda = (f')^\lambda,\qquad \lambda=8/\kappa-1,$$
where $f = \log w$ is a map onto the strip.
Smirnov used this observable with $\kappa=16/3$ in his study of the random cluster version of the Ising model, see \cite{Smirnov10}.

\ms \begin{rmk*}
If $\kappa\in (4,8),$ then the boundary values of this $M$ have the following interpretation (another special case of Cardy's formula, see Proposition~\ref{bdry MO}).
For $\eta\in\pa D\sm\{p,q\}$ let $\P(\eta,\ve)$ denote the probability that the SLE curve hits the boundary interval with endpoint $\eta$ of length $\ve,$ where $\ve\ll1$ is measured in a local chart $\phi.$
Then
$$\lim_{\ve\to 0}\frac {\P(\eta,\ve)}{\ve^\lambda} =\const\,(M^{(2b-2a)}(\eta)\,\|\,\phi),$$
where $\const$ is the normalization constant which depends on $(D, p,q)$ and $\kappa.$
Clearly, if the limit exists and is non-trivial, then it gives a boundary martingale-observable of conformal dimensions $\lambda$ at $\eta$ and zero at $q.$
On the other hand, $M^{(2b-2a)}$ is the only observable with these properties.
This argument certainly does not prove the existence of the limit but it provides a quick ``physical" answer.
(See the formula \eqref{eq: Beffara's observable} involving Beffara's observables for a similar statement at interior points.)
\end{rmk*}

\ms (c) If $\lambda +\lambda_q=0,$ then $\sigma=2b-a$ and
$$M = (w'/w^2)^\lambda = (f')^\lambda,\qquad \lambda=3/\kappa-1/2,$$
where $f=-1/w$ is a map onto the half-plane.
In the case $\kappa=2,$ this observable $M$ is the derivative of the LERW observable \eqref{eq: LERW observable}.
In the case $\kappa=3,$ $M$ plays a crucial role in Smirnov's work on spin Ising model, see \cite{CS09}.
In both cases, one can explain the relation $\lambda +\lambda_q=0$ from the point of view of discrete models -- it comes from the rescaling of the corresponding partition functions.
Smirnov suggested that the two series (b) and (c) of SLE observables describe the general random cluster and $O(N)$ models, respectively.

\bs\section{Normalized tensor products} \label{sec: star} \index{normalized tensor product} \index{tensor product!normalized}

\ms \subsec{Definition}
To define the product of $\OO^{(\sigma_1)}(z_1)$ and $\OO^{(\sigma_2)}(z_2),$ $z_1\ne z_2,$ we again need normalization because $\E[\Phichiral_{(0)}(z_1,q)\Phichiral_{(0)}(z_2,q)]$ diverges.
Applying the same idea as in Section~\ref{sec: 1-pt V} (normalizing the product of bi-vertex fields properly and then taking a limit), we define
\begin{equation} \label{eq: OstarO1}
\OO^{(\sigma_1)}(z_1)\star\OO^{(\sigma_2)}(z_2)=M_1M_2(w_1-w_2)^{\sigma_1\sigma_2}e^{\odot i\sigma_1\Phichiral_{(0)}(z_1)+ i \sigma_2\Phichiral_{(0)}(z_2)},
\end{equation}
where $M_j = \E[\OO^{(\sigma_j)}]=w_j^{\sigma_j a}(w_j')^{\sigma_j^2/2-\sigma_j b}$ and $w_j=w(z_j),w_j'=w'(z_j).$
Again, a simple way to express this definition is to say that we ignore all $q$-terms in the cross-ratio \eqref{eq: cross-ratio} when we compute the correlation function of Wick's exponentials.
The term $(w_1-w_2)^{\sigma_1\sigma_2}$ appears from the following computation:
if $\alpha_1,\alpha_2$ are two non-intersecting paths connecting from $q$ to $z_1,z_2,$ respectively, then ignoring the $q$-terms, we have
\begin{align*}
\E[\Phichiral_{(0)}(z_1){\Phichiral_{(0)}(z_2)}] &= \int_{\alpha_1}\int_{\alpha_2}\E[J_{(0)}(\zeta){J_{(0)}(\eta)}]\,d\zeta\,{d\eta}\\
&=\int_{\alpha_2} \frac{1}{\zeta-\eta}\Big|_{\zeta=q}^{\zeta=z_1}\,d\eta=-\log(z_1-z_2), \qquad (\textrm{in }\H).
\end{align*}

\ms Note the monodromy in the correlation function of the tensor product. There is no difficulty in extending the concept of martingale-observable to multivalued functions -- each continuous branch should be a martingale-observable.

\bs In a similar way, we can use $\E[\Phichiral_{(0)}(z_1)\overline{\Phichiral_{(0)}(z_2)}]=\log(z_1-\bar z_2)$ to define
\begin{equation} \label{eq: OstarO2}
\OO^{(\sigma_1)}(z_1)\star\overline{\OO^{(\sigma_2)}(z_2)}=M_1\overline{M}_2(w_1-\bar w_2)^{\sigma_1\bar \sigma_2}e^{\odot i\sigma_1\Phichiral_{(0)}(z_1)+\overline{i\sigma_2\Phichiral_{(0)}(z_2)}}.
\end{equation}

\bs\subsec{General 1-point vertex fields} This last definition~\eqref{eq: OstarO2} can be extended to the case $z=z_1=z_2,$ cf. \eqref{eq: J*bar J}, so if we denote
$$\OO^{(\sigma, \sigma_*)}(z)=\OO^{(\sigma)}(z)\star\overline{\OO^{(\bar\sigma_*)}(z)},$$
then we have
$$\OO^{(\sigma, \sigma_*)}(z)=(w-\bar w)^{\sigma\sigma_*}\,w^{\sigma a}\,\bar w^{\sigma_* a}\,(w')^\lambda\,\left(\overline{w'}\right)^{\lambda_*}e^{\odot i\sigma\Phichiral_{(0)}-i\sigma_*\overline{\Phichiral_{(0)}}},$$
where
$$
\lambda=\frac{\sigma^2}2-\sigma b,\qquad \lambda_*=\frac{\sigma_*^2}2-\sigma_* b
$$
(the conformal dimensions with respect to $z$).

\bs \subsec{Special cases} Up to constant factors, we have the following relations:

\ms (a) The non-rooted bi-vertex chiral field of $\widehat\FF_{(b)}$-theory,
$$\OO^{(\sigma)}(z_1)\star\OO^{(-\sigma)}(z_2)=\widehat V^{i\sigma}(z_1,z_2).$$
This can be shown by direct substitution $\sigma_1 = \sigma, \sigma_2 = -\sigma$ into \eqref{eq: OstarO1}.
We get \eqref{eq: hat bi-vertex} because $\Phichiral_{(0)}(z_1)-\Phichiral_{(0)}(z_2)= \Phichiral_{(0)}(z_1,z_2).$
Alternately, we can take the limit $z_0\to q$ in the identity
$$\widehat V^{i\sigma}(z_1,z_0)\widehat V^{-i\sigma}(z_2,z_0) = \widehat V^{i\sigma}(z_1,z_2);$$

\ms (b) The non-rooted and non-chiral vertex field of $\widehat\FF_{(b)}$-theory,
$$\OO^{(\sigma,-\sigma)}=\widehat \VV^{i\sigma},$$
see \eqref{eq: hat VV}.
This follows from the identity $\Phichiral_{(0)}+\overline{\Phichiral_{(0)}} = 2\,\Re\,\Phichiral_{(0)} = \Phi_{(0)};$

\ms (c) The dual vertex field of $\widehat\FF_{(b)}$-theory, \index{vertex field!dual $\OO^{(\sigma,\sigma)}$} \index{dual vertex field $\OO^{(\sigma,\sigma)}$}
$$\OO^{(\sigma,\sigma)} = (w-\bar w)^{\sigma^2}|w|^{2\sigma a}|w'|^{2\lambda} e^{\odot - \sigma\widetilde\Phi_{(0)}}.$$
(It is the properly normalized and modified exponential of the harmonic conjugate bosonic field $\widetilde\Phi,$
see \eqref{eq: dual boson}.)
Note that $\OO^{(\sigma,\sigma)}$ is a real field if $\sigma\in\R.$
(In general $\OO^{(\sigma,\sigma_*)}$ is real if and only if $\sigma_*=\bar\sigma.$)

\bs \subsec{Multipoint vertex fields} \index{vertex field!multipoint}
We can extend the definition to products of $n$ fields, e.g.,
$$\OO^{(\sigma_1)}(z_1)\star\cdots\star\OO^{(\sigma_n)}(z_n)=M_1\cdots M_n\,\prod_{j<k}(w_j-w_k)^{\sigma_j\sigma_k} e^{\odot i\sigma_1\Phichiral_{(0)}(z_1)+\cdots+ i\sigma_n\Phichiral_{(0)}(z_n)}.$$
In particular, we have
\begin{equation} \label{eq: n-pt O}
\E\,\OO^{(\sigma_1,\sigma_{*1})}(z_1)\star\cdots\star\OO^{(\sigma_n,\sigma_{*n})}(z_n)=
\prod_j \E\,\OO^{(\sigma_j,\sigma_{*j})}(z_j)
\prod_{j<k} L_{j,k}(z_j,z_k),
\end{equation}
where
$$L_{j,k}(z_j,z_k)=(w_j-w_k)^{\sigma_j\sigma_k}(\bar w_j-\bar w_k)^{\sigma_{*j}\sigma_{*k}}(w_j-\bar w_k)^{\sigma_j\sigma_{*k}}(\bar w_j-w_k)^{\sigma_{*j}\sigma_k}.$$
In general, the operation $\star$ is commutative (in the sense of multivalued functions), associative, and real
(i.e., $\overline{A\star B}=\bar A\star\bar B$).

\ms \subsec{Cardy's equations and SLE martingale-observables}
It is not difficult to show that Cardy's equations survive under the normalization procedure.
\begin{prop}
Normalized correlations $M$ of chiral vertex fields (and their complex conjugates) satisfy Cardy's equations:
$$\LL_{v_0}M = \frac1{2a^2}(\sum\pa_j+\bp_j)^2M.$$
\end{prop}
\begin{eg*}
Let $M = \E\,X,$ where $$X(z_1,\cdots,z_n)=\OO^{(\sigma_1,\sigma_{*1})}(z_1)\star\cdots\star\OO^{(\sigma_n,\sigma_{*n})}(z_n).$$
In $(\H,0,\infty),$ $M$ reads as
$$\prod_j z_j^{\sigma_ja}\bar z_j^{\sigma_{*j}a}(z_j-\bar z_j)^{\sigma_j\sigma_{*j}}\,\prod_{j< k}(z_j-z_k)^{\sigma_j\sigma_k}(\bar z_j-\bar z_k)^{\sigma_{*j}\sigma_{*k}}(z_j-\bar z_k)^{\sigma_j\sigma_{*k}}(\bar z_j-z_k)^{\sigma_{*j}\sigma_k}$$
and satisfies the 1st order linear PDE
$$\frac1{a}\sum_j(\pa_j+\bp_j)M = \sum_j \Big(\frac{\sigma_j}{z_j}+ \frac{\sigma_{*j}}{\bar z_j}\Big)M.$$
Cardy's equation for $M$ is the 2nd order PDE
\begin{align*}
\frac1{2a^2}(\sum\pa_j+\bp_j)^2M &= -(a+b)\sum\Big(\frac{\sigma_j}{z_j^2}+ \frac{\sigma_{*j}}{\bar z_j^2}\Big)M + \frac12\Big(\sum \frac{\sigma_j}{z_j}+ \frac{\sigma_{*j}}{\bar z_j}\Big)^2M\\
&= \sum\Big(-\frac{\pa_j}{z_j}+\frac{\lambda_j}{z_j^2}-\frac{\bp_j}{\bar z_j}+\frac{\lambda_{*j}}{\bar z_j^2}\Big)M = \LL_{v_0}M.
\end{align*}
\end{eg*}

\begin{prop}
Normalized correlations of chiral vertex fields (and their complex conjugates) are SLE observables.
\end{prop}
This follows from the argument in the proof of Proposition~\ref{VM}.

\ms \subsec{Conformal dimensions at $q$}
The next proposition gives a necessary and sufficient condition for a multipoint vertex field to be a 0-boundary differential at $q.$

\begin{prop} \label{dim calculus}
Let $X(z_1,\cdots,z_n)=\OO^{(\sigma_1,\sigma_{*1})}(z_1)\star\cdots\star\OO^{(\sigma_n,\sigma_{*n})}(z_n).$
Then $X$ is a boundary differential at $q$ of dimension
$$\lambda_q(X)=(a-b)\Sigma+\frac12\Sigma^2,\qquad\Sigma:=\sum_j(\sigma_j+\sigma_{*j}).$$
In particular,
$$\lambda_q(X)=0\qquad\textit{if and only if}\qquad \Sigma=0\quad {\rm or}\quad \Sigma=2b-2a.$$
\end{prop}

\begin{proof}
Let $M_j(z_j)=\E\,\OO^{(\sigma_j,\sigma_{*j})}(z_j)$ and $\Sigma_j = \sigma_j+\sigma_{*j}.$
Since $X/(\E\,X)$ is a 0-boundary differential at $q,$ it follows from \eqref{eq: n-pt O} that
\begin{align*}
\lambda_q(X)&= \sum_j \lambda_q(M_j) + \sum_{j<k} \lambda_q(L_{j,k})\\
&=\sum_j \Big((a-b)\Sigma_j+\frac12\Sigma_j^2\Big) + \sum_{j<k}\Sigma_j\Sigma_k= (a-b)\Sigma+\frac12\Sigma^2.
\end{align*}
\end{proof}

\section{1-point martingale-observables} \label{sec: 1-pt MOs}

\ms \subsec{Vertex observables}
Recall the expression for the 1-point observables that we get from vertex fields:
$$M^{(\sigma, \sigma_*)}\equiv \E\,\OO^{(\sigma, \sigma_*)}=(w-\bar w)^{\sigma\sigma_*}~w^{\sigma a}~\bar w^{\sigma_*a}~(w')^\lambda~\left(\overline{w'}\right)^{\lambda_*},$$
where
\begin{equation} \label{eq: dim O}
\lambda=\frac{\sigma^2}2-\sigma b,\qquad \lambda_*=\frac{\sigma_*^2}2-\sigma_* b.
\end{equation}
We can rewrite $M$ in terms of the conformal radius:$$M\, = \,i^{\sigma\sigma_*}\, C^{\sigma\sigma_*}\,|w'|^{\sigma\sigma_*}w^{\sigma a}\, \bar w^{\sigma_*a}\,(w')^\lambda\, \left(\overline{w'}\right)^{\lambda_*}.$$
Recall that $C=|w-\bar w|/|w'|$; and $(C\,\|\,\id_D)\asymp \dist(\cdot, \pa D).$
If we ignore the constant factor $i^{\sigma\sigma_*},$ then
$$\overline{\OO^{(\sigma, \sigma_*)}}\,=\,\OO^{(\bar \sigma_*, \bar \sigma)}.$$

\ms It follows that the 1-point observable $M^{(\sigma, \sigma_*)}$ has non-trivial boundary values on $\pa D\sm\{p,q\}$ only if either $\sigma$ or $\sigma_*$ is zero.

\ms \SS
Unlike the holomorphic case, see Proposition~\ref{hol 1-pt MO}, not all 1-point martingale-observables can be represented by vertex fields.
For example, Schramm's observable
$$M(z)=\P\{z~\textrm{is to the left of}~\gamma\}, \qquad(0<\kappa<8,\,z\in D)$$
is not of the form $M=M^{(\sigma, \sigma_*)}.$
(The winding number of the positively oriented closed curve $\gamma\,\cup\arc{qp}$ around $z$ is well-defined because $z$ is off the curve $\gamma$ almost surely.
This winding number is either 0 or 1 since $\gamma$ is almost surely non-self-crossing.
We say that $z$ is to the left of $\gamma$ if this winding number is 1.)
Indeed, Schramm's observable has dimensions $\lambda=\lambda_*=\lambda_q=0$ but the scalar field $I\,(I(z)\equiv1)$ is the only vertex field with such dimensions.

\ms In addition to vertex observables, we can consider expectations (and correlations) of singular vectors (so we get ``primary" observables).
The simplest is the case of level one singular vectors, which are (up to conjugation) just the $\pa_z$-derivative of scalar fields.
It is of course always true that derivatives of martingale-observables are martingale-observables.
Conversely, if a martingale-observable $X$ is a (1,0)-differential, then for any curve $\alpha,$ $\int_\alpha X$ is a local martingale.
We continue this discussion in the next section.

\bs \subsec{Example: Schramm's observables} \index{martingale-observables!Schramm's} \index{Schramm's observables}
Let $M(z)=\P\{z$ is to the left of $\gamma\}.$
Then $\pa M$ is a martingale-observable with
$$\lambda=1,\qquad \lambda_*=0,\qquad \lambda_q=0.$$
It is easy to find a vertex field with these conformal dimensions (applying \eqref{eq: dim O} and Proposition~\ref{dim calculus}), namely $\OO^{(-2a, 2b)},$ so its correlation function 
$$\E\,\OO^{(-2a, 2b)}=\left(\frac{w-\bar w}{|w|}\right)^{\frac8\kappa-2}\frac{w'}w$$
is a natural candidate for $\pa M$ (up to a constant).
(If $\kappa = 2,$ then
$$\E\,\OO^{(a,2b)} =\E\,\OO^{(1,-1)} = \frac{ww'}{\bar w(w-\bar w)}$$
is the other possible martingale-observable.
However, $M$ is the function of $\arg w$ only, so it cannot be a ``primitive" of $\E\,\OO^{(1,-1)}.$)

Let $N$ be a primitive of $\E\,\OO^{(-2a, 2b)},$ i.e., $\pa N=\E\,\OO^{(-2a, 2b)}.$
In $(\H,0,\infty),$ on a circle $z=re^{i\theta}$ we have
$$dN=(\sin\theta)^{\frac8\kappa-2}d\theta,$$
therefore
$$N(z)=\int_0^\theta(\sin t)^{\frac8\kappa-2}dt.$$
Taking into account the boundary values of $M,$ we get the presumptive formula
$$M(z) = \dfrac{~\displaystyle \int_{\underset{\,}{0}}^\theta (\sin t)^{\frac8\kappa-2}\,dt~} {~\displaystyle \int_0^{\overset{\,}{\pi}} (\sin t)^{\frac8\kappa-2}\,dt~},\qquad \theta=\arg w(z).$$
We only need to apply the optional stopping time (and some basic SLE properties) to justify it rigorously.

\bs \subsec{Example: Beffara's observables} \index{martingale-observables!Beffara's} \index{Beffara's observables}
Let $\kappa< 8$ and $\P(z,\ve)$ denote the probability that the SLE curve hits the disc at $z$ of size $\ve,$ where $\ve\ll1$ is measured in a local chart $\phi.$
We define
\begin{equation} \label{eq: Beffara's MO}
(M(z)\,\|\,\phi)=\lim_{\ve\to 0}\frac {\P(z,\ve)}{\ve^\delta},
\end{equation}
where $\delta$ is some positive number.
If the limit exists and is non-trivial (for all $z\in D$), then $M$ is a real martingale-observable of conformal dimensions
$$\lambda=\lambda_*=\delta/2, \qquad \lambda_q=0.$$

It is clear that the vertex field $\OO^{(b-a,b-a)}$ has those properties.
(The only other possibility $(\sigma,\sigma_*)= \pm(\sqrt\delta,-\sqrt\delta)$ with $\kappa=4$ gives wrong boundary values.)
We can conjecture that $M$ is $\E\,\OO^{(b-a,b-a)}$ up to a constant. Then
$$\delta=2\lambda= (b-a)^2-2b(b-a)=1-\frac{\kappa}{8}$$
and
$$M(z)=\const ~y^{\frac\kappa8+\frac8\kappa-2}\cdot|z|^{1-\frac8\kappa}\qquad ({\rm in}\;\H).$$
The justification of this conclusion (in particular, the existence of the limit \eqref{eq: Beffara's MO}) is not as easy, e.g., see \cite{LW10}.
Beffara proved the estimate
\begin{equation} \label{eq: Beffara's observable}
\P(z,\ve)\asymp \ve^{1-\frac{\kappa}{8}}~\E\,\OO^{(b-a,b-a)}
\end{equation}
and used it together with the second moment estimate to derive that the Hausdorff dimension of SLE curves is almost surely $1+\kappa/8,$ see \cite{Beffara08}.

\bs\section{Multi-point observables} \label{sec: n-pt MOs}
There are many natural, geometrically defined SLE multi-point observables, e.g., various multi-point generalizations of Schramm's and Beffara's observables, or the Friedrich-Werner observables in the case $\kappa\ne8/3.$
Can vertex fields (``Coulomb gas formalism") be useful in the (heuristic) identification of (at least some) of them?
The answer is not obvious but in any case it is clear that it would be useful to construct, in addition to vertex observables, as many ``primary" observables as possible.
For example, Friedrich-Werner's formula involves correlations of a singular vector.
One can consider singular vectors (of all levels) and one can also modify correlations of primary fields by ``screening," one of the basic operations in Coulomb gas formalism.
Let us start with a historically important:

\ms \subsec{Example: Cardy's observables} \index{martingale-observables!Cardy's} \index{Cardy's observables}
Let $\kappa>4,$ $z\in D,$ $\eta\in~\arc{qp}~\subset \pa D,$ where $\arc{qp}$ is the positively oriented arc from $q$ to $p.$
Three ``geometric" observables:
$$M(z,\eta)=\P(\tau_z<\tau_\eta), \quad\P(\tau_z=\tau_\eta), \quad\P(\tau_z>\tau_\eta), $$
are all real with all conformal dimensions zero.
No such vertex exists except the scalar field $I.$
We will argue as in the case of Schramm's observable in Section~\ref{sec: 1-pt MOs}.
We try to identify the derivative $\pa M$ with a multi-point vertex field which has conformal dimensions
\begin{equation} \label{eq: Cardy's dim}
\lambda=1, \quad \lambda_*=0,\quad \lambda_\eta=0,\quad \lambda_q=0.
\end{equation}
By dimension calculus, see \eqref{eq: dim O} and Proposition~\ref{dim calculus}, we immediately find that the vertex field $\OO^{(-2a)}(z)\star\OO^{(2b)}(\eta)$ satisfies \eqref{eq: Cardy's dim}.
Let $N$ be a martingale-observable (scalar in all variables) such that
$$\pa_z N(z,\eta)=\const\,\E\,\OO^{(-2a)}(z)\star\OO^{(2b)}(\eta) =\const\, w^{-\frac4\kappa}\,w_\eta^{1-\frac4\kappa}\,(w-w_\eta)^{\frac8\kappa-2}\,w',$$
where $w = w(z)$ and $w_\eta = w(\eta).$
We will take
$$N(z)=\frac{~\displaystyle \int_{\underset{\,}{p}}^z\E\,\OO^{(-2a)}(\zeta)\star\OO^{(2b)}(\eta) ~}{~\displaystyle \int_p^{\overset{\,}{q}}\E\,\OO^{(-2a)}(\zeta)\star\OO^{(2b)}(\eta) ~},$$
where $\zeta$ is the integration variable.
The integral converges because $\kappa>4.$
By Schwarz-Christoffel, $N$ is a conformal map onto the triangle with angles $\left(1-4/\kappa\right)\pi$ at $N(p)=0$ and $N(q)=1,$ and $\left(8/\kappa-1\right)\pi$ at $N(\eta)$ with $\Re\,N(\eta) = 1/2.$
Using the stopping time $\tau=\tau_z\wedge\tau_\eta$ and the fact (basic complex analysis) that
$$N_\tau = 0 ~\textrm{ if }~\tau_z<\tau_\eta,\quad N_\tau=1 ~\textrm{ if }~\tau_z>\tau_\eta,\quad N_\tau=\frac12+i\,\Im\,N(\eta) ~\textrm{ if }~ \tau_z=\tau_\eta,$$
we easily justify Cardy's formulae
\begin{equation} \label{eq: Cardy's formulae}
\P(\tau_z>\tau_\eta)=\Re\,N(z)-\frac{\Im\,N(z)}{2\,\Im\,N(\eta)}, \qquad\P(\tau_z=\tau_\eta)=\frac{\Im\, N(z)}{\Im\,N(\eta)}.
\end{equation}

\ms\SS \emph{Cardy's formulae} in the boundary case ($z = \eta_0\in \pa D$) describe the probability that the SLE curves hit some boundary intervals.

\ms \begin{prop}
If $\eta_0\in~\arc{\eta p},$ then
$$\P(\gamma~\textrm{hits}~\arc{\eta\eta}_0) =\frac{~\displaystyle \int_{\underset{\,}{\eta_{_0}}}^{\eta} \E\,\OO^{(-2a)}(\zeta)\star\OO^{(2b)}(\eta)~} {~\displaystyle \int_p^{\overset{\,}{\eta}} \E\,\OO^{(-2a)}(\zeta)\star\OO^{(2b)}(\eta)~},$$
where $\zeta$ is the integration variable.
\end{prop}

\begin{proof}
The event $\{\tau_{\eta_0} > \tau_\eta\}$ does not occur.
Thus we get
$$\P(\gamma~\textrm{hits}~\arc{\eta\eta}_0) = \P(\tau_{\eta_0} < \tau_\eta) = 1- \P(\tau_{\eta_0} = \tau_\eta) = 1-{N(\eta_0)}/{N(\eta)}.$$
\end{proof}

The denominator in the right-hand side is a constant by conformal invariance.
In the $(\H,0,\infty)$-uniformization, we have
\begin{equation} \label{eq: SLE hits I}
\P(\gamma~\textrm{hits}~[x-\ve,x])
= \frac1{C_\kappa} \int_0^{\frac{\ve}{x}} t^{\frac8\kappa-2} (1-t)^{-\frac4\kappa} \,dt,
\end{equation}
where $C_\kappa =\displaystyle \int_0^1 t^{\frac8\kappa-2} (1-t)^{-\frac4\kappa} \,dt,$
cf. \cite{RS05}.

\begin{prop}
Let $\eta_0\in~\arc{q\eta}.$
Denote by $\P(\ve;\eta_0,\eta)$ the probability that $\gamma$ hits a (boundary) $\ve$-neighborhood of $\eta_0$ but not the arc $\arc{\eta_0\eta},$ then (up to the same constant as in the previous proposition)
$$\lim_{\ve\to0}\frac{\P(\ve;\eta_0,\eta)}\ve = {\E\,\OO^{(-2a)}(\eta_0)\star\OO^{(2b)}(\eta)}.$$
\end{prop}

Indeed, the limit is the probability density function of $\P(\tau_{\eta_0} = \tau_\eta).$

\begin{prop} \label{bdry MO}
Let $\kappa\in(4,8).$
For $\eta\in\pa D\,(\eta\ne p,q)$ let $I(\eta,\ve)$ denote the boundary interval with endpoint $\eta$ of length $\ve,$ where $\ve\ll1$ is measured in a local chart $\phi.$
Then
$$\lim_{\ve\to 0}\frac1{\ve^{\frac8\kappa-1}} \P(\gamma~\textrm{hits}~I(\eta,\ve))=\const\,(M^{(2b-2a)}(\eta)\,\|\,\phi),$$
where $\const$ is the normalization constant which depends on $(D, p,q)$ and $\kappa.$
\end{prop}

\ss\begin{proof}
It suffices to show this in the $(\H,0,\infty)$-uniformization.
Then the statement follows from \eqref{eq: SLE hits I}.
\end{proof}

\ms \subsec{Screening}
This is a general method of constructing primary observables which we already used in the case of Schramm's and Cardy's observables.
The simplest situation is as follows.
Suppose we want to find 2-point boundary $\SLE(\kappa)$ martingale-observables $M(\eta_1, \eta_2)$ with given dimensions let's say $\lambda, \lambda,$ and 0, at $\eta_1,\eta_2\in\pa D,$ and $q,$  respectively.
It is not difficult to write down a differential equation in $(\H,0,\infty):$
\begin{equation} \label{eq: PDE4M}
\frac\kappa4(\pa_{\eta_1}+\pa_{\eta_2})^2M+ \Big(\frac{\pa_{\eta_1}}{\eta_1} + \frac{\pa_{\eta_2}}{\eta_2}\Big)M-\lambda\Big(\frac1{\eta_1^2}+\frac1{\eta_2^2}\Big)M=0.
\end{equation}
We will try to guess some solutions.

\ms Vertex observables give us solutions only for some special values of $\lambda,$ e.g.,
$$\E\,\OO^{(b-a)}(\eta_1)\star\OO^{(b-a)}(\eta_2)\qquad \textrm{for }\lambda=1/2-\kappa/{16}.$$
(If $\kappa =4,$ then we actually have solutions
$$\E\,\OO^{(\sqrt{2\lambda})}(\eta_1)\star\OO^{(-\sqrt{2\lambda})}(\eta_2).)$$
We can also consider singular vectors (like $T$ for $\kappa=8/3,\lambda=2$ as in Friedrich-Werner's formula) but, again, they only work in some special cases.

\ms We will now discuss a different type of 2-point boundary martingale-observables $M(\eta_1, \eta_2).$
Consider the field
$$N(\eta_1,\eta_2;\zeta) =\E \,\OO^{(\sigma_1)}(\eta_1)\star\OO^{(\sigma_2)}(\eta_2)\star\OO^{(s)}(\zeta),$$
where we choose
$$s=-2a\quad{\rm or}\quad s=\frac1a=2a+2b$$
(so that the ``screening" field $\OO^{(s)}$ is a 1-differential).
In the $(\H,0,\infty)$-uniformization, we have
\begin{equation} \label{eq: N}
N(\eta_1,\eta_2;\zeta)=\eta_1^{\sigma_1 a}\eta_2^{\sigma_2 a}\zeta^{sa}(\eta_1-\eta_2)^{\sigma_1\sigma_2}(\zeta-\eta_1)^{\sigma_1s}(\zeta-\eta_2)^{\sigma_2s}.
\end{equation}
By proposition~\ref{dim calculus}, $\lambda_q=0$ if and only if $$\sigma_1+\sigma_2+s=0\quad{\rm or}\quad 2b-2a.$$
The idea is to integrate out the variable $\zeta.$
The following lemma is obvious.

\begin{lem}
For any path $\alpha$ in $\pa D,$ the field
\begin{equation} \label{eq: screening}
M(\eta_1,\eta_2)=\int_{\alpha} N(\eta_1,\eta_2;\zeta)\qquad\text{(integration with respect to $\zeta$)}
\end{equation}
is a martingale-observable with conformal dimensions $\lambda_j ={\sigma_j^2}/2-\sigma_j b,\lambda_q=0.$
\end{lem}

\ms We are free to choose the path $\alpha$ as we wish but since we do not want to have any additional marked points, the natural choices for $\alpha$ are the following arcs:
$$\arc{\eta_1q},\qquad \arc{\eta_2 q}, \qquad \arc{\eta_1\eta_2}$$
(or integer combinations of those; in some other cases, we can also consider small half-loops around $\eta_j$ or $q$).
However, there are integrability restrictions: e.g., to use $\alpha=\,\arc{\eta_1q}$ we need
$$s\sigma_1>-1\qquad\textrm{and}\qquad s(a+\sigma_1+\sigma_2)<-1,$$
see \eqref{eq: N}.

\ms \begin{eg*}
Suppose $\lambda\ge-b^2/2.$
If we choose 
\begin{equation} \label{eq: sigma12}
s=-2a,\qquad\sigma_1 = b-\sqrt{b^2+2\lambda},\qquad \sigma_2 = b+\sqrt{b^2+2\lambda},
\end{equation}
then we have
$$\lambda_1=\lambda_2=\lambda, \qquad \lambda_q=0.$$
The integrability condition is always satisfied at $\eta_1$ and it is satisfied at $q$ if and only if $\kappa >4.$
It follows that for $\kappa>4$ and $\lambda\ge -(\kappa-4)^2/(16\kappa)$ we have solutions \eqref{eq: screening} with $\alpha=\,\arc{\eta_1q}.$
On the other hand, the integrability condition is satisfied at $\eta_2$ if $\sigma_2 < a+b,$ so for all $\kappa >0$ and $\lambda$ such that
$$-\frac{(\kappa-4)^2}{16\kappa}\le\lambda< \frac12\Big(1-\frac\kappa8\Big),$$
we have solutions \eqref{eq: screening} with $\alpha=\arc{\,\eta_1\eta_2}.$
These solutions generalize Cardy's observables mentioned earlier, which correspond to the case $\lambda=0.$
For example, one of these solutions is LSW/KPZ (Lawler-Schramm-Werner/Knizhnik-Polyakov-Zamolodchikov) martingale-observable. 

\ms For $\eta_1,\eta_2\in\R_+$ with $\eta_2<\eta_1$ and $\lambda > 0,$ let
$$M(\eta_1,\eta_2):= \lim_{t\to\tau}\E \Big(\frac{w_t'(\eta_1)w_t'(\eta_2)}{(w_t(\eta_1)-w_t(\eta_2))^2}\Big)^\lambda,$$
where $\tau =\tau_{\eta_1} \wedge \tau_{\eta_2}.$
We identify $M$ with a martingale-observable obtained by a method of screening.
To see this, let
$$u(\eta_1,\eta_2):=(\eta_1-\eta_2)^{2\lambda}M(\eta_1,\eta_2).$$
The non-negative bounded process
$$\Big(\frac{(\eta_1-\eta_2)^{2}w_t'(\eta_1)w_t'(\eta_2)}{(w_t(\eta_1)-w_t(\eta_2))^2}\Big)^\lambda =\exp\Big(-2\lambda\int_0^t \big(\frac1{w_s(\eta_1)}-\frac1{w_s(\eta_2)}\big)^2\,ds\Big)$$
is decreasing and has the limit as $t\to\tau.$
The existence of $M$ follows from Lebesgue's dominated convergence theorem.  
It is obvious that boundary dimensions of $M$ at $\eta_1,\eta_2,$ and $q,$ are $\lambda,\lambda,$ and $0$ respectively. 
It can be shown that $M_t$ is a martingale. 
Thus $M$ is a solution to \eqref{eq: PDE4M} and therefore $u$ satisfies
\begin{equation} \label{eq: PDE4u}
\frac\kappa4(\pa_{\eta_1}+\pa_{\eta_2})^2u+ \Big(\frac{\pa_{\eta_1}}{\eta_1} + \frac{\pa_{\eta_2}}{\eta_2}\Big)u-\lambda\Big(\frac1{\eta_1}-\frac1{\eta_2}\Big)^2u=0.
\end{equation}
On the other hand, it follows from the scaling property of chordal SLE that
$$u(\eta_1,\eta_2) = u(1,\eta_2/\eta_1).$$
Define $f$ on $(0,1)$ by $f(x) = u(1,x).$
It follows from \eqref{eq: PDE4u} that $f$ satisfies the second order ODE,
$$f''(x) = \frac2\kappa\,\frac{(-2+\kappa)x-2}{x(1-x)}\,f'(x)+\frac4\kappa\,\frac\lambda{x^2}\,f(x).$$
A general solution of this equation is
$$f(x) = C_+ x^{q_+} F(1-\frac4\kappa, 2q_+, \frac4\kappa + 2q_+,x) + C_- x^{q_-} F(1-\frac4\kappa, 2q_-, \frac4\kappa + 2q_-,x),$$
where exponents $q_\pm$ are given by 
$$q_\pm = \frac12-\frac2\kappa \pm \sqrt{(\frac12-\frac2\kappa)^2+\frac{4\lambda}\kappa}.$$
Since $f$ is bounded on $(0,1),$ $C_-=0.$
The other constant $C_+$ is determined by the condition $f(1-) = 1.$
Thus 
$$M(\eta_1,\eta_2) = (\eta_1-\eta_2)^{-2\lambda} \Big(\frac{\eta_2}{\eta_1}\Big)^{q} {F(1-\dfrac4\kappa, 2q, \dfrac4\kappa + 2q,\frac{\eta_2}{\eta_1})}/{F(1-\dfrac4\kappa, 2q, \dfrac4\kappa + 2q,1)},$$
where $q = q_+$ is the LSW/KPZ exponent (e.g., see Section~6.9 in \cite{Lawler05}).
With the choice of $s,\sigma_1,\sigma_2$ in \eqref{eq: sigma12}, and $\alpha = (\eta_1,\infty),$ $M$ has the representation \eqref{eq: screening} in the identity chart of $\H$ (up to constant) for $\kappa > 4$ (the integrability condition at $q$).
\end{eg*}

\ms \begin{rmk*} This construction can be extended to the case of several screening fields; it can be applied to singular vectors, etc. 
\end{rmk*}

\backmatter

\input{KM1.bbl}


\input{KM1.ind}

\end{document}

%% file: KM1.bbl
\def\cprime{$'$}
\providecommand{\bysame}{\leavevmode\hbox to3em{\hrulefill}\thinspace}
\providecommand{\MR}{\relax\ifhmode\unskip\space\fi MR }
\providecommand{\MRhref}[2]{%
  \href{http://www.ams.org/mathscinet-getitem?mr=#1}{#2}
}
\providecommand{\href}[2]{#2}

%% file: KM1.bbl
\begin{thebibliography}{RBGW07}

\bibitem[AHM11]{AHM11}
Yacin Ameur, H{\aa}kan Hedenmalm, and Nikolai Makarov, \emph{Fluctuations of
  eigenvalues of random normal matrices}, Duke Math. J. \textbf{159} (2011),
  no.~1, 31--81, \arXiv{0807.0375}. \MR{2817648}

\bibitem[BB02]{BB02}
Michel Bauer and Denis Bernard, \emph{{${\rm SLE}_\kappa$} growth processes and
  conformal field theories}, Phys. Lett. B \textbf{543} (2002), no.~1-2,
  135--138, \arXiv{math-ph/0206028}. \MR{1927368 (2003g:81216)}

\bibitem[BB03]{BB03}
\bysame, \emph{Conformal field theories of stochastic {L}oewner evolutions},
  Comm. Math. Phys. \textbf{239} (2003), no.~3, 493--521,
  \arXiv{hep-th/0210015}. \MR{2000927 (2004h:81216)}

\bibitem[Bef08]{Beffara08}
Vincent Beffara, \emph{The dimension of the {SLE} curves}, Ann. Probab.
  \textbf{36} (2008), no.~4, 1421--1452, \arXiv{math.PR/0211322}. \MR{2435854
  (2009e:60026)}

\bibitem[BPZ84]{BPZ84}
A.~A. Belavin, A.~M. Polyakov, and A.~B. Zamolodchikov, \emph{Infinite
  conformal symmetry in two-dimensional quantum field theory}, Nuclear Phys. B
  \textbf{241} (1984), no.~2, 333--380. \MR{757857 (86m:81097)}

\bibitem[Car84]{Cardy84}
John~L. Cardy, \emph{Conformal invariance and surface critical behavior},
  Nuclear Phys. B \textbf{240} (1984), 514--532.

\bibitem[Car92]{Cardy92}
\bysame, \emph{Critical percolation in finite geometries}, J. Phys. A
  \textbf{25} (1992), no.~4, L201--L206, \arXiv{hep-th/9111026}. \MR{1151081
  (92m:82048)}

\bibitem[Car05]{Cardy05}
\bysame, \emph{S{LE} for theoretical physicists}, Ann. Physics \textbf{318}
  (2005), no.~1, 81--118, \arXiv{cond-mat/0503313}. \MR{2148644}

\bibitem[Col85]{Colombeau85}
Jean-Fran{\c{c}}ois Colombeau, \emph{Elementary introduction to new generalized
  functions}, North-Holland Mathematics Studies, vol. 113, North-Holland
  Publishing Co., Amsterdam, 1985, Notes on Pure Mathematics, 103. \MR{808961
  (87f:46064)}

\bibitem[CS12]{CS09}
Dmitry Chelkak and Stanislav Smirnov, \emph{Universality in the 2{D} {I}sing
  model and conformal invariance of fermionic observables}, Invent. Math.
  \textbf{189} (2012), no.~3, 515--580, \arXiv{0910.2045}. \MR{2957303}

\bibitem[DFMS97]{DFMS97}
Philippe Di~Francesco, Pierre Mathieu, and David S{\'e}n{\'e}chal,
  \emph{Conformal field theory}, Graduate Texts in Contemporary Physics,
  Springer-Verlag, New York, 1997. \MR{1424041 (97g:81062)}

\bibitem[DS11]{DS11}
Bertrand Duplantier and Scott Sheffield, \emph{Liouville quantum gravity and
  {KPZ}}, Invent. Math. \textbf{185} (2011), no.~2, 333--393,
  \arXiv{0808.1560}. \MR{2819163 (2012f:81251)}

\bibitem[Dub09a]{Dubedat09b}
Julien Dub{\'e}dat, \emph{Duality of {S}chramm-{L}oewner evolutions}, Ann. Sci.
  \'Ec. Norm. Sup\'er. (4) \textbf{42} (2009), no.~5, 697--724,
  \arXiv{0711.1884}. \MR{2571956}

\bibitem[Dub09b]{Dubedat09}
\bysame, \emph{S{LE} and the free field: partition functions and couplings}, J.
  Amer. Math. Soc. \textbf{22} (2009), no.~4, 995--1054, \arXiv{0712.3018}.
  \MR{2525778}

\bibitem[Dup00]{Duplantier00}
Bertrand Duplantier, \emph{Conformally invariant fractals and potential
  theory}, Phys. Rev. Lett. \textbf{84} (2000), no.~7, 1363--1367,
  \arXiv{cond-mat/9908314}. \MR{1740371 (2001c:82040)}

\bibitem[FLM88]{FLM88}
Igor Frenkel, James Lepowsky, and Arne Meurman, \emph{Vertex operator algebras
  and the {M}onster}, Pure and Applied Mathematics, vol. 134, Academic Press
  Inc., Boston, MA, 1988. \MR{996026 (90h:17026)}

\bibitem[FW02]{FW02}
Roland Friedrich and Wendelin Werner, \emph{Conformal fields, restriction
  properties, degenerate representations and {SLE}}, C. R. Math. Acad. Sci.
  Paris \textbf{335} (2002), no.~11, 947--952, \arXiv{math/0209382}.
  \MR{1952555 (2003k:81207)}

\bibitem[FW03]{FW03}
\bysame, \emph{Conformal restriction, highest-weight representations and
  {SLE}}, Comm. Math. Phys. \textbf{243} (2003), no.~1, 105--122,
  \arXiv{math-ph/0301018}. \MR{2020222 (2005b:81172)}

\bibitem[GF68]{GF68}
I.~M. Gel{\cprime}fand and D.~B. Fuks, \emph{Cohomologies of the {L}ie algebra
  of vector fields on the circle}, Funkcional. Anal. i Prilo\v zen. \textbf{2}
  (1968), no.~4, 92--93. \MR{0245035 (39 \#6348a)}

\bibitem[IK11]{IK11}
Kenji Iohara and Yoshiyuki Koga, \emph{Representation theory of the {V}irasoro
  algebra}, Springer Monographs in Mathematics, Springer-Verlag London Ltd.,
  London, 2011. \MR{2744610 (2011m:17058)}

\bibitem[Jan97]{Janson97}
Svante Janson, \emph{{G}aussian {H}ilbert spaces}, Cambridge Tracts in
  Mathematics, vol. 129, Cambridge University Press, Cambridge, 1997.
  \MR{1474726 (99f:60082)}

\bibitem[Kac98]{Kac98}
Victor Kac, \emph{Vertex algebras for beginners}, second ed., University
  Lecture Series, vol.~10, American Mathematical Society, Providence, RI, 1998.
  \MR{1651389 (99f:17033)}

\bibitem[Kan07]{Kang07a}
Nam-Gyu Kang, \emph{Boundary behavior of {SLE}}, J. Amer. Math. Soc.
  \textbf{20} (2007), no.~1, 185--210 (electronic). \MR{2257400 (2008c:60095)}

\bibitem[KR87]{KR87}
V.~G. Kac and A.~K. Raina, \emph{Bombay lectures on highest weight
  representations of infinite-dimensional {L}ie algebras}, Advanced Series in
  Mathematical Physics, vol.~2, World Scientific Publishing Co. Inc., Teaneck,
  NJ, 1987. \MR{1021978 (90k:17013)}

\bibitem[Law05]{Lawler05}
Gregory~F. Lawler, \emph{Conformally invariant processes in the plane},
  Mathematical Surveys and Monographs, vol. 114, American Mathematical Society,
  Providence, RI, 2005. \MR{2129588 (2006i:60003)}

\bibitem[Law09]{Lawler09}
\bysame, \emph{Schramm-{L}oewner evolution ({SLE})}, Statistical mechanics,
  IAS/Park City Math. Ser., vol.~16, Amer. Math. Soc., Providence, RI, 2009,
  pp.~231--295. \MR{2523461}

\bibitem[LSW01a]{LSW01b}
Gregory~F. Lawler, Oded Schramm, and Wendelin Werner, \emph{Values of
  {B}rownian intersection exponents. {I}. {H}alf-plane exponents}, Acta Math.
  \textbf{187} (2001), no.~2, 237--273. \MR{1879850 (2002m:60159a)}

\bibitem[LSW01b]{LSW01c}
\bysame, \emph{Values of {B}rownian intersection exponents. {II}. {P}lane
  exponents}, Acta Math. \textbf{187} (2001), no.~2, 275--308. \MR{1879851
  (2002m:60159b)}

\bibitem[LSW02a]{LSW02b}
\bysame, \emph{Analyticity of intersection exponents for planar {B}rownian
  motion}, Acta Math. \textbf{189} (2002), no.~2, 179--201. \MR{1961197
  (2003m:60231)}

\bibitem[LSW02b]{LSW02a}
\bysame, \emph{Values of {B}rownian intersection exponents. {III}. {T}wo-sided
  exponents}, Ann. Inst. H. Poincar\'e Probab. Statist. \textbf{38} (2002),
  no.~1, 109--123. \MR{1899232 (2003d:60163)}

\bibitem[LSW03]{LSW03}
\bysame, \emph{Conformal restriction: the chordal case}, J. Amer. Math. Soc.
  \textbf{16} (2003), no.~4, 917--955 (electronic), \arXiv{math/0209343}.
  \MR{1992830 (2004g:60130)}

\bibitem[LSW04]{LSW04}
\bysame, \emph{Conformal invariance of planar loop-erased random walks and
  uniform spanning trees}, Ann. Probab. \textbf{32} (2004), no.~1B, 939--995,
  \arXiv{math/0112234}. \MR{2044671 (2005f:82043)}

\bibitem[LW10]{LW10}
Gregory~F. Lawler and Brent~M. Werness, \emph{Multi-point {G}reen's functions
  for {SLE} and an estimate of {B}effara}, Preprint, \arXiv{1011.3551}.

\bibitem[MS]{Sheffield05}
Jason Miller and Scott Sheffield, \emph{Imaginary geometry {I}: Interacting
  {SLE}s}, Preprint, \arXiv{1201.1496}.

\bibitem[RBGW07]{RBGW07}
I.~Rushkin, E.~Bettelheim, I.~A. Gruzberg, and P.~Wiegmann, \emph{Critical
  curves in conformally invariant statistical systems}, J. Phys. A \textbf{40}
  (2007), no.~9, 2165--2195, \arXiv{cond-mat/0610550}. \MR{2316324
  (2008m:82035)}

\bibitem[RS05]{RS05}
Steffen Rohde and Oded Schramm, \emph{Basic properties of {SLE}}, Ann. of Math.
  (2) \textbf{161} (2005), no.~2, 883--924, \arXiv{math/0106036}. \MR{2153402
  (2006f:60093)}

\bibitem[RY99]{RY99}
Daniel Revuz and Marc Yor, \emph{Continuous martingales and {B}rownian motion},
  third ed., Grundlehren der Mathematischen Wissenschaften [Fundamental
  Principles of Mathematical Sciences], vol. 293, Springer-Verlag, Berlin,
  1999. \MR{1725357 (2000h:60050)}

\bibitem[Sch00]{Schramm00}
Oded Schramm, \emph{Scaling limits of loop-erased random walks and uniform
  spanning trees}, Israel J. Math. \textbf{118} (2000), 221--288. \MR{1776084
  (2001m:60227)}

\bibitem[Sim74]{Simon74}
Barry Simon, \emph{The {$P(\phi )\sb{2}$} {E}uclidean (quantum) field theory},
  Princeton University Press, Princeton, N.J., 1974, Princeton Series in
  Physics. \MR{0489552 (58 \#8968)}

\bibitem[Smi01]{Smirnov01}
Stanislav Smirnov, \emph{Critical percolation in the plane: conformal
  invariance, {C}ardy's formula, scaling limits}, C. R. Acad. Sci. Paris S\'er.
  I Math. \textbf{333} (2001), no.~3, 239--244, \arXiv{0909.4499}. \MR{1851632
  (2002f:60193)}

\bibitem[Smi06]{Smirnov06}
\bysame, \emph{Towards conformal invariance of 2{D} lattice models},
  International {C}ongress of {M}athematicians. {V}ol. {II}, Eur. Math. Soc.,
  Z\"urich, 2006, pp.~1421--1451. \MR{2275653 (2008g:82026)}

\bibitem[Smi10]{Smirnov10}
\bysame, \emph{Conformal invariance in random cluster models. {I}.
  {H}olomorphic fermions in the {I}sing model}, Ann. of Math. (2) \textbf{172}
  (2010), no.~2, 1435--1467, \arXiv{0708.0039}. \MR{2680496}

\bibitem[SS10]{SS10}
Oded Schramm and Scott Sheffield, \emph{A contour line of the continuum
  {G}aussian free field}, Preprint, \arXiv{1008.2447}.

\bibitem[Wer04]{Werner04}
Wendelin Werner, \emph{Random planar curves and {S}chramm-{L}oewner
  evolutions}, Lectures on probability theory and statistics, Lecture Notes in
  Math., vol. 1840, Springer, Berlin, 2004, pp.~107--195. \MR{2079672
  (2005m:60020)}

\bibitem[Wer09]{Werner09}
\bysame, \emph{Lectures on two-dimensional critical percolation}, Statistical
  mechanics, IAS/Park City Math. Ser., vol.~16, Amer. Math. Soc., Providence,
  RI, 2009, pp.~297--360. \MR{2523462}

\bibitem[Zha08]{Zhan08b}
Dapeng Zhan, \emph{Duality of chordal {SLE}}, Invent. Math. \textbf{174}
  (2008), no.~2, 309--353, \arXiv{0712.0332}. \MR{2439609 (2010f:60239)}

\end{thebibliography}
